\documentclass[12pt]{article}
\usepackage{amsmath,amssymb,latexsym, amsfonts, amscd, amsthm}
\usepackage{graphicx,epsfig}

\theoremstyle{plain}
\newtheorem{theorem}{Theorem}[section]

\newtheorem{proposition}[theorem]{Proposition}
\newtheorem{corollary}[theorem]{Corollary}

\newtheorem{assumption}[theorem]{Assumption}
\newtheorem{lemma}[theorem]{Lemma}

\theoremstyle{definition}
\newtheorem{definition}[theorem]{Definition}
\newtheorem{remark}[theorem]{Remark}

\newcommand{\lra}{\longrightarrow}

\newcommand{\Spf}{\mbox{Spf}}
\newcommand{\Spm}{\mbox{Spm}}
\newcommand{\End}{\mbox{End}}
\newcommand{\Ker}{\mbox{Ker}}
\newcommand{\Fil}{\mbox{\rm Fil}}

\newcommand{\GL}{{\rm \mathbf{GL}}}

\newcommand{\Z}{{\mathbb Z}}
\newcommand{\Q}{{\mathbb Q}}
\newcommand{\C}{{\mathbb C}}

\newcommand{\F}{{\mathbb F}}
\newcommand{\N}{{\mathbb N}}
\newcommand{\V}{{\mathbb V}}

\newcommand{\A}{{\underline{A}}}

\newcommand{\Hsharp}{{\rm H}_E^\sharp}
\newcommand{\barHsharp}{\overline{\rm H}_E^\sharp}

\newcommand{\cI}{{\cal I}}
\newcommand{\cJ}{{\cal J}}
\newcommand{\cU}{{\cal U}}
\newcommand{\cP}{{\cal P}}

\newcommand{\cF}{{\cal F}}

\newcommand{\cE}{{\cal E}}

\newcommand{\bG}{{\mathbb G}}

\newcommand{\bV}{{\mathbb V}}

\newcommand{\bW}{{\mathbb W}}
\newcommand{\fcw}{{\mathfrak w}}
\newcommand{\fw}{{\mathfrak w}}
\newcommand{\bA}{{\mathbb A}}
\newcommand{\bB}{{\mathbb B}}
\newcommand{\bof}{\bf{f}}
\newcommand{\bg}{\bf{g}}
\newcommand{\bh}{\bf{h}}

\newcommand{\cO}{{\cal O}}

\newcommand{\Xf}{{\mathfrak X}}
\newcommand{\Uf}{{\mathfrak U}}
\newcommand{\Vf}{{\mathfrak V}}

\newcommand{\CN}{{\mathcal N}}

\newcommand{\CE}{{\mathcal E}}

\newcommand{\CS}{{\mathcal S}}
\newcommand{\CM}{{\mathcal M}}

\def\limi#1{\displaystyle\lim_{\longrightarrow\atop #1}}

\newcommand{\udelta}{\underline{\delta}}
\newcommand{\ubeta}{\underline{\beta}}

\newcommand{\oW}{{\mathbb W}^{\rm ord}}

\newcommand{\cW}{{\cal W}}

\newcommand{\fX}{{\mathfrak{X}}}
\newcommand{\cX}{{\mathcal{X}}}
\newcommand{\fIG}{{\mathfrak{IG}}}
\newcommand{\fY}{{\mathfrak{Y}}}

\newcommand{\fT}{{\mathfrak{T}}}

\newcommand{\fF}{{\mathfrak{F}}}
\newcommand{\fK}{{\mathfrak{K}}}

\newcommand{\mat}[4]{\left( \begin{array}{cc} {\sharp1} & {\sharp2} \\ {\sharp3} & {\sharp4}
\end{array} \right)}

\begin{document}

\bigskip

\bigskip

\title{Triple product $p$-adic $L$-functions associated to finite slope $p$-adic families of modular forms,}
\author{by Fabrizio Andreatta,  Adrian Iovita}\maketitle

\centerline{\large with an appendix by Eric Urban.}

\begin{abstract}
Let $p$ be a positive prime integer. We construct $p$-adic families of de Rham cohomology classes
and therefore $p$-adic families of nearly overconvergent elliptic modular forms. As an application we define triple product $p$-adic $L$-functions attached to three finite slope families of modular forms satisfying certain assumptions.
\end{abstract}

\tableofcontents

\section{Introduction}

The main theme of this article is that of $p$-adic variation of arithmetic objects. More precisely we will point out a very general geometric construction, called
vector bundles with marked sections, which we claim, when applied to (certain) families of $p$-divisible groups produces $p$-adic variations of certain modular
sheaves naturally existing there. In fact this construction produces all the known $p$-adic families and some which are new. So far this method has been tested on
modular curves and the results are recorded in this article but we think that the method, suitably adapted, works universally.

The motivation for this study is twofold: on the one hand it comes from the desire and need to find a general construction of $p$-adic $L$-functions attached to a
triple of $p$-adic finite slope families of modular forms. It has been known for a while, by work of H.~Hida (\cite{Hida2}), M.~Harris and J.~Tilouine
(\cite{harris_tilouine}), how to attach such a $p$-adic $L$-function to a triple of Hida families (or ordinary $p$-adic families) and its special values have been
investigated in work of M.~Harris and S.~Kudla (\cite{harris_kudla}) and more recently of A.~Ichino (\cite{ichino}) and T.C.~Watson (\cite{watson}). There have been
essays in the literature to extend this construction to finite slope families but so far they were not successful. For example in \cite{UNO} a construction of a
Rankin-Selberg $p$-adic $L$-function (which is a particular case of the Garret-Rankin triple product $p$-adic $L$-function constructed in this article) in the
finite slope case is claimed, but the article had a fatal gap. The gap is explained and fixed using the constructions and results of this article in section \S 7 by
E.~Urban. We refer to \cite{Marco} and the refinements in \cite{HS} for a construction of triple product $p$-adic $L$-functions which interpolate special values in
the balanced region, as opposite to the unbalanced regions considered in this paper and in the references mentioned so far. See also \cite{Loeffler} for an approach
using the Euler system of Beilinson-Flach elements, that provides a construction of two  dimensional ``slices" of the sought for three variable $p$-adic
$L$-function. The second motivation for the study of $p$-adic variation of modular sheaves is connected to our long term effort to provide crystalline
Eichler-Shimura isomorphisms associated to overconvergent eigenforms of finite slopes. This line of inquiry is not followed-up in this article but we hope to report
on such results soon.

\medskip
\noindent

Let us now be more precise and start by briefly reviewing the
triple product $p$-adic $L$-functions in the ordinary case
following the exposition of H.~Darmon and V.~Rotger in
\cite{darmon_rotger}. We will content ourselves to explain a
particular case in the introduction in order to simplify notations
but see the articles quoted or Section \S \ref{sec:tripleordinary}
and Remark \ref{rmk:DRanalogue} of this article for the general
case.

Let $N\ge 5$ be a square free integer and $f$, $g$, $h$ classical, normalized, primitive cuspidal eigenforms for $\Gamma_1(N)$ of weights $k$, $\ell$, $m$
respectively (and trivial characters) which are supposed to be unbalanced, i.e., there is an integer $t\ge 0$ such that $k=\ell+m+2t$. Let $p\ge 5$ be a prime
integer such that $(p,N)=1$ and we assume that $f,g,h$ are all ordinary at $p$. Let $\bof$, $\bg$, $\bh$ be Hida families of modular forms for $\Gamma_1(N)$
interpolating in weights $k$, $\ell$, $m$ the forms $f$, $g$, $h$ respectively. Here $\bof$, $\bg$, $\bh$ are seen as $q$-expansions with coefficients in the finite
flat extensions of $\Lambda:=\Z_p[\![\Z_p^\ast]\!]$ denoted $\Lambda_f$, $\Lambda_g$, $\Lambda_h$ respectively.

Before we start defining the $p$-adic $L$ function attached to $\bof$, $\bg$, $\bh$ let us make a short revisit of $q$-expansions and their properties. If $R$ is a
finite flat extension of $\Lambda$ we denote by $U$, $V\colon R[\![q]\!]\lra R[\![q]\!]$ the following $R$-linear operators: let $\alpha(q)=\sum_{n=0}^\infty
a_nq^n\in R[\![q]\!]$, then $U(\alpha)(q)=\sum_{n=0}^\infty a_{np}q^n$ and $V(\alpha)(q)=\sum_{n=0}^\infty a_nq^{pn}$. We immediately remark that $U\circ V={\rm
Id}_{R[\![q]\!]}$ and define, for $\alpha\in R[\![q]\!]$ as above $$\alpha^{[p]}(q):=\bigl({\rm Id}-V\circ U)\bigr)(\alpha)(q)=\sum_{n\ge 1, (p,n)=1}^\infty
a_nq^n.$$ One sees that $\alpha^{[p]}(q)\in R[\![q]\!]^{U=0}$ and moreover that if $\beta(q)\in R[\![q]\!]^{U=0}$ then $\beta^{[p]}(q)=\beta(q)$, i.e.
$R[\![q]\!]^{U=0}=\bigl(R[\![q]\!] \bigr)^{[p]}$. The operators $U,V$ defined above on $q$-expansions 
preserve the subspaces of $p$-adic modular forms for various weights.

We define the differential operator $d\colon R[\![q]\!]\lra R[\![q]\!]$ to be the $R$-derivation $\displaystyle d:=q\frac{d}{dq}.$ Let us remark that if $s\colon
\Z_p^\ast \lra R^\ast$ is a continuous homomorphism (it is called ``an $R$-valued weight'') it makes sense to define the operator $d^s\colon R[\![q]\!]^{U=0}\lra
R[\![q]\!]^{U=0}$ by
$$
d^s(\sum_{n=1, (n,p)=1}^\infty a_nq^n)=\sum_{n=1, (n,p)=1}^\infty a_ns(n)q^n.
$$
In particular for the universal weight $\Z_p^\ast\lra  \Lambda^\ast=\bigl(\Z_p[[\Z_p^\ast]]\bigr)^\ast \lra R^\ast$ sending $t\in \Z_p^\ast $ to the image in
$R^\ast$ of the grouplike element $[t]\in \Z_p^\ast\subset \Lambda^\ast$, we denote (following \cite{darmon_rotger}), by $d^\bullet$ the corresponding differential
operator on $q$-expansions, i.e., the operator defined by
$$
d^\bullet(\sum_{n=1,(n,p)=1}^\infty a_nq^n):=\sum_{n=1,(p,n)=1}^\infty a_n[n]q^n.
$$

\medskip\noindent
Let us now go back to our three Hida families $\bof$$\in \Lambda_f[\![q]\!], \bg$$\in \Lambda_g[\![q]\!], \bh$$\in \Lambda_h[\![q]\!]$. Following
\cite[Def.~4.4]{darmon_rotger} we define
$$
\mathcal{L}^f_p({\bof,\bg,\bh}):=\frac{\langle \bof, e^{\rm ord}\bigl(d^\bullet(\bg^{[p]})\times \bh  \bigr) \rangle}{\langle \bof, \bof \rangle} \in
{{\Lambda'_f\otimes \Lambda_g\otimes\Lambda_h}},
$$
where $\Lambda'_f$ denotes the total ring of fractions of $\Lambda_f$,
$e^{\rm ord}:={\rm lim}_{r\rightarrow \infty}U^{r!}$ is Hida's
``ordinary projector''  from $p$-adic families of nearly
overconvergent forms as in \cite{darmon_rotger}, to ordinary
modular forms and the inner product $\langle \ ,\ \rangle$ in the
above formula is the Peterson inner product for ordinary families
of weight the weight of $\bof$. We refer to \cite[\S
2.6]{darmon_rotger} for details. Then the specialization of this
three variable $p$-adic $L$-function at a triple of unbalanced
classical weights $(x,y,z)$ (where $d^\bullet$ is specialized at
$d^t$, with $x=y+z+2t$, $t\in \Z_{\ge 0}$) can be expressed as a
square root of the algebraic part of the classical central value
of the triple product of $\bof_x,\bg_y,\bh_z$.

\medskip
\noindent Suppose now that $f$, $g$, $h$ are classical normalized, primitive cuspidal eigenforms as above which have finite slope instead of being ordinary at $p$.
Then let us remark that the formula above defining $\mathcal{L}^f_p(\bof,\bg,\bh)$ makes no sense as there is no finite slope idempotent analogous to $e^{\rm ord}$
apart from the ordinary one. The reason is that the operator $U$  is not compact on $q$-expansions or on
$p$-adic modular forms.  One has to work with finite slope families of modular forms seen as overconvergent sections of the modular sheaves $\fw^{k_f}$,
$\fw^{k_g}$, $\fw^{k_h}$, where $k_f$, $k_g$, $k_h$ are the weights of the families interpolating $f$, $g$, $h$ respectively. Most importantly, instead of the
operator $d$ on $q$-expansions we have to work with a connection $ \nabla_{k_g}$ on a certain de Rham sheaf of weight $k_g$. This makes the whole
construction geometric and before proceeding to the construction of the $p$-adic $L$-function one has to define the new de Rham sheaves and study their properties.

More precisely, let $\cX$ denote the adic analytic space associated to the modular curve $X_1(N)_{\Q_p}$ and for every integer $r\ge 0$ and interval $I=[0,b]$,
$b\in \Z$ let $\cX_{r,I}$ denote the strict neighbourhood of the ordinary locus in $\cX\times W_I$ where the generalized elliptic curve has  a canonical subgroup
of order $1\leq n\leq r+b+1$.  Here $W$ is the weight space, i.e.,  the adic analytic space of analytic points attached to the formal scheme $\Spf(\Lambda)$ and
$W_I$ is a certain open subspace of weights containing $k$, $\ell$, $m$ (for details see Section \S \ref{sec:omega}).

Let now $\fX_{r,I}\lra \fX$ and $\mathfrak{W}_I$  be precisely defined formal models of $\cX_{r,I}\lra \cX$, and respectively of $W_I$, for example
 $\fX$ is the formal completion along its special fiber of the modular curve $X_1(N)_{\Z_p}$.
Let $\pi\colon  E\lra \fX_{r,I}$ be the inverse image of the universal generalized elliptic curve on $\fX$ and define
$\omega_E:=\pi_\ast\Bigl(\Omega^1_{E/\fX_{r,I}}\bigl({\rm log}(\pi^{-1}({\rm cusps})\bigr)\Bigr)$ and ${\rm
H}_E:=\mathbb{R}^1\pi_\ast\Bigl(\Omega^\bullet_{E/\fX_{r,I}}\bigl({\rm log}(\pi^{-1}({\rm cusps})\bigr)\Bigr)$. Then $\omega_E$ is a locally free modular sheaf of
rank one and ${\rm H}_E$ is a locally free modular sheaf of rank two, related by the Hodge filtration exact sequence on $\fX_{r,I}$:
$$
0\lra \omega_E\lra {\rm H}_E\lra \omega_E^\ast\lra 0.
$$
Moreover the Gauss-Manin connection defines a logarithmic connection
$$\nabla\colon {\rm H}_E\lra {\rm H}_E\otimes_{\cO_{\fX_{r,I}}}\Omega^1_{\fX_{r,I}/\mathfrak{W}_I}\bigl({\rm log}({\rm cusps})\bigr).
$$

On $\fX_{r,I}$ we have a family of line bundles $\bigl(\omega_E^{\otimes m}\bigr)_{m\in \N}$ and a family of locally free $\cO_{\fX_{r,I}}$-modules with connections
and Hodge filtrations $\bigl({\rm Sym}^m({\rm H}_E), {\rm Fil}_{m,\bullet}, \nabla_m  \bigr)_{m\in \N}$ and the main tasks before us is to $p$-adically interpolate
these two families by using weights in $W_I$.

Let us recall that the first family has already  been interpolated in various degrees of generality in \cite{halo_spectral}, \cite{andreatta_iovita_stevens2} and
\cite{SiegelAIP}.

More precisely if $\alpha\in W_I$ is any weight there is a sheaf $\fw^\alpha$ on $\fX_{r,I}$ such that if $\alpha\in \Z$ then $\fw^\alpha$ and $\omega_E^\alpha$
coincide on the analytic space $\cX_{r,I}$ and such that the elements of ${\rm H}^0(\fX_{r,I}, \fw^\alpha)$ are (integral models of) the overconvergent modular
forms or families of weight $\alpha$ and tame level $N$.

In particular, returning for a moment to our construction of the $p$-adic $L$-function, given $f$, $g$, $h$  we have modular sheaves $\fw^{k_f}$, $\fw^{k_g}$,
$\fw^{k_h}$ and (integral) families $\omega_f\in {\rm H}^0(\fX_{r,I}, \fw^{k_f})$, $\omega_g\in {\rm H}^0(\fX_{r,I}, \fw^{k_g})$, $\omega_h\in {\rm H}^0(\fX_{r,I},
\fw^{k_h})$ interpolating $f$, $g$, $h$ respectively in weights $k$, $\ell$, $m$.

The integral $p$-adic interpolation of the families $\bigl({\rm Sym}^m({\rm H}_E), {\rm Fil}_{m,\bullet}, \nabla_m  \bigr)_{m\in \N}$ in this article is new and it
follows from using the formal vector bundle with marked sections attached to a sheaf like ${\rm H}_E$ and a section of it coming from a generator of the Cartier
dual of the canonical subgroup of $E$ via the map ${\rm dlog}$. Our first result is the following, where we summarize Theorem \ref{thm:descentbWk}, Theorem
\ref{theorem:griffith}, Section \ref{sec:upoperator} and Theorem \ref{theorem:mostgeneral}:

\begin{theorem}  For every weight $\alpha\in W_I$ there exists a formal sheaf $\bW_\alpha$ on $\fX_{r,I}$ with meromorphic connection $\nabla_\alpha$ and filtration
${\rm Fil}_\bullet(\bW_\alpha)$ which define on the adic analytic fiber $\cX_{r,I}$ a sheaf of Banach modules $\bW_\alpha^{\rm an}$ with   a connection
$\nabla_\alpha$ and filtration ${\rm Fil}_\bullet(\bW_\alpha^{\rm an})$ satisfying Griffith's transversality.

Moreover if $\alpha\in\Z_{\ge 0}$ then $\bigl({\rm Sym}^\alpha({\rm H}_E), {\rm Fil}_{\alpha,\bullet}, \nabla_\alpha\bigr)$ is canonically a submodule (with
connection and filtration) of the sheaf defined by $\bigl(\bW_\alpha^{\rm an}, {\rm Fil}_\bullet(\bW_\alpha^{\rm an}), \nabla_\alpha   \bigr)$ on $\cX_{r,I}$ and
their global sections of slopes $\le \beta$, for $\beta<\alpha-1$ are equal.

Finally we show that there is $b\ge r$ such that for every $w\in {\rm H}^0(\cX_{r,I}, \bW^{\rm an}_k)^{U=0}$ and  for every weight $\gamma\in W_I$ satisfying the
conditions of Assumption (\ref{ass10}), there is a  section $\nabla_\alpha^\gamma(w)\in {\rm H}^0(\cX_{b,I}, \bW^{\rm an}_{\alpha+2\gamma})$ whose $q$-expansion is
$d^\gamma(w(q))$.
\end{theorem}

The Assumption (\ref{ass10}) on $\alpha$ and $\gamma$ for the existence of $\nabla_\alpha^\gamma(w)$ amounts to demand that $\alpha$ and $\gamma$ are $p$-adically
close to classical weights. In view of Remark \ref{rmk:Nablasonqexp} it seems difficult to weaken these assumptions, namely one does not have a formula for
$\nabla_\alpha^\gamma(w)$ valid for $\alpha$ and $\gamma$ varying over the whole weight space. As these are the technical tools needed to construct the $p$-adic
$L$-function in the finite slope case in Definition \ref{def:finiteslopel} we get an interpolation property over this type of regions of weight space. As in
loc.~cit., though, one needs to take overconvergent projections of forms of the type $\nabla_\alpha^\gamma(w)$ times an overconvergent form, it might still be
possible that an interpolation for the triple product $L$-function exists more generally as hinted in \cite{Loeffler} and in \S \ref{sec:finalremark}.

We'd like to point out that Zheng Liu has defined a sheaf similar to the adic analytic sheaf $\bW_\alpha^{\rm an}$ and a connection  $\nabla_\alpha$ on it in
\cite{zheng} but this is not sufficient to define the triple product $p$-adic $L$-functions in the finite slope case.  The $q$-expansions of sections of the sheaves
$\bW_\alpha$ of the Theorem are called {\it nearly overconvergent modular forms}; see Definition \ref{def:overconv} and the Remark \ref{rmk:overconv} for
connections with previous work of Harron-Xiao \cite{HX}, Darmon-Rotger \cite{darmon_rotger} and Urban \cite{UNO}.

\bigskip
\noindent We now describe the structure of the article. In Chapter \S \ref{sec:FVBMS} we introduce one of the main players of this article, the formal vector
bundles with marked sections, and study their main properties. In other words we present a geometric construction associating to every formal scheme $S$ (which has
an invertible ideal of definition $\cI$) and data $(\cE, s_1,\ldots,s_d)$ consisting of a locally free $\cO_S$-module of rank $n\ge 1$ and ``marked global sections"
$s_1,\ldots,s_d$ of $\cE/\cI\cE$ (satisfying certain properties) a formal scheme $\pi\colon\bV_0(\cE, s_1,\ldots,s_d)\lra S$ whose sheaf of functions is
``interpolable".

We show that if $(\cE, s_1,\ldots,s_d)$ has extra structure e.g. a connection, a filtration, a group action then the sheaf $\pi_\ast\bigl(\cO_{\bV_0(\cE,
s_1,\ldots,s_d)} \bigr)$ has an induced extra structure of a similar nature.

In Chapter \S \ref{sec:appmodular} we apply the above construction to modular curves and locally free sheaves which are modifications of $\omega_E$ and respectively
${\rm H}_E$. The marked section will be the image of a generator of the Cartier dual of the canonical subgroup via the map ${\rm dlog}$, and therefore we have to
place ourselves on a partial formal blow-up of a formal modular curve where such a canonical subgroup exists. The sheaves $\omega_E$ and respectively ${\rm H}_E$
have to be modified in order for the section coming from the dual canonical subgroup to satisfy the required property of a ``marked section".  This way we associate
to every weight $\alpha\in W_I$ a sheaf $\fw^\alpha$ and a triple $\Bigl(\bW_\alpha, \nabla_\alpha, {\rm Fil}_\bullet(\bW_\alpha)\Bigr)$, consisting of a sheaf $\bW_\alpha$, a meromorphic connection
$\nabla_\alpha$ on $\bW_\alpha$ and an increasing  filtration of $\bW_\alpha$  such that ${\rm Fil}_0(\bW_\alpha)=\fw^\alpha$. Furthermore we prove in Theorem \ref{thm:descentbWk} that,
forgetting the connection, the sheaves $\Bigl(\bW_\alpha, {\rm Fil}_\bullet(\bW_\alpha)\Bigr)$ can be extended to the whole interval $I=[0,\infty]$ and we provide in Theorem
\ref{thm:bWkinfty} an explicit description of these sheaves at the points at infinity.

On the global sections of $\fw^\alpha$ and of $\bW_\alpha$ as well as on the de Rham cohomology groups with coefficients in $(\bW_\alpha, \nabla_\alpha)$ we have natural, linear
actions of Hecke operators such that $U$ is compact.

In Chapter \S \ref{sec:IterateManin}, which is the main technical chapter of the article, we show that if $\alpha$, $\gamma$ are weights satisfying certain conditions (see
(\ref{ass}) and $w$ is a global section of $\bW^0_\alpha$ such that $U(w)=0$, then there is a canonical section denoted $\nabla_\alpha^\gamma(w)$ of the sheaf $\bW^0_{\alpha+2\gamma}$
over a ``smaller strict neighbourhood of the ordinary locus" whose $q$-expansion is $d^\gamma(w(q))$.

Having thus defined all the technical tools needed, in Chapter \S \ref{sec:triple} we review the construction of the triple product $p$-adic $L$-function in the
ordinary case in all its generality and construct the triple product $p$-adic $L$-functions attached to finite slope $p$-adic families of modular forms.

In Appendix I we show how given a general $p$-divisible group $G$, over a formal scheme, ``which is not too supersingular", one can attach to   its
sheaves $\omega_G$ and ${\rm H}_G$ (in fact to modifications of them) and a basis of the points of the Cartier dual of its canonical subgroup, canonical formal
vector bundles with marked sections. We think that if this construction is applied to certain Shimura varieties of PEL type (for example to Hilbert modular varieties) it would be possible to define triple product
$p$-adic $L$-functions in that setting. It should be clear though that we do not perform that construction here.

Finally Appendix II, written by E. Urban contains a corrigendum to the article \cite{UNO}: the author (of \cite{UNO} and this appendix) explains and fixes the gap
in the cited article using the results of the previous sections of this paper.

\bigskip
{\bf Acknowledgements} We are grateful to H. Darmon for many
stimulating discussions on triple product $p$-adic $L$-functions
without which this article would not have been possible. We thank:
Zheng Liu for suggesting an improvement of a proof in chapter \S
4, Eric Urban for interesting discussions pertaining to twists of
modular forms and for suggesting Definition
\ref{prop:thetachi} and the members of the ``working seminar
on $p$-adic $L$-functions" in Montreal for their interest,
questions and suggestions related to this work. Finally we  thank Yangyu Fan for   the careful reading of a previous version of the paper. We are grateful to the referee whose remarks helped us improve the article.

The beginning  of this research  was done while the second author was a visitor to the Max Planck Institute and it was finialized two years later while he was a Simons Fellow in Mathematics. He thanks the Max Planck Institute for its hospitality and the Simons Foundation for the additional
sabbatical provided which made this research possible.

\section{Formal vector bundles with marked sections.}
\label{sec:FVBMS}

In this chapter we present a general construction which associates
to every formal scheme $S$ with ideal of definition $\cI$ which is
supposed to be invertible and data $(\cE, s_1,s_2,\ldots,s_m)$
consisting of a locally free sheaf $\cE$ of $\cO_S$-modules of
rank $n\geq m$ on $S$ and global sections $s_1,s_2,\ldots, s_m$ of
$\overline{\cE}:=\cE/\cI\cE$ which generate a locally free direct
summand of rank $m$ of $\overline{\cE}$, a formal scheme
$\pi:\bV_0(\cE, s_1,s_2,\ldots,s_m)\lra S$ called {\bf vector
bundle with marked sections}, with the property that ${\rm
H}^0\bigl(\bV_0(\cE,s_1,s_2,\ldots,s_m),
\cO_{\bV_0(\cE,s_1,s_2,\ldots,s_m)}\bigr)$ can be seen as  the
ring of $R:={\rm H}^0(S, \cO_S)$-valued analytic functions on the
set
$$
E_0:=\{v:{\rm H}^0(S, \cE)\lra R \ |\ v\mbox{ is }R\mbox{ linear
and } v(\mbox{mod }\cI)(s_i)=1, i=1,2,\ldots,m\}.
$$
The construction is functorial in $(\cE, s_1,s_2,\ldots,s_m)$ and
if $\cE$ has additional structure compatible with
$(s_1,s_2,\ldots,s_m)$, such as a filtration, a connection, a
group action, then the sheaf
$\pi_\ast\bigl(\cO_{\bV_0(\cE,s_1,s_2,\ldots,s_m)}  \bigr)$ has an
induced extra structure of a similar nature.

\subsection{Formal vector bundles.}

Consider as above a formal scheme $S$ with invertible ideal of definition $\cI\subset \cO_S$. Denote by $\overline{S}$ the scheme with structure sheaf defined by
$\cO_S/\cI$.

In this section all formal schemes considered will be formal
schemes $f\colon T \to S$ over $S$, with ideal of definition
$f^{\ast} (\cI)\subset \cO_T$ which is an invertible ideal, i.e.,
locally on $T$ it is generated by an element that is not a zero
divisor.

\begin{definition}\label{def:fvb}
A formal vector bundle of rank $n$ over $S$ is a formal vector group scheme 
$f\colon X\lra S$ over $S$, locally on $S$ isomorphic to the $n$-fold product of the additive group ${\mathbb G}_{a,S}^n$.
Equivalently it is a formal scheme $f\colon X\lra S$   such that there exist
an affine open covering $\{U_i\}_{i\in I}$ of $S$ and for every
$i\in I$ an isomorphism $\psi_i\colon
X\vert_{U_i}:=f^{-1}(U_i)\cong \bA^n_{U_i}$, where $\bA^n_{U_i}$
is the formal $n$-dimensional affine space over $U_i$, such that
for every $i$, $j\in I$ and every affine open formal subscheme
$U\subset U_i\cap U_j$, the automorphism induced by $\psi_j\circ
\psi_i^{-1}$ on $\bA^n_U$ is a linear automorphism.

If $f\colon X\to S$ and $f'\colon Y\to S$ are two vector bundles
over $S$ of rank $n$ and $n'$ respectively, a morphism
(resp.~isomorphism) $g\colon X\lra Y$ of formal vector bundles
 over $S$ is a morphism (resp.~an isomorphism) as formal vector group schemes.

If we have charts
$\bigl(\{U_i\}_{i\in I}, \{\psi_i\}_{i\in I}\bigr)$ and
$\bigl(\{U'_j\}_{j\in J}, \{\psi'_j\}_{j\in J}\bigr)$ of $X$ and $Y$ respectively, a morphism
(resp.~isomorphism) $g\colon X\lra Y$ of formal vector bundles
 over $S$ is a morphism (resp.~an isomorphism) of formal
schemes over $S$ such that for every $i\in I$, every $j\in J$
and every affine open formal subscheme $U\subset U_i\cap U_j'$ the
induced map
$$\bA^n_{U} \stackrel{\psi_i^{-1}}{\lra} X\vert_{U} \stackrel{g\vert_U}{\lra}  Y \vert_{U} \stackrel{\psi_j'}{\lra}
\bA^{n'}_{U}$$is a linear map.
\end{definition}

\begin{lemma}\label{lem:V(E)}
Let $\cE$ be a locally free $\cO_S$-module of rank $n$ over $S$. Then there exists a unique formal vector bundle $\bV(\cE)$ of  rank $n$ over $S$ representing the
functor that associates to any formal scheme $t\colon T \to S$ the ${\rm H}^0(T,\cO_T)$-module $\mathrm{Hom}_{\cO_T}\bigl(t^\ast(\cE), \cO_T \bigr)$ of homomorphisms
$t^\ast(\cE)\to \cO_T $ as $\cO_T$-modules.

This contravariant functor $\bV$ defines an equivalence of
categories between the category of locally free $\cO_S$-modules of
constant rank and the category of formal vector bundle of finite
rank over $S$ and this equivalence preserves the notion of rank.
\end{lemma}
\begin{proof}
Let $\cE$ be a locally free $\cO_S$-module of rank $n$ over $S$.
Define $f\colon \bV(\cE)\to S$ to be the formal scheme over $S$
defined by the $\cI$-adic completion $\widehat{{\rm Sym}}_S(\cE)$
of the $\cO_S$-symmetric algebra  ${\rm Sym}_S(\cE)=\oplus_{i\in
\N}{\rm Sym}^i_{\cO_S}(\cE)$ associated to $\cE$. Consider any
affine covering $\{U_i\}_{i\in I}$ of $S$ such that
$\cE\vert_{U_i}$ is a free $\cO_{U_i}$-module of rank $n$. If
$e_{1,i}, e_{2,i},\ldots,e_{n,i}$ is a basis of $\cE\vert_{U_i}$
as $\cO_{U_i}$-module, then we have natural isomorphisms of
$\cO_{U_i}$-algebras $\psi_i\colon {\rm
Sym}_S(\cE)\vert_{U_i}\cong \cO_{U_i}\langle
X_1,X_2,\ldots,X_n\rangle$ sending $e_{j,i}\mapsto X_j$. One
readily checks that $\bigl(\bV(\cE), f, \{U_i\}_{i\in I}, \{
\psi_i\}_{i\in I}  \bigr)$  is  a vector bundle of rank $n$ over
$S$. For any formal scheme $t\colon T \to S$ the $T$-valued points
of $\bV(\cE)$ (over $t$), correspond bijectively and functorially
in $T$ and in $t$ with the $\cO_T$-linear homomorphisms
$t^\ast\bigl({\rm Sym}^1_{\cO_S}(\cE)\bigr)=t^\ast(\cE)\to \cO_T$.
This provides the claimed representability.

We exhibit an inverse to this functor. Let $f\colon X\to S$ be a
formal vector bundle of rank $n$ over $S$. Define $\cE$ to be the
presheaf of sets that associates to any open formal subscheme
$U\subset S$ the set of sections of $X$ over $U$. We leave it to the
reader to show that $\cE(U)$ has a natural structure of
${\rm H}^0(U,\cO_U)$-module that makes $\cE$ a locally free
$\cO_S$-module of rank $n$. We then associate to $X$ the
$\cO_S$-module $\cE^\vee$. This is the sought for inverse.

\end{proof}

\subsection{Formal vector bundles with marked sections.}\label{sec:vbfs}

For  a locally free $\cO_S$-module  $\cE$ of rank $n$ we denote by
$\overline{\cE}$ the associated $\cO_{\overline{S}}$-module. Let
$s_1,\ldots, s_m$, with $m\leq n$, be  sections in
${\rm H}^0(\overline{S},\overline{\cE})$ such that the induced map
$\oplus_{i=1}^m \cO_{\overline{S}} \mapsto \overline{\cE}$,
sending $\sum_i a_i\to \sum_i a_i s_i$, identifies
$\cO_{\overline{S}}^m$ with a locally direct summand of
$\overline{\cE}$.

\begin{definition}\label{def:V0}
Define $\bV_0(\cE, s_1,s_2,\ldots,s_m)$ as the sub-functor of
$\bV(\cE)$ that associates to any formal scheme $t\colon T \to S$
the subset of sections $\rho\in \bV(\cE)(T)={\rm
H}^0\bigl(T,t^\ast(\cE)^\vee\bigr)$ whose reduction
$\overline{\rho}:=\rho$ modulo $\cI$ satisfies
$\overline{\rho}\bigl(t^\ast(s_i)\bigr))= 1$ for every
$i=1,\ldots,m$.

\medskip

Notice that this construction is functorial with respect to the tuples $(\cE,s_1,\ldots,s_m)$. Namely given a homomorphism $g\colon \cE' \to \cE$ of locally free $\cO_S$-modules
of finite rank and sections $s'_1,\ldots,s_m'\in \overline{\cE}'$ and $s_1,\ldots,s_m\in \overline{\cE}$, satisfying the requirements above and such that
$\overline{g}(s'_i)=s_i$ for every $i=1,\ldots,m$, we obtain a commutative diagram

$$\begin{matrix}\bV_0(\cE, s_1,\ldots,s_m) & \lra & \bV(\cE) \cr
\downarrow & & \downarrow \cr \bV_0(\cE', s'_1,\ldots,s'_m) & \lra
& \bV(\cE').\cr \end{matrix}$$

\end{definition}

\begin{lemma}\label{lem:V0}
The morphism $\bV_0(\cE, s_1,\ldots,s_m) \to \bV(\cE)$ is
represented by an open formal subscheme of a formal
$\cI$-admissible blow up of $\bV(\cE)$.
\end{lemma}
\begin{proof} The sections $s_1,\ldots,s_m$ define a subsheaf of $\cO_{\overline{S}}$-modules
$\cO_{\overline{S}}^m \subset \overline{\cE}$ with quotient $Q$ which is a
locally free sheaf of $\cO_{\overline{S}}$-modules of rank $n-m$. There is a
a quotient map ${\rm Sym}^\bullet(\overline{\cE}):=\oplus_{i\in \N}{\rm
Sym}^i_{\cO_{\overline{S}}}(\overline{\cE}) \to \oplus_{i\in
\N}{\rm Sym}^i_{\cO_{\overline{S}}}(Q)$ whose kernel is the ideal
$(s_1-1,\ldots,s_m-1)$. Taking the induced map of spectra, relative to $\overline{S}$, such quotient map defines a closed subscheme $C$ in
$\overline{\bV}(\cE):={\rm Spec}\bigl({\rm Sym}^\bullet(\overline{\cE})  \bigr)$. Let $\overline{\cJ}:=(s_1-1,\ldots,s_m-1)$ be the corresponding ideal sheaf
and let $\cJ\subset \widehat{{\rm Sym}}_S(\cE)$ be its inverse image.  Consider the $\cI$-adic completion $\bB$ of the open formal
subscheme of the blow up of $\bV(\cE)$ with respect to the ideal
$\cJ$, open defined by the requirement that the ideal generated by
the inverse image of $\cJ$ coincides with the ideal generated by
the inverse image of $\cI$.

In local coordinates if $U=\Spf(R)\subset S$ is an open formal
subscheme such that $\cI$ is generated by $\alpha\in R$,
$\cE\vert_{U_i}$ is free of rank $m$ with basis $e_1,\ldots,e_n$
such that $e_i\equiv s_i$ modulo $\alpha$ for $i=1,\ldots,m$ and
$e_{m+1},\ldots,e_n$ modulo $\alpha$ define a basis of $Q$, then
$\bV(\cE)\vert_{U}$ is the formal scheme associated to $R\langle
X_1,\ldots,X_n\rangle$ and $\cJ\vert_{U}$ is the ideal
$\bigl(\alpha, X_1-1,\ldots,X_m-1\bigr)$. In particular
$\bB\vert_{U}=R\langle Z_1,\ldots,Z_m,X_{m+1},\ldots,X_n\rangle$
with morphism $\bB\vert_{U}\to \bV(\cE)\vert_{U}$ defined by
sending $X_i\to X_i$ for $i=m+1,\ldots,n$ and $X_i\to 1+ \alpha
Z_i$ for $i=1,\ldots,m$.

For every formal scheme $T$ over $U$ a section $\rho\in
\bV(\cE)(T)$ is defined by the images $a_1,\ldots,a_n$ of
$X_1,\ldots,X_n$ that we can identify via the identification
$\bV(\cE)(T)=\mathrm{Hom}_{\cO_T}\bigl(t^\ast(\cE), \cO_T \bigr)$
with the images of $t^\ast(e_1),\ldots,t^\ast(e_n)$ via $\rho$.
Then $\rho$ lies in $\bV_0(\cE, s_1,\ldots,s_m)(T)$ if and only if
$\rho\bigl(t^\ast(e_i)\bigr)=a_i\equiv 1$ modulo $\alpha$ for
$i=1,\ldots,m$. Hence $\rho$ uniquely lifts to a $T$-valued point
of $\bB\vert_{U}$ given by sending $X_i\to a_i$ for
$i=m+1,\ldots,n$ and $Z_i\mapsto \frac{a_i-1}{\alpha}$ for
$i=1,\ldots,m$ (which is well defined as $\alpha$ is not a zero
divisor in $\cO_T$). Viceversa any $T$-valued point of
$\bB\vert_{U}$ defines a  section $\rho\in \bV(\cE)(T)$, by the
formula above, that in fact lies in $\bV_0(\cE,
s_1,\ldots,s_m)(T)$ by construction.

One verifies that the isomorphisms $\bB\vert_{U_i}\cong \bV_0(\cE,
s_1,\ldots,s_m)\vert_{U_i}$ one obtains in this way varying $U_i$
glue and provide the sought for isomorphism $\bB\cong  \bV_0(\cE,
s_1,\ldots,s_m)$ as formal schemes over $ \bV(\cE)$.

The functoriality is immediately checked.

\end{proof}

\subsection{Filtrations on the sheaf of functions of a formal vector bundle with marked sections}\label{sec:vbfil}

Let $\cE$ be a locally free $\cO_S$-module  of rank $n$ and assume
that there exists an $\cO_S$-submodule $\cF\subset \cE$, locally
free as $\cO_S$-module of rank $m$, which is a locally direct
summand in $ \cE$. Equivalently $\cE/\cF$ is also locally free as
$\cO_S$-module of rank $n-m$. Assume also that the global sections
$s_1,\ldots,s_m$ of $\overline{\cE}$ as in \S\ref{sec:vbfs} define
an $\cO_{\overline{S}}$-basis of $\overline{\cF}$. By
the functoriality property in Definition \ref{def:V0} we obtain a commutative
diagram

$$\begin{matrix} \bV_0(\cE, s_1,\ldots,s_m) & \lra & \bV(\cE) \cr
\downarrow & & \downarrow \cr \bV_0(\cF, s_1,\ldots,s_m) & \lra &
\bV(\cF).
\end{matrix}$$

Denote by $f\colon \bV(\cE) \to S$  and $f_0\colon \bV_0(\cE,
s_1,\ldots,s_m) \to S$ the structural morphisms. Notice that the morphism $\bV(\cE)\to \bV(\cF)$ is a principal homogeneous space under
the action of the formal vector group scheme $\bV(\cE/\cF)$ (the action is provided by the inclusion of formal vector group schemes $\bV(\cE/\cF)\subset \bV(\cE)$ and the group scheme structure on $\bV(\cE)$; the fact that it is a principal homogeneous space follows as locally on $S$ one can choose a splitting of the projection $\cE\to \cE/\cF$ which identifies $\bV(\cE)$ with the product  $\bV(\cF)\times_S \bV(\cE/\cF)$). 

\begin{lemma} The diagram above is cartesian. In particular, the vertical morphisms are principal homogenous spaces under the formal vector group scheme
$\bV(\cE/\cF)$.

\end{lemma}

\begin{proof}
Let $U=\Spf(R)$ be an  affine formal subscheme of $S$ such that
$\cI\vert_U$ is generated by $\alpha\in R$ and $\cF$, $\cE$ over
$U$ are free with basis $e_1,\ldots,e_m$, resp.~$e_1,\ldots,e_m,
f_1,\ldots,f_{n-m}$ and $e_i\equiv s_i$ modulo $\alpha$ for
$i=1,\ldots,m$ and $f_1,\ldots,f_{n-m}$ define  the complementary
direct summand of $\overline{\cF}$ in $\overline{E}$. Then
$\bV(\cF)\vert_U=\Spf\bigl(R\langle X_1,\ldots,X_m\rangle\bigr)$, $\bV(\cE)\vert_U=\Spf\bigl(R\langle X_1,\ldots,X_m,Y_1,\ldots,
Y_{n-m}\rangle\bigr)$, $\bV_0(\cF,
s_1,\ldots,s_m)\vert_U=\Spf\bigl(R\langle
Z_1,\ldots,Z_m\rangle\bigr)$, $\bV_0(\cE,
s_1,\ldots,s_m)\vert_U=\Spf\bigl(R\langle
Z_1,\ldots,Z_m,Y_1,\ldots, Y_{n-m}\rangle\bigr)$ where
$X_i=1+\alpha Z_i$ for $i=1,\ldots,m$. The statement
follows.

\end{proof}

\begin{corollary}\label{cor:fil}
With the notations above,  $f_{0,\ast}  \cO_{\bV_0(\cE,
s_1,\ldots,s_m)}$ is endowed with an increasing filtration
$\Fil_\bullet f_{0,\ast} \cO_{\bV_0(\cE, s_1,\ldots,s_m)}$ with
graded pieces

$$\mathrm{Gr}_h f_{0,\ast} \cO_{\bV_0(\cE, s_1,\ldots,s_m)}\cong f_{0,\ast}  \cO_{\bV_0(\cF, s_1,\ldots,s_m)}\otimes_{\cO_S} \mathrm{Sym}^h(\cE/\cF).$$

The filtration is characterized by the following local
description. If $U=\Spf(R)\subset S$ is an open formal affine
subscheme such that $\cF$, $\cE$ over $U$ are free with basis
$e_1,\ldots,e_m$, respectively $e_1,\ldots,e_m, f_1,\ldots,f_{n-m}$ so
that
$$\bV_0(\cF, s_1,\ldots,s_m)\vert_U=\Spf\bigl(R\langle
Z_1,\ldots,Z_m\rangle\bigr), \bV_0(\cE,
s_1,\ldots,s_m)\vert_U=\Spf\bigl(R\langle
Z_1,\ldots,Z_m,Y_1,\ldots, Y_{n-m}\rangle\bigr),
$$
then  $\Fil_h
f_{0,\ast} \cO_{\bV_0(\cE, s_1,\ldots,s_m)}(U)$ consists of the
polynomials of degree $\leq h$ in the variables $Y_1,\ldots,
Y_{n-m}$ with coefficients in $R\langle Z_1,\ldots,Z_m\rangle$.

\end{corollary}
\begin{proof} We use the fact that $\bV_0(\cE, s_1,\ldots,s_m)\to \bV_0(\cF, s_1,\ldots,s_m)$ is a principal  homogenous spaces under
$\bV(\cE/\cF)$ to prove that the local definition of  $\Fil_h
f_{0,\ast}  \cO_{\bV_0(\cE, s_1,\ldots,s_m)}$ is well defined and
glues for varying $U$'s. 

If $U=\Spf(R)\subset S$ is an open formal affine, any other  choice of bases defines new coordinates
$X_1',\ldots,X_m',Y_1',\ldots, Y_{n-m}'$ that are related to
$X_1,\ldots,X_m,Y_1,\ldots, Y_{n-m}$ by an $R$-linear
transformation. In particular the induced map $R\langle
X_1,\ldots,X_m,Y_1,\ldots, Y_{n-m}\rangle\cong R\langle
X_1',\ldots,X_m',Y_1',\ldots, Y_{n-m}'\rangle$ sends each
$Y_i$ to an $R$-linear combination of the
$X_1',\ldots,X_m',Y_1',\ldots, Y_{n-m}'$ and  is then an affine
transformation relative to $R\langle X_1',\ldots,X_m'\rangle=
R\langle X_1,\ldots,X_m\rangle$. The second claim follows as well.

\end{proof}

The construction of the filtration is clearly functorial. Namely given a homomorphism $g\colon \cE' \to \cE$ of locally free $\cO_S$-modules of finite rank,
$\cF'\subset \cE'$ and $\cF\subset \cE$,  locally free as $\cO_S$-modules of rank $d$ and  locally direct summands such that $g(\cF')\subset \cF$, and sections $s'_1,\ldots,s_m'\in
\overline{\cE}'$ and $s_1,\ldots,s_m\in \overline{\cE}$, satisfying the requirements above and such that $\overline{g}(s'_i)=s_i$ for every $i=1,\ldots,m$, we
obtain

\begin{corollary}\label{cor:functfil} Let $f_0\colon \bV_0(\cE, s_1,\ldots,s_m)\to S$ and $f_0'\colon  \bV_0(\cE', s'_1,\ldots,s'_m)\to S$ be the structural morphism. The morphism $g\colon
f_{0,\ast}'\cO_{\bV_0(\cE', s'_1,\ldots,s'_m)} \to
f_{0,\ast}\cO_{\bV_0(\cE, s_1,\ldots,s_m)}$ defined by $g$ (see
Definition \ref{def:V0}) sends $ \Fil_h f_{0,\ast}'
\cO_{\bV_0(\cE', s'_1,\ldots,s'_m)}$ to $\Fil_h f_{0,\ast}
\cO_{\bV_0(\cE, s_1,\ldots,s_m)}$.
\end{corollary}

\subsection{Connections on the sheaf of functions of a formal vector bundle with marked sections}\label{sec:fvvconenction}

Suppose that  we have fixed a $\Z_p$-algebra $A_0$ and an element
$\tau\in A_0$ such that $A_0$ is $\tau$-adically complete and
separated. Let $S$ be a formal scheme locally of finite
type over $\Spf(A_0)$ such that the topology of $S$ is the
$\tau$-adic topology, i.e. $\cI=\tau\cO_S$. We let $\Omega^1_{S/A_0}$ be the
$\cO_S$-module of continuous Kh\"aler differentials.

Consider a locally free $\cO_S$-module $\cE$ endowed with an
integrable connection $\nabla\colon \cE \to \cE \otimes_{\cO_S}
\Omega^1_{S/A_0}$. Assume that we have fixed sections
$s_1,\ldots,s_m \in \overline{\cE}$ as in \S\ref{sec:vbfs} which
are horizontal for the reduction of $\nabla$ modulo $\cI$. Let
$f_0\colon \bV_0(\cE, s_1,\ldots,s_m) \to S$ be the structural
morphism. We explain how $\nabla$ defines an integrable connection
$$\nabla_0\colon f_{0,\ast} \cO_{\bV_0(\cE, s_1,\ldots,s_m)} \to f_{0,\ast}
\cO_{\bV_0(\cE, s_1,\ldots,s_m)} \widehat{\otimes}
\Omega^1_{S/A_0}.$$

\smallskip

{\it Grothendieck's description of integrable connections:}  First
of all recall Grothendieck's approach to connections (see for
example \cite{berthelot_ogus} section \S 2). Let
$\cP_{S/A_0}:=S\times_{A_0} S$ and let $\Delta\colon S \to \cP_S$
be the diagonal embedding. It is a locally closed immersion and we
let $\cP^{(1)}_{S/A_0}$ be the first infinitesimal neighborhood of
$\Delta$: if locally on $S \times_{A_0} S$ the morphism $\Delta$
is the closed immersion defined by an ideal $\cJ$, then
$\cP^{(1)}_{S/A_0}\subset S\times_{A_0} S$ is defined by $\cJ^2$.
We have the two projections $j_1$, $j_2\colon \cP^{(1)}_{S/A_0}
\to S$. Then, to give an integrable connection $\nabla\colon M
\lra M\otimes_{\cO_S}\Omega^1_{S/A_0}$ on a locally free
$\cO_S$-module $M$ of finite rank,  is equivalent to giving an
isomorphism of $\cO_{\cP^{(1)}_{S/A_0}}$-modules $\epsilon\colon
j_2^\ast(M):=\cO_{\cP^{(1)}_{S/A_0}}\otimes_{\cO_S} M\cong
j_1^\ast(M):=M\otimes_{\cO_S} \cO_{\cP^{(1)}_{S/A_0}}$ such that
$\Delta^\ast(\epsilon)={\rm Id}$ on $M$ and $\epsilon$ satisfies a
suitable cocycle condition with respect to the three possible
pull-backs of $\epsilon$ to $S\times_{A_0} S\times_{A_0} S$. In
fact the relationship between $\epsilon$ and $\nabla$ is given by
the following formula, for every $x\in M$
$$
\epsilon(1\otimes x)=x\otimes 1+\nabla(x), \mbox{ where }
\nabla(x)\in M\otimes_{\cO_S} \bigl(\cJ/\cJ^2\bigr)\cong
M\otimes_{\cO_S}\Omega^1_{S/A_0}.
$$

\
\begin{remark} Let us remark that with notations as above, even
if $M$ is an arbitrary quasi-coherent $\cO_S$-module (i.e. not necessarily locally free
of finite rank) and $\epsilon\colon j_2^\ast(M)\cong j_1^\ast(M)$ is an
$\cO_{\cP^{(1)}_{S/A_0}}$-linear
isomorphism such that
$\Delta^\ast(\epsilon)={\rm Id}_M$, then $\epsilon$ defines a
connection $\nabla\colon M\lra M\hat{\otimes}_{\cO_S}\Omega^1_{S/A_0}$ by
the formula: $\nabla(x)=\epsilon(1\otimes x)-x\otimes 1$.

\end{remark}

\medskip

Consider now the given  locally free $\cO_S$-module $\cE$ with
integrable connection $\nabla\colon \cE \to \cE \otimes_{\cO_S}
\Omega^1_{S/A_0}$ and with sections $s_1,\ldots,s_m\in
\overline{\cE}$ horizontal for the reduction of $\nabla$ modulo
$\cI$.  This means that the associated isomorphism $\epsilon\colon
j_2^\ast(\cE)\lra j_1^\ast(\cE)$ has the property that its
reduction $\overline{\epsilon}$ modulo $\cI$ satisfies
$\overline{\epsilon}\bigl(j_2^\ast(s_i)\bigr)=j_1^\ast(s_i)$ for
every $i=1,\ldots,m$. We deduce from the functoriality of
Definition \ref{def:V0} that $\epsilon$ defines compatible
isomorphisms of formal schemes over $S$:
$$\begin{matrix}
\cP^{(1)}_{S/A_0} \times_S \bV_0(\cE, s_1,\ldots,s_m) &
\stackrel{\epsilon_0}{\lra} &   \bV_0(\cE, s_1,\ldots,s_m)
\times_S \cP^{(1)}_{S/A_0}\cr \downarrow & & \downarrow \cr
\cP^{(1)}_{S/A_0} \times_S \bV(\cE) & \stackrel{\epsilon'}{\lra} &
\bV(\cE) \times_S  \cP^{(1)}_{S/A_0}\cr\end{matrix}$$such that
$\Delta^\ast(\epsilon_0)={\rm Id}$ and $\Delta^\ast(\epsilon)={\rm
Id}$. Passing to functions  we obtain compatible isomorphisms

$$\begin{matrix}
j_2^\ast\bigl(f_\ast \cO_{\bV(\cE)} \bigr) &
\stackrel{\epsilon^{',\ast}}{\lra} & j_1^\ast\bigl(f_\ast
\cO_{\bV(\cE)} \bigr ) \cr \downarrow & & \downarrow \cr
j_2^\ast\bigl(f_{0,\ast} \cO_{\bV_0(\cE, s_1,\ldots,s_m)} \bigr ) &
\stackrel{\epsilon^\ast_0}{\lra} &  j_1^\ast\bigl(f_{0,\ast}
\cO_{\bV_0(\cE, s_1,\ldots,s_m)} \bigr )\cr
\end{matrix}$$
such that $\Delta^\ast(\epsilon^\ast_0)={\rm Id}$ and
$\Delta^\ast(\epsilon^{',\ast})={\rm Id}$. By construction
$\epsilon^{',\ast}$ coincides with the isomorphism $\epsilon$,
once restricted to the $\cO_S$-submodule $\cE\subset f_\ast
\cO_{\bV(\cE)}$, and is uniquely characterized by this property as
$f_\ast \cO_{\bV(\cE)}$ is the $\cI$-adic completion of the
symmetric algebra defined by $\cE$. Since the vertical maps are
obtained via a blowup by Lemma \ref{lem:V0} the commutativity of
the diagram above uniquely characterizes $\epsilon^\ast_0$. In
particular it satisfies the cocyle condition as
$\epsilon^{',\ast}$ does since $\epsilon$ does. Via Grothendieck's
correspondence this defines the sought for, compatible, integrable
connections:

$$\begin{matrix}
\cE & \stackrel{\nabla}{\lra} &
\cE\otimes_{\cO_S}\Omega^1_{S/A_0}\cr \downarrow & & \downarrow
\cr f_\ast \cO_{\bV(\cE)} & \stackrel{\nabla'}{\lra} & f_\ast
\cO_{\bV(\cE)} \widehat{\otimes}_{\cO_S} \Omega^1_{S/A_0}\cr
\downarrow & & \downarrow \cr f_{0,\ast} \cO_{\bV_0(\cE,
s_1,\ldots,s_m)} & \stackrel{\nabla_0}{\lra} & f_{0,\ast}
\cO_{\bV_0(\cE, s_1,\ldots,s_m)} \widehat{\otimes}_{\cO_S}
\Omega^1_{S/A_0}\cr .\end{matrix} $$ As remarked above both
$\nabla'$ and $\nabla_0$ are the unique connections that make the
diagram above commutative, i.e., compatible with $\nabla$.

\

Assume that we are in the hypothesis of \S\ref{sec:vbfil} with
locally free $\cO_S$-module and direct summand $\cF\subset \cE$.
Consider the filtrations $\Fil_\bullet f_{\ast} \cO_{\bV(\cE)}$
and $\Fil_\bullet f_{0,\ast} \cO_{\bV_0(\cE, s_1,\ldots,s_m)}$ of
Corollary \ref{cor:fil}.

\begin{lemma}\label{lem:Griffiths}

The connection $\nabla_0$ satisfies Griffith's transversality
property with respect to the filtration  $\Fil_\bullet f_{0,\ast}
\cO_{\bV_0(\cE, s_1,\ldots,s_m)}$, namely for every integer $h$ we
have
$$ \nabla\bigl( \Fil_h f_{0, \ast} \cO_{\bV_0(\cE, s_1,\ldots,s_m)}\bigr)
\subset \Fil_{h+1} f_{0,\ast} \cO_{\bV_0(\cE, s_1,\ldots,s_m)}
\widehat{\otimes}_{\cO_S} \Omega^1_{S/A_0}.$$Furthermore the
induced map
$$\mathrm{gr}_h(\nabla_0)\colon \mathrm{Gr}_h f_{0, \ast} \cO_{\bV_0(\cE, s_1,\ldots,s_m)} \lra \mathrm{Gr}_{h+1} f_{0,\ast}
\cO_{\bV_0(\cE, s_1,\ldots,s_m)} \widehat{\otimes}_{\cO_S}
\Omega^1_{S/A_0}$$is an $\cO_S$-linear map and, via the
identification $\mathrm{Gr}_\bullet f_{\ast} \cO_{\bV_0(\cE,
s_1,\ldots,s_m)}\cong f_{0,\ast} \cO_{\bV_0(\cF,
s_1,\ldots,s_m)}\otimes_{\cO_S} \mathrm{Sym}^\bullet(\cE/\cF)$  of
Corollary \ref{cor:fil}, the morphism
$\mathrm{gr}_\bullet(\nabla_0)$ is
$\mathrm{Sym}^\bullet(\cE/\cF)$-linear.

\end{lemma}
\begin{proof} The statement can be checked locally. Assume that $U=\Spf(R)\subset S$
is an open formal affine subscheme such that $\cI\vert_U$ is
generated by $\alpha\in R$, the sheaves $\cF$, $\cE$ over $U$ are free with
basis $e_1,\ldots,e_m$, resp.~$e_1,\ldots,e_m,
f_1,\ldots,f_{n-m}$, so that $$\bV(\cF)\vert_U=\Spf\bigl(R\langle
X_1,\ldots,X_m\rangle\bigr), \quad  \bV(\cE)\vert_U=\Spf\bigl(R\langle X_1,\ldots,X_m,Y_1,\ldots,
Y_{n-m}\rangle\bigr)$$and
$$\bV_0(\cF,
s_1,\ldots,s_m)\vert_U=\Spf\bigl(R\langle
Z_1,\ldots,Z_m\rangle\bigr) , \, \bV_0(\cE,
s_1,\ldots,s_m)\vert_U=\Spf\bigl(R\langle
Z_1,\ldots,Z_m,Y_1,\ldots, Y_{n-m}\rangle\bigr).$$By construction
$\nabla'(X_s) =\sum_{i=1}^m \alpha X_i \otimes \omega_{s,i} +
\sum_{j=1}^{n-m} \alpha Y_j \otimes \beta_{s,j}$ where the
elements $\omega_{s,i}$ and  $\beta_{s,j}$ are uniquely
characterized by the fact that $\nabla(e_s)=\sum_{i=1}^m \alpha e_i
\otimes \omega_{s,i} + \sum_{j=1}^{n-m} \alpha f_j \otimes
\beta_{s,j}$ (recall that $\nabla(e_s)\equiv 0$ modulo $\alpha$
for $s=1,\ldots,m$). Similarly $\nabla'(Y_t) =\sum_{i=1}^m X_i
\otimes \gamma_{t,i} + \sum_{j=1}^{n-m} Y_j \otimes \delta_{s,j}$
where $\nabla(f_t)=\sum_{i=1}^m e_i \otimes \gamma_{s,i} +
\sum_{j=1}^{n-m} f_j \otimes \delta_{s,j}$.

Since $X_i=1+\alpha Z_i$ then $\nabla_0(\alpha Z_i)=\nabla'(X_i)$
and we deduce that $\nabla_0(Z_s) =\sum_{i=1}^m X_i \otimes
\omega_{s,i} + \sum_{j=1}^{n-m}  Y_j \otimes \beta_{s,j} - Z_i
\otimes d\alpha$. Recall from Corollary \ref{cor:fil} that $\Fil_h
f_{0,\ast} \cO_{\bV_0(\cE)}(U)$ consists of the polynomials of
degree $\leq h$ in the variables $Y_1,\ldots, Y_{n-m}$ with
coefficients in $R\langle Z_1,\ldots,Z_m\rangle$. The fact that
Griffith's transversality holds for $\Fil_h f_{0,\ast}
\cO_{\bV_0(\cE, s_1,\ldots,s_m)}(U)$ follows from the explicit
expression of $\nabla$  and Leibniz' rule.

\end{proof}

\section{Applications to modular curves.}\label{sec:appmodular}

In this chapter we present applications of the main constructions in Section \ref{sec:FVBMS}, that is to say given a weight $k$ we present a new construction of the
modular sheaves $\fw^k$ already defined and studied in \cite{halo_spectral} and the construction of a modular sheaf $\bW_k$ interpolating the integral symmetric
powers of the sheaf of relative de Rham cohomology of the universal elliptic curve over the appropriate modification of a modular curve.

The sheaf $\bW_k$ has a natural filtration whose graded quotients are well understood, an integrable connection $\nabla_k$ which satisfies the Griffith
transversality property with respect to the filtration and a natural action of the Hecke algebra on its global sections such that the operator $U$ is compact.
Moreover, the global sections of $\bW_k$ have  natural $q$-expansions which allows one, as in the case of $p$-adic modualr forms, to define {\bf nearly overconvergent $p$-adic modular
forms} as formal $q$-expansions arising from sections of $\bW_k$.

\subsection{ The sheaves $\fw_I$.}\label{sec:omega}

{\bf Convention}. In what follows we will denote by $X,Y,Z,\ldots$ (algebraic) schemes, by $\frak{X},\frak{Y}, \frak{Z},\ldots$ formal schemes and by $\mathcal{X},
\mathcal{Y}, \mathcal{Z},\ldots$ analytic, adic spaces.

\bigskip
\noindent In this section we follow the constructions of \cite{halo_spectral}. Let $N\ge 4$ be an integer and $p$ a prime which does not divide $N$. Let $Y:=X_1(N)$
be the smooth, proper modular curve over $\Z_p$ which classifies generalized elliptic curves with $\Gamma_1(N)$-level structure and let $\mathfrak{Y}$ denote the
formal completion of $Y$ along its special fiber. We write $E\to Y$ for the universal semiabelian scheme and $\omega_E$ for its invariant differentials; away from
the cusps $E$ is the universal elliptic curve. We denote by $\mathrm{Hdg}_{\mathfrak{Y}}$ the ideal of $\cO_{\mathfrak{Y}}$ defined locally by: if $U={\rm Spf}(R)$
is an open affine of $\fY$ such that $\omega_E\vert_{U}$ is a free $R$-module of rank $1$, then $\mathrm{Hdg}\vert_U$ is generated by $p$ and by the value
$\widetilde{\mathrm{Ha}}(E/R, \omega)$ of a lift $\widetilde{\mathrm{Ha}}$ of the Hasse invariant modulo $p$, where $\omega$ is any $R$-generator of
$\omega_E\vert_U$. Note that for $p\geq 5$ one can take $\widetilde{\mathrm{Ha}}=E_{p-1}$, the Eisenstein series of weight $p-1$.

\

{\it The weight space:} \enspace Set $\Lambda$ to be the Iwasawa algebra $\Z_p[\![\Z_p^\ast]\!]\cong \Z_p\bigl[(\Z/q\Z)^\ast\bigr][\![T]\!]$, where $q=p$ if $p\geq 3$ and $q=4$ if $p=2$ and the
isomorphism is defined by sending $\displaystyle \exp(q):=\sum_{n=0}^\infty \frac{q^n}{n!}\in 1+q\Z_p$ to $1+T$. Consider the complete local ring $\Lambda^0=\Z_p[\![T]\!]\subset \Lambda$.

Let $\mathfrak{W}:={\rm Spf}(\Lambda)$, respectively
$\mathfrak{W}^0:={\rm Spf}(\Lambda^0)$, where the ideal of
definition of these formal schemes is $\mathfrak{m}:=(p,T)$ and let
us denote by $\cW:=\Bigl({\rm Spa}(\Lambda, \Lambda)\Bigr)^{\rm
an}$ (and similarly for $\cW^0$) the associated analytic adic weight space. Here the superscript ``$an$" stands for {\it analytic}, i.e., $\cW$ is the adic subspace consisting of the analytic points of the adic space associated
to the formal scheme $\mathfrak{W}$. For every closed interval
$I:=[p^a, p^b]\subset [0, \infty]$, with $a\in \N\cup
\{-\infty\}$ and $b\in \N \cup \{\infty\}$, we denote by $$\cW_I=\{x\in \cW\quad \vert \quad \vert p\vert _x\le \vert T^{p^a}\vert _x\neq 0\mbox{ and } \vert T^{p^b}\vert _x\le \vert p\vert_x\neq 0\}.$$These are rational open subsets  and we have two notable cases:\smallskip

(1) $\cW_I=\{x\in \cW\quad \vert \quad \vert T^{p^b}\vert _x\le \vert p\vert _x\neq 0\}$ with $I=[0,p^b]$ and $b\neq \infty$;
\smallskip

(2) $ \cW_I=\{x\in \cW\quad \vert \quad \vert p\vert _x\le \vert T^{p^a}\vert _x\neq 0\mbox{ and } \vert T^{p^b}\vert _x\le \vert p\vert_x\neq 0\}$ if $
a\neq -\infty$ and $a\leq b$.
\smallskip

In the first case $$\displaystyle \cW_I={\rm Spa}\Bigl(\Lambda\langle \frac{T^{p^b}}{p} \rangle\Bigl[\frac{1}{p}\Bigr], \Lambda\langle \frac{T^{p^b}}{p} \rangle  \Bigr)$$ and in the
second $$\displaystyle \cW_I={\rm Spa}\Bigl( \Lambda\langle \frac{p}{T^{p^a}}, \frac{T^{p^b}}{p}\rangle \Bigl[\frac{1}{T}\Bigr], \Lambda\langle \frac{p}{T^{p^a}},
\frac{T^{p^b}}{p}\rangle  \Bigr).$$Let us remark that for every $I\subset  [0,\infty)$ as above $\cW_I$ is an open adic subspace of  $\cW_{[0,\infty)}=\cW^{\rm rig}$

For each $I=[p^a, p^b]$ as above we let $k_I\colon\Z_p^\ast \lra \bigl(\cO_{\cW_I}^+\bigr)^\ast$ denote the universal character associated to $\cW_I$.
Let now $\mathfrak{X}:=\mathfrak{Y}\times_{\rm
Spf(\Z_p)}\mathfrak{W}^0$. We define $\mathfrak{X}_I=
\mathfrak{Y}\times_{\rm Spf(\Z_p)}{\rm Spf}(\cO_{\cW_I^0}^+)$.

We consider pairs $(\Lambda_I, \alpha)$ where $\displaystyle
\Lambda_I:=\Lambda\langle \frac{T}{p}\rangle$ and
$\alpha:=p\in \Lambda_I$ if $I$ is in case (1) and $\displaystyle
\Lambda_I:=\Lambda\langle \frac{p}{T^{p^a}},
\frac{T^{p^b}}{p}\rangle$ and $\alpha:=T\in \Lambda_I$ if $I$ is
in case (2).

\

{\it Formal admissible partial blow-ups of  modular curves:}\enspace  We continue using the notations above  and for every integer $r\ge 1$ we define $\mathfrak{X}_{r,I}$ to be
the formal scheme over $\mathfrak{X}_I$ which represents the
functor  associating to every $\Lambda_I^0$-algebra
$\alpha$-adically complete $R$ the set of equivalence classes of
pairs $(f, \eta)$, where $f\colon {\rm Spf}(R)\lra \mathfrak{X}_I$
and $\eta\in {\rm H}^0\bigl({\rm Spf}(R),
f^\ast(\omega^{(1-p)p^{r+1}})\bigr)$ such that
$$
\eta\cdot \widetilde{\mathrm{Ha}}^{p^{r+1}}=\alpha(\mbox{ mod
}p^2).
$$
Here $\widetilde{\mathrm{Ha}}$ denotes any lift of the Hasse
invariant. One sees that the definition is well posed, i.e., it
does not depend on the choice of the lift. Moreover the ideal of $R$
denoted $\mathrm{Hdg}_R$ at the beginning of this section becomes
invertible. See section \S 3.1 of \cite{halo_spectral} for the
proof of this fact and for the definition of the equivalence relation.
By abuse of notation we often write $\mathrm{Hdg}_R$ for a (local)
generator of this ideal as well.

Let us remark that if $I$ is in the case (1), i.e. $I=[0, p^b]$ then $\displaystyle \frac{p}{\mathrm{Hdg}^{p^{r+1}}}\in \cO_{\mathfrak{X}_{r,I}}$ and if $I$ is in
case (2), i.e., $I=[p^a, p^b]$ with $0\le a< b\leq \infty$ then $\displaystyle \frac{p}{\mathrm{Hdg}^{p^{a+r+1}}}=\frac{p}{T^{p^a}}\cdot
\frac{T^{p^a}}{\mathrm{Hdg}^{p^{a+r+1}}}\in \cO_{\mathfrak{X}_{r,I}}$. Therefore if we denote by $n$ any integer with $1\le n\le r$ if $I$ is in the case (1) and
$1\le n\le a+r$ if $I$ is in the case (2), then the semiabelian scheme $E\lra \mathfrak{X}_{r,I}$ has a canonical subgroup $H_n$ of order $p^n$. This is  a subgroup scheme of order $p^n$ lifting the kernel of the $n$-th power of the Frobenius isogeny modulo $\frac{p}{\mathrm{Hdg}^{p^{\frac{p^n-1}{p-1}}}}$  (see \cite[Cor. A.1 \& A.2]{halo_spectral} for the construction). Over the ordinary locus $H_n$ is the connected part of the $p^n$-th torsion of $E$.

\

{\it  Partial Igusa tower:} \enspace For every $r$, $I$, $n$ as above we denote by $\mathcal{X}_{r,I}$ the adic generic fiber of the formal scheme $\mathfrak{X}_{r,I}$ and let $\mathcal{IG}_{n,r,I}\lra
\mathcal{X}_{r,I}$ denote the adic space of trivializations of the group scheme $H_n^\vee\lra \mathcal{X}_{r,I}$, the Cartier dual of $H_n$. Then
$\mathcal{IG}_{n,r,I}\lra \mathcal{X}_{r,I}$ is a finite \'etale and Galois morphism of adic spaces with Galois group $(\Z/p^n\Z)^\ast$. We define by
$\mathfrak{IG}_{n,r,I}\lra \mathfrak{X}_{r,I}$ the normalization of $\mathfrak{X}_{r,I}$ in $\mathcal{IG}_{n,r,I}$, which is well defined. Moreover the morphism
$\mathfrak{IG}_{n,r,I}\lra \mathfrak{X}_{r,I}$ is finite and is endowed with an action of $(\Z/p^n\Z)^\ast$.

\

{\it The construction of the torsor $\mathfrak{F}_{n,r,I}$:}\enspace In \cite[\S 5.2]{halo_spectral} we define the formal scheme $f_n\colon \mathfrak{F}_{n,r,I} \to \mathfrak{IG}_{n,r,I}$. It represents the functor from the category of
affine formal schemes $\Spf(R)\to \mathfrak{IG}_{n,r,I}$, with $R$ an $\alpha$-adically complete and separated $\Z_p$-algebra without $ \alpha$-torsion, to the category of sets
$$
\mathfrak{F}_{n,r,I}(R)=\{(\omega,P)\in \omega_E(R)\times \bigl(H_n^\vee(R)-H_n^\vee[p^{n-1}](R)\bigr)\quad \vert \quad \omega={\rm dlog}(P)\mbox{ in
}\omega_E/p^n\mathrm{Hdg}^{-\frac{p^n-1}{p-1}} \}.
$$

We also denote by $h_n\colon \mathfrak{F}_{n,r,I}\lra \fX_{r,I}$ the composition $g_n\circ f_n$. We have a natural action of
$\Z_p^\ast\Bigl(1+p^n\mathrm{Hdg}^{-\frac{p^n-1}{p-1}}\mathbb{G}_a\Bigr)$ on $\mathfrak{F}_{n,r,I}$ given by: if $\lambda\in \Z_p^\ast$ and $x\in
\Bigl(1+p^n\mathrm{Hdg}^{-\frac{p^n-1}{p-1}}\mathbb{G}_a\Bigr)$, then  $(\lambda x)(\omega, P)=\bigl((\lambda x)\omega, \lambda P)$. This action is well defined and
the action of $\Z_p^\ast$ lifts the Galois action of $(\Z/p^n\Z)^\ast$ on $\mathfrak{IG}_{n,r,I}$. In fact $\mathfrak{F}_{n,r,I}$ admits an action of
$\Z_p^\ast\Bigl(1+p^n\mathrm{Hdg}^{-\frac{p^n-1}{p-1}}\mathbb{G}_a\Bigr)$ (in the \'etale topology) over $\fX_{r,I}$ with quotient $\fX_{r,I}$. Furthermore if $n\geq b+2$ for $p\neq 2$ or $n\geq b+4$ if $p=2$ then $k_I$ extends to a character $\Z_p^\ast\Bigl(1+p^n\mathrm{Hdg}^{-\frac{p^n-1}{p-1}}\mathbb{G}_a\Bigr)\to \mathbb{G}_m$.

In conclusion, given $r$, $I$, $n$ as above we have (see \cite{halo_spectral}) a sequence of formal schemes and morphisms
$$
\mathfrak{F}_{n,r,I}\stackrel{f_n}{\lra}\mathfrak{IG}_{n,r,I}\stackrel{g_n}{\lra}\mathfrak{X}_{r,I},
$$
which leads to the following definition. We summarize the various assumptions on $I$, $n$ and $r$ in the following two cases:

\smallskip

(1) $I=[0,1]$, $r\geq 2$ if $p\neq 2$ or $r\geq 4$ if $p=2$ and $n$ is an integer $n$ satisfying $1\leq n \leq r$. 

\smallskip

(2)  $I=[p^a,p^b]$ with $a,b\in \N$,  $r\geq 1$ and $r+a\geq b+2$ if $p\neq 2$ or $r\geq 2$ and $r+a\geq b+4$ if $p=2$ and $n$ is an integer  such that $1\le n \leq a+r$.

\begin{definition}\label{defi:wnrI}
Let $k_{I,f}$ be the character given by the restriction of the character $k$ to $(\Z/q\Z)^\ast\subset \Z_p^\ast$. Define $\fw^{k_{I,f}}$ to be the coherent
$\cO_{\fX_{r,I}}$-module $\bigl(g_{i,\ast} \bigl(\cO_{\mathfrak{IG}_{i,r,I}}\bigr)\otimes_{\Lambda^0} \Lambda\bigr) \bigl[k_{I,f}^{-1}\bigr] $ (see \cite[\S
6.8]{halo_spectral}). Here $i=1$ for $p$ odd and $i=2$ for $p=2$.

Let  $k_I^0:=k_I k_{I,f}^{-1}\colon \Z_p^\ast\to (\Lambda^0_I)^\ast$. Set $\fw_{n,r,I}^{k_I,0}:=\bigl((g_n\circ
f_n)_\ast\bigl(\cO_{\mathfrak{F}_{n,r,I}}\bigr)\bigr)[(k_I^0)^{-1}]$, i.e., the sheaf on $\mathfrak{X}_{r,I}$ of sections of $\cO_{\mathfrak{F}_{n,r,I}}$ which
transform under the action of $\Z_p^\ast\bigl(1+p^n\mathrm{Hdg}^{-\frac{p^n-1}{p-1}}\mathbb{G}_a\bigr)$ via the inverse of the universal character $k_I^0$. Define
$\fw_{n,r,I}^{k_I}:=\fw_{n,r,I}^{k_I^0}\otimes_{\cO_{\fX_{r,I}}} \fw^{k_f}$.

\end{definition}

{\it Overconvergent modular forms:} \enspace It is proved in section \cite[\S 5.3.2]{halo_spectral} that, under the assumptions in Definition \ref{defi:wnrI}, $\fw_{n,r,I}^{k_I,0}$ is an invertible sheaf on $\mathfrak{X}_{r,I}$ with the following property: for
every interval $I$ as above, there exist $r_I$, $n_I$ such that for all $r\ge r_I$, $n\ge n_I$ (satisfying the relations at the beginning of this section)
$\varphi_{r,r_I}^\ast\bigl(\fw_{n_I,r_I,I}^{k_{I},0} \bigr)\cong \fw_{n,r,I}^{k_I,0}$ as $\cO_{\mathfrak{X}_{r,I}}$-modules, where $\varphi_{r,r_I}\colon
\mathfrak{X}_{r,I}\lra \mathfrak{X}_{r_I,I}$ is the natural projection.

Therefore we denote $\fw_{n_I,r_I,I}^{k_{I},0}$ and $\fw_{n_I,r_I,I}^{k_{r_I}}$  by $\fw_I^{k_I,0}$ and  $\fw_I^{k_I}$ respectively and call them {\sl modular
sheaves}. Note that $\fw_I^{k_I}$ defines  an invertible sheaf, denoted $\omega^k_{I}$ in \cite{halo_spectral}, on the adic space $\mathcal{X}_{r_I,I}$, whose
global sections are {\sl the overconvergent $p$-adic families} of modular forms over $\mathcal{W}_I$.

\subsubsection{Some properties of $ \fIG_{n,r,I}$.}\label{sec:localcoo}

Consider the natural morphisms of formal schemes $$\eta \colon \fIG_{n,r,I}\stackrel{\gamma_n}{\lra}\fIG_{1,r,I}\lra \fX_{r,I}\lra \fX_I \lra \fX.$$Denote by $j\colon \fX^{\rm ord}_I \hookrightarrow \fX_{r,I} $ the $\alpha$-adic open formal sub-scheme of $\fX_{r,I}$ defined by the ordinary locus. Let $\iota
\colon \fIG_{n,r,I}^{\rm ord}\subset \fIG_{n,r,I}$ be the inverse image of $\fX^{\rm ord}_I$.
 We recall the following:

\begin{remark}\label{rmk:delta}
Fix a local lift $\widetilde{\mathrm{Ha}}$ of the Hasse invariant over an open formal subscheme $U=\Spf(R)$ of $\fIG_{1,r,I}$. For $p\geq 5$ one can take a global lift, namely $E_{p-1}$. There exists
a unique section of $\omega_E$ over $R$ denoted $\Delta$ such that its $q$-expansion has constant coefficient $1$ modulo $p$ and
$\Delta^{p-1}=\widetilde{\mathrm{Ha}}$; see \cite[Prop.~A.3]{halo_spectral}. We define the ideal $\udelta$  of $\cO_{\mathfrak{IG}_{1,r,I}}$ to be the ideal sheaf
which is generated locally by the functions $\delta_R:=\Delta(E/R, \omega)$'s, where $\omega$ is an $R$-basis of $\omega_{E/R}$. It coincides with the ideal
$\rm{Hdg}^{\frac{1}{p-1}}$, where ${\rm Hdg}$ is the ideal of \S\ref{sec:pdivconstructions}.
\end{remark}

\

\noindent
 We have the following result:

\begin{lemma}\label{lemma:CokerOmega} The induced map $\eta^\ast\bigl(\Omega^1_{\fX/\Z_p}\bigr)\lra \Omega^1_{\fIG_{n,r,I}/\Lambda_I^0}$
has kernel and cokernel annihilated by a power of $\udelta$ and in particular by a power of $\alpha$, depending on $n$.
\end{lemma}
\begin{proof} The morphism $\fX_{r,I}\lra \fX_I$ is an isomorphism over the ordinary locus. The morphism $\fIG_{n,r,I}^{\rm ord}\to \fX^{\rm ord}_I$ is the Igusa tower classifying trivializations of the \'etale group scheme  $H_n^\vee$. In particular, it is \'etale and Galois with group $(\Z/p^n\Z)^\ast$. Thus the induced map on differentials is an isomorphism. The ordinary locus is defined, modulo $\alpha$, by inverting $\udelta$. This implies the lemma for the differentials modulo $\alpha$.  As $\Omega^1_{\fX/\Z_p}$ and $\Omega^1_{\fIG_{n,r,I}/\Lambda_I^0} $ are coherent $\cO_{\fIG_{n,r,I}/\Lambda_I^0}$-modules and $\alpha\in \udelta$, the claim follows. \end{proof}

We next prove the following:

\begin{lemma}\label{lemma:kerloc}  For every $h\in\N$ the kernel of the map
$\cO_{\fX_{r,I}}/\alpha^h \cO_{\fX_{r,I}} \to j_\ast \bigl(\cO_{\fX^{\rm ord}}/\alpha^h \cO_{\fX^{\rm ord}} \bigr)$ is annihilated by $\mathrm{Hdg}^{h p^{r+1}}$.
Similarly, the kernel of $\cO_{\fIG_{n,r,I}}/\alpha^h \cO_{\fIG_{n,r,I}} \to \iota_\ast \bigl(\cO_{\fIG_{n,r,I}^{\rm ord}}/\alpha^h \cO_{\fIG_{n,r,I}^{\rm ord}}
\bigr)$ is annihilated by $\udelta^{h p^{r+1} (p-1) + p^n-p} $.
\end{lemma}
\begin{proof} We prove the first statement. We'll work locally so let $U=\Spf(R')\subset \fX:=\fX_1(N)$ be an affine open such that $\omega_E\vert _U$ is free of rank one and we choose a basis $\omega$ of
$\omega_E\vert_U$. If we denote $x:=\widetilde{\mathrm{Ha}}(E/R', \omega)\in R'$, then $\Z_p\langle x\rangle \subset R'$ is an \'etale extension after possibly shrinking $U$, i.e.,  in a small enough neighborhood of the supersingular points,  as Igusa proved that
$x$ has simple zeroes on $R'/pR'$ exactly at the supersingular points in $U$. Now let $U_I=\Spf(R)$ with $R=R'\widehat{\otimes}_{\Z_p} \Lambda_I^0$ the inverse image of $U$ in $\fX_I$ and let  $V=\Spf(A)$ be the inverse image of $U$ via  the morphsim $ \fX_{r,I}\lra \fX$. As $\fX_{r,I}$ is the open of the blow-up of $\fX_I$ along the ideal $\bigl(p, x^{p^{r+1}}\bigr)$ where this ideal is generated
by $x^{p^{r+1}}$, we have $A=R\langle y\rangle/\bigl(x^{p^{r+1}}y-\alpha \bigr)$.

We have that $U^{\rm ord}_I=\Spf(R^{\rm
ord})$ and $V^{\rm ord}=\Spf(A^{\rm ord})$ where  $R^{\rm
ord}=R\langle 1/x\rangle$ and $A^{\rm ord}=A\widehat{\otimes}_R
R\langle 1/x\rangle$. For every $h$ the morphism $R/\alpha^h R \to
R^{\rm ord}/\alpha^h R^{\rm ord}$ is injective since the special fiber
of $X_1(N)$ at $p$ is irreducible.

We claim that the kernel of $A/\alpha^h A \to A^{\rm ord}/\alpha^h A^{\rm
ord}$ is annihilated by $x^{hp^{r+1}}$. This is equivalent to proving the first statement of the Lemma. Notice that, as $\Z_p\langle x\rangle \subset R'$ is \'etale, then
$A_0:=\Lambda^0_I\langle X,Y \rangle/\bigl(X^{p^{r+1}}Y-\alpha
\bigr)\to A$ sending $X\mapsto x$ ad $Y\mapsto y$ is \'etale. In
particular it is flat and, hence, it suffices to prove the
statement replacing $A$ with $A_0$ and $A^{\rm ord}$ with
$\Lambda^0_I\langle X, X^{-1}\rangle$.  The kernel $I_h$ of the map
$$A_0/\alpha^h A_0= \Lambda^0_I/\alpha^h \Lambda^0_I [X,Y]
/\bigl(X^{p^{r+1}}Y-\alpha \bigr) \lra \Lambda^0_I/\alpha^h
\Lambda^0_I[X,X^{-1}]$$is $I_h=\bigl(Y^h,\alpha
Y^{h-1},\ldots,\alpha^{h-1} Y\bigr)$ which is principal and
generated by $Y^h$ since $\alpha=Y X^{p^{r+1}}$ so that $\alpha^j
Y^{h-j}=Y^h X^{jp^{r+1}}$ for every $1\leq j\leq h$. Since $X^{h
p^{r+1}} Y^h=\alpha^h$, the claim follows.

We consider now the morphism $\fIG_{1,r,I}\to \fX_{r,I} $ and let $\Spf(C)$ be the inverse image of  $V$. Since $x$ admits a $p-1$-th root in $C$ (see Remark
\ref{rmk:delta}) then $C$ is the normalization of $A[z]/(z^{p-1}- x)$. Since $x$ has simple poles modulo $p$ then $R''=R[z]/(z^{p-1}-x)$ is normal and $A[z]/(z^{p-1}-
x)= R'' \langle y\rangle/\bigl(z^{p^{r+1}}y-\alpha \bigr)$ is already normal (cf.~\cite[Lemme 3.4]{halo_spectral}) and, hence, equal to $C$. We conclude that $C$ is
flat over $A$ and the second statement of the Lemma for $n=1$ follows.

From the proof of Proposition 3.5 of \cite{halo_spectral} it follows that there is  a natural morphism $ \fIG_{n,r,I} \to H_n^\vee$ and that $ \fIG_{n,r,I}$ is the
normalization of $\fIG_{n,r,I}':=H_n^\vee\times_{H_1^\vee} \fIG_{1,r,I} $. Note that $\fIG_{n,r,I}' \to   \fIG_{1,r,I}$ is flat so that the kernel of
$\cO_{\fIG_{n,r,I}'}/\alpha^h \cO_{\fIG_{n,r,I}'} \to \iota_\ast \bigl(\cO_{\fIG_{n,r,I}^{\rm ord}}/\alpha^h \cO_{\fIG_{n,r,I}^{\rm ord}} \bigr)$ is annihilated by
$\udelta^{h p^{r+1} (p-1)} $. Again from  \cite[Prop.~3.5]{halo_spectral} we know that the different ideal  $\mathcal{D}(H_n^\vee/H_1^\vee)$  of $H_n^\vee$ over $H_1^\vee$ is
such that $\udelta^{p^n-p}\subset \mathcal{D}(H_n^\vee/H_1^\vee)$. We conclude that the same must be true for the different ideal
$\mathcal{D}(\fIG_{n,r,I}'/\fIG_{r,1,I}')$ of $\fIG_{n,r,I}'$ over $\fIG_{r,1,I}'$, i.e., $\udelta^{p^n-p}\subset \mathcal{D}(\fIG_{n,r,I}'/\fIG_{1,r,I}')$. Since $
\fIG_{n,r,I}\to \fIG_{n,r,I}'$ is defined by taking the normalization, it is finite and  we conclude that $\udelta^{p^n-p} \cO_{\fIG_{n,r,I}} \subset
\cO_{\fIG_{n,r,I}'}$ and the second statement of the Lemma for $n\geq 2$ follows.

\end{proof}

\subsection{A new definition of $\fw^{k,0}.$}
\label{sec:new}

We'd now like to use the theory  of Section \ref{sec:FVBMS}, i.e.,  the vector bundles with marked sections in order to give a new definition of the sheaves $\fw^k$ defined
in \cite{halo_spectral} and recalled in Section \ref{sec:omega} of this article.

We choose $I$, $r$, $n$ satisfying the assumptions of Definition \ref{defi:wnrI} and $k=k_I\colon \Z_p^\ast\lra \Lambda_I^\ast$ the universal character and $k_I^0=k_I
k_{I,f}^{-1}\colon \Z_p^\ast \to (\Lambda^0_I)^\ast$ its ``restriction" to $\Lambda^0_I$. There exists  a unique element $u=u_k\in p^{1-n} \Lambda_I^0$ such that
$t^k:=k(t)=\exp\bigl(u\log(t)\bigr)$ for all $t\in 1+p^n\Z_p^\ast$. Let $E\lra \fIG_{n,r,I}$ be the semiabelian scheme over the level $n$-th Igusa curve.

It would be natural to use as a pair $(\cE,s)$ consisting of a locally free sheaf $\cE$ of rank one and marked section $s$, the $\cO_{\fIG_{n,r,I}}$-module
$\omega_E$ and $s={\rm dlog}(P_n)$ seen as a section of  $\omega_E/\underline{\beta}_n\omega_E$ via the following diagram  (see the Section
\ref{sec:pdivconstructions}).

\begin{equation}\label{eq:dlog}
\begin{array}{cccccccc}
&&\omega_E\\
&&\downarrow\\
H_n^\vee&\stackrel{d\log}{\lra}&\omega_{H_n}\\
&&\downarrow\\
&&\omega_E /\underline{\beta}_n\omega_E
\end{array}
\end{equation}

Here,  $\ubeta_n=p^n\mathrm{Hdg}(E)^{-\frac{(p^n-1)}{p-1}}$ and  $P_n$ is the universal generator of $H_n^\vee$ over $\fIG_{n,r,I}$. Unfortunately the
pair $(\omega_E, {\rm dlog}(P_n))$ does not satisfy the conditions of  Section \ref{sec:vbfs} because the cokernel of the inclusion ${\rm
dlog}(P_n)\bigl(\cO_{\fIG_{n,r,I}}/\ubeta_n\bigr)\hookrightarrow \omega_E/\ubeta_n\omega_E$ is precisely annihilated by the ideal $\udelta$. Therefore one of $\omega_E$ or
${\rm dlog}(P_n)$ must be modified.

\begin{definition} \label{def:bigomega}  We denote by $\Omega_E$ the $\cO_{\fIG_{n,r,I}}$-submodule
of $\omega_E$ generated by all the lifts of ${\rm dlog}(P_n)$.
\end{definition}
Recall the following properties (see Section \ref{sec:pdivconstructions}):

\begin{itemize}

\item[a)] $\Omega_E$ is a locally free $\cO_{\fIG_{n,r,I}}$-module of rank $1$.

\item[b)] The map ${\rm dlog}$ defines an isomorphism:
$$
H_n^\vee\bigl(\fIG_{n,r,I}\bigr)\otimes_{\Z}\cO_{\fIG_{n,r,I}}/\ubeta_n\cong \Omega_E/\ubeta_n\Omega_E.
$$

\end{itemize}
In particular it follows that the pair $(\Omega_E, s)={\rm dlog}(P_n))$ is a locally free sheaf with a marked section. Concerning (a) we have
$\Omega_E=\udelta\omega_E$ (with the notation of Remark \ref{rmk:delta}). In particular for $p\geq 5$ the sheaf $\Omega_E$ is a {\it free}
$\cO_{\fIG_{n,r,I}}$-module of rank one (instead of only a locally free one). We now apply the theory of Section \ref{sec:vbfs} to the pair $(\Omega_E,s)$ and we have the morphisms of formal schemes
$$
\bV_0(\Omega_E,s)\stackrel{\pi}{\lra}\fIG_{n,r,I}\stackrel{h_n}{\lra}\fX_{r,I},
$$
and we denote by $f_0:=h_n\circ \pi\colon \bV_0(\Omega_E,s)\lra \fX_{r,I}$.

\medskip
\noindent
Denote by $\fT\subset \fT^{\rm ext}$ the formal groups over $\fX_{r,I}$ defined on points by: if $\rho\colon S\lra \fX_{r,I}$ is a morphism of formal schemes, we
let $\fT(S):=1+\rho^\ast(\ubeta)_n \cO_S\subset \fT^{\rm ext}(S):=\Z_p^\ast\bigl(1+\rho^\ast(\ubeta_n) \cO_S\bigr)\subset \bG_{m,S}$. We only need to remark that
$\ubeta_n=p^n{\rm Hdg}^{\frac{p^n-1}{p-1}}$, which was so far used as an ideal of $\cO_{\fIG_{n,r,I}}$  is in fact an ideal of $\cO_{\fX_{r,I}}$.

We have natural actions of $\fT$ and respectively $\fT^{\rm ext}$ on $\bV_0(\Omega_E, s)$ over $\fIG_{n,r,I}$ and respectively
$\fX_{r,I}$, defined on points as follows:
\smallskip

(1) Let $\rho\colon S\lra \fIG_{n,r,I}$ be a morphism of formal schemes and let $t$ be an element of $\fT(S)$ and $v$ a point in $\bV_0(\Omega_E, s)(S)$. Let us
recall that $v\colon\rho^\ast(\Omega_E)\lra \cO_S$ is an $\cO_S$-linear map such that if denote by $\overline{v}:=v\bigl(\mbox{mod
}\rho^\ast(\ubeta_n\Omega_E)\bigr)$, then $\overline{v}(s)=1$. We define the action of $\fT(S)$ on $\bV_0(\Omega_E,s)(S)$  by $t\ast v:=tv.$ This is functorial and so
it defines an action of $\fT$ on on the morphism $\bV_0(\Omega_E, s)\lra \fIG_{n,r,I}$.

\smallskip
(2) Let now $u\colon S \lra \fX_{r,I}$ be a formal scheme. Then a point of $\bV_0(\Omega_E,s)(S)$ is a pair $(\rho, v)$ consisting of a lift $\rho\colon S \lra
\fIG_{n,r,I}$ of $u\colon S\lra \fX_{r,I}$ and an $\cO_S$-linear map $v\colon \rho^\ast(\Omega_E)\lra \cO_S$ such that
$\overline{v}\bigl(\overline{\rho}^\ast(s)\bigr)=1$.

\smallskip
Take $\lambda\in \Z_p^\ast$ and let $\overline{\lambda}$ be its image in $(\Z/p^n\Z)^\ast$, seen  as the Galois group of the adic generic fiber $IG_{n,r,I}$ of
$\fIG_{n,r,I}$ and let us denote by $\overline{\lambda}\colon \fIG_{n,r,I}\cong \fIG_{n,r,I}$ the automorphism over $\fX_{r,I}$ that it defines. Associated to
$\lambda$ there is a natural isomorphism $\gamma_\lambda\colon \Omega_E\cong \overline{\lambda}^\ast(\Omega_E)$ characterized  by the relation:
$\overline{\gamma}_\lambda\bigl({\rm dlog}(P_n)\bigr)=\overline{\lambda}^\ast\bigl({\rm dlog}(P_n)  \bigr)=(\overline{\lambda})^{-1}\cdot {\rm dlog}(P_n)$. Now if
$(\rho, v)$ is a point of $\bV_0(\Omega_E,s)(S)$ and $\lambda\in \Z_p^\ast$ we define $\lambda\ast (\rho, v):=\bigl(\overline{\lambda}\circ \rho,  v\circ
\gamma_\lambda^{-1}\bigr)$. One easily shows that with this definition $t\ast (\rho, v)\in \bV_0(\Omega_E, s)(S)$.

\begin{definition}\label{def:newomegak}
We define the sheaf $\fw^{\rm new, k,0}:=f_{0,\ast}\bigl(\cO_{\bV_0(\Omega_E, s)}  \bigr)\bigl[k_I^0\bigr]$, i.e., $\fw^{\rm new, k,0}$ is the sheaf on $\fX_{r,I}$
on whose sections $x$, the sections $t$ of $\fT^{\rm ext}$ act by $ t\ast x=k_I^0(t)\cdot x.$

\end{definition}

We first have

\begin{lemma}\label{lemma:fisovzero}
We have a natural isomorphism of formal schemes $a\colon \fF_{n,r,I}\lra \bV_0(\Omega_E, s)$ over $\fX_{r,I}$, wich behaves as follows with respect to the $\fT^{\rm ext}$-action: if $\sigma,x$ are section of $\fT^{\rm ext}$ and respectively of $\fF_{n,r,I}$, then $a(\sigma\ast x)=\sigma^{-1}\ast a(x)$.
\end{lemma}

\begin{proof}
We will define the morphism $a$ on $S$-points, where $s\colon S\lra \fX_{r,I}$ is a morphism of formal schemes. A point of $\fF_{n,r,I}(S)$ is a pair $(\rho,
\omega)$ where $\rho=\rho_P$ is a lift of $s$ to a morphism $\rho: S \lra \fIG_{n,r,I}$, $P=\overline{\rho}(P_n)$ is a generator of  $H_n^\vee(S)$ and $\omega\in {\rm
H}^0(S, \rho^\ast(\Omega_E))$ is such that $\overline{\omega}={\rm dlog}(P)$. We define $a_S(\rho, \omega):=(\rho, \omega^\vee)$, i.e. if $\overline{\omega}={\rm
dlog}(P)$ and $P$ is a generator of $H_n^\vee(S)$, then $\omega$ is an $\cO_S(S)$-basis of ${\rm H}^0(S, \rho^\ast(\Omega_E))$ and we denote by $\omega^\vee$ the
unique $\cO_S$-linear map $\omega^\vee\colon \rho^\ast(\Omega_E)\lra \cO_S$ such that $\omega^\vee(\omega)=1$.

It is obvious that $(\rho, \omega^\vee)\in \bV_0(\Omega_E, s)(S)$
and we leave it to the reader to check that all the properties
claimed in the lemma follow easily.
\end{proof}
Lemma \ref{lemma:fisovzero} implies the following

\begin{corollary}\label{cor:wnew=w}
We have an isomorphism of $\cO_{\fX_{r,I}}$-modules $\fw^{\rm new, k,0}\cong \fw^{k,0}$ on $\fX_{r,I}$.
\end{corollary}

\subsubsection{Local description of $\fw^{\rm new, k,0}$.}

Let $\rho\colon S=\Spf(R)\lra \fIG_{n,r,I}$ be a morphism of
formal schemes, where $S$ does not have $\alpha$-torsion and such
that $\rho^\ast(\omega_E)$ is a free $R$-module. We choose an
$R$-basis $\omega$ of $\rho^\ast(\omega_E)$ and denote
$\beta_n:=p^n/(\widetilde{\mathrm{Ha}}(E/R,
\omega))^{\frac{p^n-1}{p-1}}$, $\delta:=\Delta(E/R, P_1, \omega)$
generators of $\rho^\ast(\ubeta_n)$ and respectively
$\rho^\ast(\udelta)$. Let $e$ denote an $R$-basis of
$\rho^\ast(\Omega_E)$ such that $e(\mbox{mod
}\beta_nR)=\rho^\ast({\rm dlog}(P_n))$. Then we have:
$\bV_0(\Omega_E, s)(S)= \{v\colon\rho^\ast(\Omega_E)\lra R,
R-\mbox{linear, such that }\overline{v}(s)=1
\}=(1+\beta_nR)e^\vee$, where $e^\vee$ is the dual basis to $e$.
As described in Section \ref{sec:vbfs},
$\rho^\ast\bigl(\cO_{\bV_0(\Omega_E, s)}\bigr)=R\langle Z\rangle$,
where the point $v=(1+\beta_nr)e^\vee\in \bV_0(\Omega_E, s)(S)$
corresponds to the $R$-algebra homomorphism $\displaystyle
R\langle Z\rangle \lra R$ sending $Z\rightarrow
r=\frac{(1+\beta_nr)-1}{\beta_n}$. We define the action of $\fT(S)$ on $R\langle Z\rangle$ by:
$$
t\ast Z \mbox{ is the element of } R\langle Z\rangle \mbox{ such that } (t\ast v)(Z)=v(t\ast Z)\mbox{ for all } v\in \bV_0(\Omega_E,s)(S), t\in \fT(S).
$$

More precisely, suppose that $t=1+\beta_n b\in \fT(S)$ and $v=(1+\beta_n a)e^\vee$. Let us denote
$t\ast Z=\sum_{n=0}^\infty a_nZ^n$ with $a_n\in R$ and $a_n\rightarrow 0$ as $n\rightarrow \infty$.
Then we have: $v(t\ast Z)=\sum_{n=0}^\infty a_n a^n$ and $(t\ast v)(Z)=a+b+\beta_n ab.$
Therefore we obtain: $b+(1+\beta_n b)a=\sum_{n=0}^\infty a_n a^n$ for all $a\in R$.

It follows that for $t\in \fT(S)$, the action of $t$ on $R\langle Z\rangle$ id given by:
$$
t\ast Z:=\frac{t-1}{\beta_n}+tZ.
$$
Let us denote by $\fw^{',k}:=\pi_\ast\bigl(\cO_{\bV_0(\Omega_E, s)}\bigr)[k_I]$, where the action is the action of $\fT$. Here $\fw^{',k}$ is a sheaf on
$\fIG_{n,r,I}$ which can be described locally by the following lemma.

\begin{lemma}\label{lemma:local1}
a) $\rho^\ast\bigl(\fw^{', k}  \bigr)=R\langle Z\rangle [k]=R\cdot (1+\beta_nZ)^{k}=R\cdot k(1+\beta_nZ)$.

b) $\fw^{',k}$ is a locally free $\cO_{\fIG_{n,r,I}}$-module of rank one.
\end{lemma}

\begin{proof}
b) is a consequence of a) so let us prove a).

Let us first see that because of the analyticity properties of $k$, we have $(1+\beta_nZ)^{k}:=k(1+\beta_nZ)\in R\langle Z\rangle^\ast$. Moreover if $t\in \fT(S)$ we
have:
$$
t\ast (1+\beta_nZ)=t\cdot (1+\beta_nZ)\mbox{ which implies } t\ast(1+\beta_nZ)^{k}=
k(t)(1+\beta_nZ)^{k},
$$
i.e., $R(1+\beta_nZ)^{k}\subset R\langle Z\rangle [k].$ To show
the inverse inclusion it would be enough to see that $R\langle
Z\rangle ^{\fT(S)}=R$. Let $g(Z)=\sum_{n=0}^\infty a_nZ^n\in
R\langle Z\rangle^{\fT(S)}$. Then if $t\in \fT(S)$ we have
$$
g(Z)=t\ast g(Z)=\sum_{n=0}^\infty a_n\Bigl(\frac{t-1}{\beta_n}+tZ\Bigr)^n,
$$
for $Z=0$ the above relation implies: for all $u\in R$, $\sum_{n=1}^\infty a_nu^n=0$ which implies that $a_n=0$ for all $n\ge 1$. Therefore $g(Z)\in R$.
\end{proof}

\

\noindent If we wish to describe $\fw^{\rm new, k}$ we need to consider the residual $\Z_p^\ast$-action on $\fw^{',k}$ and descend from $\fIG_{n,r,I}$ to
$\fX_{r,I}$.

Suppose first $n=1$. In that case $\Omega_E=\Delta\cO_{\fIG_{1,r,I}}$ and if $\rho:S=\Spf(R) \lra\fIG_{1,r,I}$ is a morphism of formal schemes lifting $s:S\lra
\fX_{r,I}$, and $\omega,\beta_1, \delta$ are as above, we can choose a basis $e$ of $\rho^\ast(\Omega_E)$ to be: $e=\delta\omega=\Delta$. We then know that
$e(\mbox{mod }\beta_1R)={\rm dlog}(P_1)$ and $e^\vee:=\delta^{-1}\omega^\vee$. If $\lambda\in \Z_p^\ast$ and denote by $\overline{\lambda}$ its image in
$(\Z/p\Z)^\ast$, then we have $\overline{\lambda}\ast e^\vee= w(-\overline{\lambda})e^\vee$, where $w:(\Z/p\Z)^\ast \lra \mu_{p-1}$ is the Teichm\"uller map.

Therefore if we denote by $Z$ the function on $\bV_0(\Omega_E, s)_S$ corresponding to this choice of basis we have: if $\lambda \in \Z_p^\ast$ then
$$
\lambda\ast Z=\frac{\lambda-w(\overline{\lambda})}{p}+\lambda Z,
$$
and so $\lambda\ast (1+pZ)^{k}=k(\lambda)(1+pZ)^{k}=k(\lambda)(1+pZ)^{k}$. Therefore $s^\ast (\fw^{', k})=R\langle Z\rangle [k]$ for the action of the big torus $\fT^{\rm
ext}(S)$, i.e., $s^\ast(\fw^{', k})=R\cdot (1+pZ)^{k}$. In particular $\fw^{\rm new, k}$ is a locally free $\cO_{\fX_{r,I}}$-module of rank $1$.

\

\noindent If $n\ge 2$ then such explicit actions of $\Z_p^\ast$ on $\fw^{',k}$ cannot be found and so in order to descend to $\fX_{r,I}$ one has to use traces as in
\cite{halo_spectral}.

\subsection{The sheaf $\bW_k$.}
\label{sec:wk}

We fix a closed  interval $I\subset [0,\infty)$ as in  Section \S \ref{sec:omega} and denote by $n$, $r$ integers compatible with this choice of interval as in Definition \ref{defi:wnrI}. Let us
also denote by $(\Lambda_I^0,\alpha)$ a pair as in Section \S \ref{sec:omega} and denote by $\Omega_E$ the subsheaf of $\omega_E$ given in Definition
\ref{def:bigomega}.

Let ${\rm H}_E$ denote the contravariant Dieudonn\'e-module attached to the $p$-divisible group of the universal elliptic curve of the complement of the cusps in
$X=X_1(N)$. It is a locally free coherent sheaf on this complement with an integrable connection $\nabla$ and a Hodge filtration. The sheaf ${\rm H}_E$ extends
naturally to a locally free $\cO_X$-module over the whole $X$, denoted also ${\rm H}_E$ with:
\smallskip

a)  a logarithmic connection $\nabla\colon {\rm H}_E\lra {\rm H}_E\otimes_{\cO_X}\Omega^1_{X/\Z_p}({\rm log}(C))$, where $C$ is the divisor of the cusps.

\smallskip

b) a Hodge filtration
$$
0\lra \omega_E\lra {\rm H}_E\lra \omega_E^{-1}\lra 0.
$$

Having fixed $I$, $r$, $n$ we have natural formal schemes with morphisms $\fIG_{n,r,I}\lra \fX_{r,I}\lra \fY$ (the formal scheme associated to $X_1(N)$) and we can base-change the triple $({\rm H}_E,
\nabla, {\rm Fil}^\bullet)$ over $\fIG_{n,r,I}$ where we denote it by the same symbols: $({\rm H}_E, \nabla, {\rm Fil}^\bullet)$.

The data $I$, $r$, $n$ also fixes a universal weight $k\colon \Z_p^\ast\lra \Lambda_I^\ast$. Let us denote by (see Section \ref{sec:pdivconstructions}) ${\rm
H}_E^\#:=\Omega_E+\udelta^p{\rm H}_E$. As $\udelta$ is a locally free $\cO_{\fIG_{n,r,I}}$-module of rank $1$, then $\Hsharp$ is a locally free $\cO_{\fIG_{n,r,I}}$-module of
rank $2$, with Hodge filtration given by the exact sequence
$$
0\lra \Omega_E\lra \Hsharp\lra \udelta^p\omega_E^{-1}\lra 0.
$$
Therefore if we consider the ideal $\ubeta_n=p^n/ {\rm Hdg}^{\frac{p^n-1}{p-1}}$ of
$\cO_{\fX_{r,I}}$ and we denote by $s:={\rm dlog}(P_n)\in \Omega_E/\ubeta_n\Omega_E\hookrightarrow \Hsharp/\ubeta_n\Hsharp$, then the pair $(\Hsharp, s)$ is a pair consisting
of a locally free sheaf and a marked section. We therefore have the sequence of formal schemes and morphisms of formal schemes $\displaystyle \bV_0(\Hsharp,
s)\stackrel{\pi}{\lra}\fIG_{n,r,I}\stackrel{h_n}{\lra}\fX_{r,I}$.

We denote by  $\mathrm{H}_{E,\sharp}$ the dual of $\mathrm{H}_E^\sharp$ and by $f_0\colon  \V_0(\Hsharp,s)\lra \mathfrak{X}_{r,I}$ the structure morphism, i.e. $f_0:=h_n\circ \pi$.

\subsubsection{Actions of formal tori on  $\bV_0(\Hsharp,s)$.}

Let us recall that we have denoted by $\fT\subset \fT^{\rm ext}$
the formal groups over $\fX_{r,I}$ defined by: if $\rho\colon S \lra
\fX_{r,I}$ is a morphism of formal schemes, then
$\fT(S):=1+\rho^\ast(\ubeta_n) \cO_S\subset \fT^{\rm
ext}(S):=\Z_p^\ast\bigl(1+\rho^\ast(\ubeta_n) \cO_S\bigr)\subset
\bG_{m,S}$.

As in Section \ref{sec:new}, we have natural actions of $\fT$ and respectively $\fT^{\rm ext}$ on $\bV_0(\Hsharp, s)$ over $\fIG_{n,r,I}$ and respectively
$\fX_{r,I}$. Let us quickly recall how this action is defined on points:
\smallskip

(1) Let $\rho\colon S\lra \fIG_{n,r,I}$ be a morphism of formal schemes and let $t$ be an element of $\fT(S)$ and $v$ a point in $\bV_0(\Hsharp, s)(S)$.
We define the action of $\fT(S)$ on $\bV_0(\Hsharp,s)(S)$  by $t\ast v:=tv.$ This is functorial and so
it defines an action of $\fT$ on $\bV_0(\Hsharp, s)$ over $\fIG_{n,r,I}$.

\smallskip
(2) Let now $u\colon S \lra \fX_{r,I}$ be a formal scheme. Then a point of $\bV_0(\Hsharp,s)(S)$ is a pair $(\rho, v)$ consisting of a lift $\rho:S \lra
\fIG_{n,r,I}$ of $u\colon S\lra \fX_{r,I}$ and an $\cO_S$-linear map $v\colon \rho^\ast(\Hsharp)\lra \cO_S$ such that
$\overline{v}\bigl(\overline{\rho}^\ast(s)\bigr)=1$.

\smallskip
Let now $\lambda\in \Z_p^\ast$ and let $\overline{\lambda}$ be its image in $(\Z/p^n\Z)^\ast$, seen  as the Galois group of the adic generic fiber $IG_{n,r,I}$ of
$\fIG_{n,r,I}$ and let us denote by $\overline{\lambda}\colon \fIG_{n,r,I}\cong \fIG_{n,r,I}$ the automorphism over $\fX_{r,I}$ that it defines.

Associated to $\lambda$ there is a natural isomorphism $\gamma_\lambda\colon \Hsharp\cong \overline{\lambda}^\ast(\Hsharp)$ characterized  by:
$\overline{\gamma}_\lambda\bigl({\rm dlog}(P_n)\bigr)=\overline{\lambda}^\ast\bigl({\rm dlog}(P_n)  \bigr)=(\overline{\lambda})^{-1}\cdot {\rm dlog}(P_n)$.

Therefore if $(\rho, v)$ is a point of $\bV_0(\Hsharp,s)(S)$ and $\lambda\in \Z_p^\ast$ we define $\lambda\ast (\rho, v):=\bigl(\overline{\lambda}\circ \rho,  v\circ
\gamma_\lambda^{-1}\bigr)$. As in Section \ref{sec:new} one can show that with this definition $t\ast (\rho, v)\in \bV_0(\Hsharp, s)(S)$.

Let us recall that we have a universal weight associated to our choices of $r$, $I$, $n$ and that the analyticity properties of this weight imply that if $t\in
\fT^{\rm ext}(S)$ for some formal scheme $S\lra \fX_{r,I}$ then we can evaluate $k(t)$ and we get a section of $\cO_S^\ast$.

\

\begin{definition}\label{def:omegak} Fix $r$, $n$ and a closed interval $I:=[p^a, p^b]\subset [0, \infty)$ as in Definition \ref{defi:wnrI}.  We define the sheaf $\bW^0_{k,I}:=f_{0,\ast}\bigl(\cO_{\bV_0(\Hsharp, s)} \bigr)[k]$, i.e.
$\bW^0_{k,I}$ is the sheaf on $\fX_{r,I}$ on whose sections $x$, the sections $t$ of $\fT^{\rm ext}$ act by $ t\ast x=k(t)\cdot x$. 

For $r\geq 3$ if $p\geq 3$ and $r\geq 5$ for $p=2$ and $I=[p,\infty]$ we
define $\bW^0_{k,I}=\lim_{n \geq 1} \bW^0_{k,[p^n,p^{n+1}]}$.

We let $\bW_{k,I}:=\bW_{k,I}^0\otimes_{\cO_{\fX_{r,I}}}\fw^{k_f}$.

\end{definition}

Let us point out that the inclusion $\Omega_E\subset \Hsharp$
gives a filtration of locally free sheaves with marked sections
$(\Omega_E, s)\hookrightarrow (\Hsharp, s)$, therefore the sheaf
$f_{0,\ast}\bigl(\cO_{\bV_0(\Hsharp, s)}\bigr)$  has a canonical
filtration ${\rm
Fil}_\bullet\Bigl(f_{0,\ast}\bigl(\cO_{\bV_0(\Hsharp, s)}\big)\Bigr):=f_{0,\ast}\bigl(\Fil_\bullet \cO_{\bV_0(\Hsharp, s)}\bigr)$.

\begin{theorem}\label{thm:descentbWk}
The action $\fT^{\rm ext}$ on $f_{0,\ast}\bigl(\cO_{\bV_0(\Hsharp,s)}\bigr)$ preserves the filtration $f_{0,\ast}\bigl(\Fil_\bullet \cO_{\bV_0(\Hsharp, s)}\bigr)$
defined in Corollary \ref{cor:fil}. For every $h$ define $\Fil_h\bW^0_{k,I}:=f_{0,\ast}\bigl(\Fil_h \cO_{\bV_0(\Hsharp, s)}\bigr)[k].$ for $r$, $n$ and  $I\subset [0, \infty)$ as in Definition \ref{defi:wnrI} and, for $r\geq 3$ if $p\geq 3$ and $r\geq 5$ for $p=2$,  as $\Fil_h\bW^0_{k,I}:=\lim_{ n \geq 1} \Fil_h \bW^0_{k,[p^n,p^{n+1}]}$ for $I=[p,\infty]$. Then,

\begin{itemize}

\item[i.]  $\Fil_h\bW^0_{k,I}$ is a locally free $\cO_{\fX_{r,I}}$-module for the Zariski topology on $\fX$;

\item[ii.] $\bW^0_{k,I}$ is the $\alpha$-adic completions of $\lim_h \Fil_h\bW^0_{k,I}$.

\item[iii.] $\Fil_0\bW^0_{k,I}\cong \fw^{k,0}_I$ and $ \mathrm{Gr}_h\bW^0_{k,I} \cong \fw^{k,0}_I \otimes_{\cO_{\mathfrak{X}_{r,I}}} \mathrm{Hdg}^h \omega_{E}^{-2h}$.

\end{itemize}

Define  ${\rm Fil}_h\bW_{k,I}:=
{\rm Fil}_h\bW_{k,I}^0\otimes_{\cO_{\fX_{r}}}\fw^{k_f}$. It defines an increasing filtration $\{\Fil_h\bW_{k,I}\}_h$ by direct summands such that the analogous Claims (ii) and (iii) hold
(replacing $\fw^{k,0}_I$ with $\fcw^{k}_I$). \smallskip

The sheaves $\bW_{k,I}^0$, $\bW_{k,I}$ and their filtrations glue to sheaves $\bW_{k}^0$, $\bW_{k}$,  $\{\Fil_h\bW^0_{k}
\}_h$, $\{\Fil_h\bW_{k}
\}_h$ over $\fX_{r,[0,\infty)}$ (resp. $\fX_{r,[0,\infty]}$) if $r\geq 1$ (resp. $r\geq 3$) if $p\geq 3$ and $r\geq 3$ (resp. $r\geq 5$) for $p=2$.\smallskip

Finally if $k\in \N$ is a classical weight then we have a canonical identification $${\rm Sym}^k\bigl({\rm
H}_E\bigr)[1/\alpha]={\rm Fil}_k(\bW_k)[1/\alpha]$$as sheaves on $\mathcal{X}_{r,I}$, compatibly with the filtrations considering on ${\rm Sym}^k\bigl({\rm
H}_E\bigr)$ the natural Hodge filtration.

\end{theorem}
\begin{proof}
The proof of Claims (i), (ii) and (iii) of the theorem for $r$, $n$ and a closed interval $I\subset [0, \infty)$ as in Definition \ref{defi:wnrI} for $\bW^0_k$ will be given in Section \ref{sec:proof}. As $\Fil_h\bW^0_k
\subset \bW^0_k$ is locally a direct summand by (i) and (ii), then the analogous statements for $\bW_k$ follow. Since the construction of $\bV_0(\Hsharp,s)$ does not depend on $I$ and is functorial in $n$ and since $\fw^{k,0}$ arises from an invertible sheaf on the whole $\fX_{r,[0,\infty)}$ by \cite[Thm. 5.1]{halo_spectral}, then Claims (ii) and (iii) imply that $\bW_{k,I}^0$, $\bW_{k,I}$ and their filtrations do not depend on $n$ and glue for varying intervals $I$  to sheaves  $\bW_{k}^0$, $\bW_{k}$ on $\fX_{r,[0,\infty)}$.

We deduce from this the theorem for $I=[p,\infty]$ assuming that  $r\geq 3$ if $p\geq 3$ and $r\geq 5$ for $p=2$. Claim (i) holds for $\Fil_0\bW^0_{k,I}\cong \fw^{k,0}_I$: it is free $\fX_{r,I}$-module over every affine formal subscheme of $\fX$ on which $\omega_E$ is free and it coincides with the limit $\lim_{n \geq 1} \fw^{k,0}_{k,[p^n,p^{n+1}]}$ due to \cite[Rmk. 6.2]{halo_spectral}. Then the same statements hold true for the sheaves $\Fil_h\bW^0_{k,I}$ thanks to claims (i) and (iii) for $\Fil_h\bW^0_{k,[p^n,p^{n+1}]}$ and their functoriality in the interval $[p^n,p^{n+1}]$.  As $\fw^{k,0}_I=\lim_{n \geq 1}
\fw^{k,0}_{[p^n,p^{n+1}]}$ by \cite[Thm. 6.4]{halo_spectral} and $\cO_{\fX_{r,I}}=\lim_{m\geq n} \cO_{\fX_{r,[p^n,p^{n+1}]}}$ by \cite[Lemme 6.5]{halo_spectral}  claims  (ii) and (iii) hold also for $I=[p,\infty]$.

For the last part of the Theorem for integral weights $k$, recall that we have an inclusion $\Hsharp\subset {\rm H}_E$ of sheaves over $\fIG_{n,r,I}$, compatible with the filtrations, which
is an isomorphism after inverting $\alpha$. By Definition \ref{def:V0} and Lemma \ref{lem:V(E)} this provides natural morphisms $\bV_0(\Hsharp, s)\to \bV(\Hsharp)
\leftarrow \bV({\rm H}_E)\vert_{\fIG_{n,r,I}}$ of formal vector bundles (with section) over $\fIG_{n,r,I}$. Notice that structure sheaf of $\bV(\Hsharp)$,
resp.~$\bV({\rm H}_E)$ is identified with the $\alpha$-adic completion of the symmetric algebra of $\Hsharp$, resp.~${\rm H}_E$ (see the proof of Lemma
\ref{lem:V(E)}). It follows from the local description of $\Fil_k\bW^0_k$ in Lemma \ref{lemma:localwkappa} that these morphisms map ${\rm Sym}^k\bigl(\Hsharp\bigr)
\to \Fil_k\bW^0_k$ and clearly  ${\rm Sym}^k\bigl(\Hsharp\bigr) \to {\rm Sym}^k\bigl({\rm H}_E\bigr)$ and that these are isomorphisms of sheaves  over
$\fIG_{n,r,I}$ after inverting $\alpha$. This provides the claimed identification over $\fIG_{n,r,I}$. It follows directly  from the definition of the filtration in \S \ref{sec:vbfil}
that this isomorphism is compatible with
the filtrations.

\end{proof}

\subsubsection{Local description of $\bV_0\bigl(\mathrm{H}^\sharp_E, s\bigr)$}\label{sec:Hzero}

\medskip
\noindent Let $\rho\colon S=\Spf(R)\lra \fIG_{n,r,I}$ be a
morphism of formal schemes over $\Lambda_I^0$ (without
$\alpha$-torsion)  such that $\rho^\ast(\omega_E)$ is a free
$R$-module of rank one. Let as usual $\omega, \beta_n, \delta$
denote an $R$-basis of $\rho^\ast(\omega_E)$, the appropriate
generator of $\rho^\ast(\ubeta_n)$ and an appropriate generator of
$\rho^\ast(\udelta)$. We fix an $R$-basis $(f,e)$ of
$\rho^\ast(\Hsharp)$ such that $f(\mbox{mod
}\beta_nR)=\rho^\ast({\rm dlog}(P_n))$,  where, let us recall,
$P_n$ is the universal generator of $\mathrm{H}_n^\vee$ over
$\mathfrak{IG}_{n,r,I}$. We denote by $(f^\vee, e^\vee)$ the dual
$R$-basis of $\mathrm{H}_{E,\sharp}$. Since
$f^\vee(f)=1=f^\vee\bigl(\rho^\ast({\rm dlog}(P_n)\bigr)$ and
$e^\vee(f)=0=e^\vee\bigl({\rm dlog}(P_n))\bigr)$ modulo
$\beta_nR$, it follows from Definition \ref{def:V0} that
$$
\bV_0\bigl(\mathrm{H}^\sharp_E, s\bigr)(S)=\{a f^\vee+b e^\vee\ \vert  \ a\in 1+\beta_n R\mbox{ and } \gamma\in R   \}
$$
and thanks to Lemma \ref{lem:V0} we have that $\bV_0\bigl(\mathrm{H}^\sharp_E\bigr)\times_{\fIG_{n,r,I}} S=\Spf\bigl(R\langle Z,Y \rangle\bigr)$. A point $x=a
f^\vee+b e^\vee \in \bV_0(\Hsharp)(S)$ corresponds to the $R$-algebra homomorphism $R\langle Z,Y\rangle \lra R$ sending $Z\mapsto \frac{a-1}{\beta_n}$ and $Y\mapsto
b$.

\begin{remark}
We have an interesting interpretation of the sections of the sheaf
$\rho^\ast\bigl(\cO_{\bV_0(\Hsharp, s)}\bigr)$. Let us first
remark that $\bV_0(\Hsharp)(S)$ can be naturally identified with
$${\rm H}_{E,0}(S)=\{u\in \rho^\ast\bigl({\rm H}_{E,\sharp}  \bigr)\quad \vert \quad u(\mbox{mod }\beta_nR)\rho^\ast\bigl({\rm dlog}(P_n)\bigr)=1\}.
$$Recall that ${\rm H}_{E,\sharp} $ is the dual of $\Hsharp$ and the expression $ u(\mbox{mod }\beta_nR)\rho^\ast\bigl({\rm dlog}(P_n)\bigr)$ stands for the pairing of $u$, modulo $\beta_n$, and  ${\rm dlog}(P_n)$.  
Therefore the sections of $\rho^\ast\bigl(\cO_{\bV_0(\Hsharp, s)}\bigr)$ can be seen as functions $\gamma\colon {\rm H}_{E,0}(S)\lra R$ which are analytic in the sense that there is $g(Z,Y)\in
\rho^\ast\bigl(\cO_{\bV_0(\Hsharp, s)}\bigr)=R\langle Z,Y\rangle$ such that for all $z$, $y\in R$ we have $\gamma\bigl((1+\beta_n z)f^\vee+y e^\vee\bigr)= g(z,y)$. We recall that $(f^\vee,
e^\vee)$ is the basis of $\rho^\ast({\rm H}_{E,\sharp})$ which is $R$-dual to $(f,e)$.

\end{remark}

\medskip
\noindent
Let $v\in \bV_0(\Hsharp)(S)$, then $v\colon \rho^\ast(\Hsharp)\to \cO_S$  is $\cO_S$-linear and
$\overline{v}(s)=1$, i.e. $v\in \mathrm{H}_{E, 0}(S)$. Take $t\in
\fT(S)\subset \fT^{\rm ext}(S)$ and $\gamma\colon {\rm H}_{E,0}(S)\lra R$ an analytic function, we have $(t^{-1}\ast \gamma)\bigl( v)=\gamma(tv)$. If $v= a f^\vee+  b e^\vee\in \mathrm{H}_{E,0}(S)$ then
$(t^{-1}\ast \gamma)(af^\vee+be^\vee)=\gamma\bigl(t a f^\vee+t b e^\vee \bigr)=g\bigl(t a , t b\bigr)$, where $g(Z,Y)\in R\langle Z,Y\rangle=\rho^\ast\bigl(\cO_{\bV_0(\Hsharp)}\bigr)$ is the
section associated to $\gamma$. Then $(t^{-1}\ast g)(Z,Y)=g\Bigl(\frac{t-1}{\beta_n}+t Z, t Y\Bigr)$, i.e., then $t\ast
(1+\beta_nZ)=t^{-1}(1+\beta_nZ)$ and $t\ast Y=t^{-1}Y$.

\

Let us recall that we denoted $\pi\colon\bV_0(\Hsharp, s)\lra \fIG_{n,r,I}$ the structure morphism and that we have an action of $\fT$ on this morphism.

\begin{lemma}
\label{lemma:localwkappa}
$$\rho^\ast\Bigl(\pi_{\ast}\bigl(\cO_{\bV_0(\Hsharp, s}\bigr)[k]\Bigr)=R\langle Z,Y\rangle [k]=\Bigl\{\sum_{m=0}^\infty a_m(1+\beta_nZ)^{k}\frac{Y^m}{(1+\beta_nZ)^m}\Bigr\},$$ where
$a_m\in R$ for all $m\ge 0$ such that $a_m\rightarrow 0$ as $m\rightarrow \infty$. Similarly $\displaystyle \pi_\ast\bigl(\Fil_h\bW^0_k\bigr)[k]=\bigl\{\sum_{m=0}^h
a_m(1+\beta_nZ)^{k}\frac{Y^m}{(1+\beta_nZ)^m}, \mbox{ where } a_m\in R \mbox{ for all } m=0,...,h\bigr\}$.
\end{lemma}

\begin{proof}

Clearly  $(1+\beta_nZ)^{k-m}Y^m\in \rho^\ast(\cO_{\bV_0(\Hsharp, s)})=R\langle Z,Y\rangle$ and if $a\in \fT(S)$, then $a\ast
(1+\beta_nZ)^{k-m}Y^m=k(a)\bigl((1+\beta_nZ)^{k-m}Y^m  \bigr)$ for every $m\ge 0$.

In order to prove the converse let us first prove: $\bigl(R\langle Z,Y\rangle\bigr)^{\fT(S)}=R\langle V\rangle$, where we denoted by $\displaystyle V:=\frac{Y}{1+\beta_nZ}\in
R\langle Z,Y\rangle$. It is obvious that $R\langle V\rangle\subset \bigl(R\langle Z,Y\rangle \bigr)^{\fT(S)}$.

Let us notice that every element $f(Z,Y)=\sum_{i,m=0}^\infty a_{i,m}Z^iY^m\in R\langle Z,Y\rangle$ can be written uniquely as $f(Z,Y)=g(Z,V):=\sum_{u,v=0}^\infty
b_{u,v}Z^uV^v$ by writing $Y^m=V^m(1+\beta_nZ)^m$, where $b_{u,v}\rightarrow 0$ as $u+v\rightarrow \infty$. Then if $a\in 1+\beta_nR=\fT(S)$, we have that $a\ast
g(Z,V)=g(Z,V)$ if and only if:
$$
\sum_{u,v=0}^\infty b_{u,v}\bigl(\frac{a-1}{\beta_n}+aZ \bigr)^uV^v=\sum_{u,v=0}^\infty
b_{u,v}Z^uV^v.
$$

Regarding the above equality as an equality in $R\langle Z\rangle [[V]]$ for every $v\ge 0$ we must have: $\sum_{u=0}^\infty
b_{u,v}\bigl(\frac{a-1}{\beta_n}+aZ \bigr)^u=\sum_{u=0}^\infty b_{u,v}Z^u.$ For $Z=0$ this gives $\sum_{u=1}^\infty
b_{u,v}\Bigl(\frac{a-1}{\beta_n}\Bigr)^u=0$, for every $a\in 1+\beta_nR$. The Weierstrass preparation theorem implies that for every $v\ge 0$, $b_{u,v}=0$
for $u\ge 1$, i.e. $g(Z,V)=\sum_{v=0}^\infty b_{0,v}V^v\in R\langle V\rangle$, which proves the claim.

Now obviously $R\langle Z,Y\rangle[k]$ is naturally an $R\langle V\rangle$-module and if $f(Z,Y)\in R\langle Z,Y\rangle[k]$ then $$\frac{f(Z,Y)}{(1+\beta_nZ)^{k}}\in \bigl(R\langle
Z,Y\rangle \bigr)^{\fT(S)}=R\langle V\rangle,$$ therefore $f(T,Y)=(1+\beta_n Z)^{k}R\langle V\rangle$ which proves the lemma.
\end{proof}

\subsubsection{The proof of theorem \ref{thm:descentbWk}.}\label{sec:proof}

Let us recall the sequence of formal schemes and morphisms:
$$
\bV_0(\Hsharp, s)\stackrel{\pi}{\lra}\fIG_{n,r,I}\stackrel{h_n}{\lra}\fX_{r,I}\mbox{  we denoted by }f_0:=h_n\circ \pi.
$$
We have denoted by $\bW^0_k:=f_{0,\ast}\bigl(\cO_{\bV_0(\Hsharp, s)}  \bigr)[k]$, where the action is that of $\fT^{\rm ext}$ and by
$\widetilde{\bW}^0_k:=\pi_\ast\bigl(\cO_{\bV_0(\Hsharp, s)}\bigr)[k]$, for the action of $\fT$. Then we obviously have
$\bW^0_k=h_{n,\ast}\bigl(\widetilde{\bW}^0_k \bigr)[k]$, for the action of  $\Z_p^\ast$.

Let us remark that lemma \ref{lemma:localwkappa} implies that the filtration of $\widetilde{\bW^0}_k$ defined as
$\Fil_h(\widetilde{\bW^0}_k):=\pi_\ast\bigl(\Fil_h(\cO_{\bV_0(\Hsharp, s)}  \bigr)[k]$ is a locally free $\cO_{\fIG_{n,r,I}}$-module of rank $h+1$ and
$\widetilde{\bW}^0_k$ is the $\alpha$-adic completion of $\lim_h \Fil_h(\widetilde{\bW}^0_k)$. Moreover we have ${\rm Gr}^i\widetilde{\bW}^0_k\cong
(\fw')^{k}\otimes\bigl({\rm Hdg}^i_E \omega_E^{-2i}\bigr)$, where let us recall that $(\fw')^{k}= \pi_\ast\bigl(\cO_{\bV_0(\Omega_E, s)}  \bigr)[k]$ for the action of $\fT$.

Let $U=\Spf(R)$ be an open affine subscheme of $\fX_{r,I}$ such that $\omega_E\vert_U$ is free and let $\omega$ be an $R$-basis of $\omega_E(U)$. Let us denote by
$V=\Spf(R_n):=h_n^{-1}(U)\subset \fIG_{n,r,I}$.

It is shown in  Lemma 5.3 of \cite{halo_spectral} that the map $\bV_0(\Omega_E, s) \to \fIG_{n,r,I}$ induces an $\fT$-equivariant isomorphism
$R_n/qR_n\cong (\fw')^{k}(V)/q(\fw')^{k}(V)$ (using Definition \ref{def:newomegak} and Corollary \ref{cor:wnew=w} to identify $(\fw')^{k}$ as a subsheaf of the structure sheaf of $\bV_0(\Omega_E, s)$).
For every $i\ge 0$ choose an element $$\overline{s}_i\in \bigl(
(\fw')^{k}(V)\otimes {\rm Hdg}^i_E \omega_E^{-2i}(V)/(q)\bigr)(V) \cong (R_n/qR_n)\otimes \bigl({\rm Hdg}^i_E \omega_E^{-2i}(V)/(q)\bigr)(V), $$
mapping  to the class of ${\rm Hdg}^i_E  \omega^{-2i}$.
In particular $\overline{s}_i$ is an $R_n/pR_n$-generator of $\bigl({\rm
Gr}^i\widetilde{\bW}^0_k/q{\rm Gr}^i\widetilde{\bW}^0_k\bigr)(V)$ such that
for every $t\in \fT^{\rm ext}(V)$ we have $t\ast
\overline{s}_i=\overline{s}_i$ as ${\rm Hdg}^i_E  \omega^{-2i}$ is invariant for the $\fT$-action. Let $s_i\in
\Fil_i(\widetilde{\bW}^0_k)(V)$ be such that $s_i\bigl(\mbox{mod
}p\Fil_i(\widetilde{\bW}^0_k)\bigr)=\overline{s}_i$ for every
$i\ge 0$. We denote by
$h:=\widetilde{\mathrm{Hdg}}(E/R, \omega)$.
By Corollary 3.1 of \cite{halo_spectral}, there is an element
$c_n\in h^{-\frac{p^n-p}{p-1}}R_n$ such that
we have ${\rm Tr}(c_n)=1$, where ${\rm Tr}$ denotes the trace of
$R_n$ over $R$.

Let us denote for every $i\ge 0$, by
$$
\tilde{s}_i:=e_{c_n}(s_i):=\sum_{\sigma\in (\Z/p^n\Z)^\ast}k(\tilde{\sigma})\bigl(\sigma(c_ns_i)\bigr) \in {\rm H}^0\Bigl(U,
h_{n,\ast}\bigl(\widetilde{\bW}^0_k\bigr)\Bigr).
$$
Here $\tilde{\sigma}\in \Z_p^\ast$ is a lift of $\sigma$. We have

\begin{lemma}\label{lemma:gradedpieces}
$\tilde{s}_i\in {\rm H}^0\Bigl(U, h_{n,\ast}\bigl(\Fil_i(\bW^0_k)  \bigr)  \Bigr)$ and $\tilde{s}_iR\cong {\rm Gr}^i(\bW^0_k)(U)=\fw^{k,0}(U)\otimes_R
\omega_E^{-2i}(U)$
\end{lemma}

\begin{proof}

Let us first remark that the elements $\tilde{s}_i$ belong to ${\rm H}^0(U, \bW^0_k)$. We write $s_i=h^i\omega^{-2i}+p f_i$ where $f_i\in {\rm H}^0(V,
\widetilde{\bW}^0_k)$. Therefore
$$
\tilde{s}_i=\Bigl(\sum_{\sigma\in (\Z/p^n\Z)^\ast}k(\tilde{\sigma})\sigma(c_n)  \Bigr)h^i\omega^{-2i}+ p\sum_{\sigma\in
(\Z/p^n\Z)^\ast}k(\tilde{\sigma})\sigma(c_nf_i).
$$
Let us observe that if we denote by $R^{oo}$ the ideal of $R$ of its topologically nilpotent elements, then following the arguments of Lemma 5.4 of
\cite{halo_spectral} we have: $\Bigl(\sum_{\sigma\in (\Z/p^n\Z)^\ast}k(\tilde{\sigma})\sigma(c_n)  \Bigr)\in 1+R^{oo}R_n$ and $p\sum_{\sigma\in
(\Z/p^n\Z)^\ast}k(\tilde{\sigma})\sigma(c_nf_i)\in R^{oo} \cO_{\bV_0(\Hsharp, s)}(f_0^{-1}(U))$. Again the arguments in the proof of Lemma 5.4 of
\cite{halo_spectral} imply that $\tilde{s}_i\in \Fil_i(\bW^0_k)(U)$ and its image in ${\rm Gr}^i(\bW^0_k)(U)=\fw^{k,0}(U)\otimes_R (h^i \omega_E^{-2i}(U))$ generates this
$R$-module.
\end{proof}

Lemma \ref{lemma:gradedpieces} proves the part related to the filtration in the statement of Theorem \ref{thm:descentbWk}.

To prove the rest of Theorem \ref{thm:descentbWk} let us also remark that we have (using the arguments
in the proof of Lemma 5.4 of \cite{halo_spectral} and the notation of Lemma \ref{lemma:gradedpieces}) that:
$$
\tilde{s}_i-s_i=\sum_{\sigma\in (\Z/p^n\Z)^\ast}k(\tilde{\sigma})\bigl( \sigma(c_ns_i)-s_i\bigr)= \sum_{\sigma\in
(\Z/p^n\Z)^\ast}k(\tilde{\sigma})\bigl(\sigma(c_ns_i)-\sigma(c_n)s_i\bigr)=
$$

$$
=\sum_{\sigma\in (\Z/p^n\Z)^\ast}\bigl(\sigma(c_n)(\sigma(s_i)-s_i  \bigr)\in R^{oo}\Fil_h(\widetilde{\bW}^0_k).
$$
It follows that $(\tilde{s}_i)_{i=0}^h$ is an $R_n$-basis of $\Fil_h(\widetilde{\bW}^0_k)$ and also an $R$-basis of $\Fil_h(\bW^0_k)$. Therefore $\bW^0_k(U)$ is the
$\alpha$-adic completion of the $R$-module  ${\rm lim}_h \Fil_h(\bW^0_k)$.

For future applications it is also useful to denote by $\bW^0:=f_{0,\ast}\bigl(\cO_{\bV_0({\rm H}_E^\#,s)}\bigr)$. It is a sheaf of $\cO_{\fX_{r,I}}$-algebras on
$\fX_{r,I}$ containing all $\bW^0_k$ for various weights $k$. We have

\begin{lemma}\label{lemma:wfree}
Suppose that $\alpha=p$ and let $i>0$. Then $\bW^0/p^i\bW^0$ is a locally free $\cO_{\fX_{r,I}}/p^i\cO_{\fX_{r,I}}$-module.
\end{lemma}

\begin{proof}
Let us remark that if we denote by $\widetilde{\bW}^0:=\pi_\ast\bigl(\cO_{\bV_0({\rm H}_E^\#,s)} \bigr)$, then the local description of this sheaf in Section
\ref{sec:Hzero} implies immediately that $\widetilde{\bW}^0/p^i\widetilde{\bW}^0$ is a locally free $\cO_{\fIG_{n,r,I}}/p^i\cO_{\fIG_{n,r,I}}$. Now if $\alpha=p$
then $\fX_I, \fX_{r,I}, \fIG_{n,r,I}$ are all base changes to $\Spf(\Lambda_I^0)$ of $p$-adic formal schemes $\fX$, $\fX_r$, $\fIG_{n,r}$ which have absolute
dimension $2$ and such that $\fX_r$ is regular and $\fIG_{n,r}\lra \fX_r$ is finite and normal, therefore this morphism is finite and flat. As a consequence
$\cO_{\fIG_{n,r,I}}/p^i\cO_{\fIG_{n,r,I}}$ is a locally free $\cO_{\fX_{r,I}}/p^i \cO_{\fX_{n,r,I}}$ module, which proves the lemma.
\end{proof}

\subsubsection{An alternative construction of $\bW_{k,\infty}$.}

In this section we provide a purely characteristic $p$ construction of $\bW^0_{k,\infty}$ and $\bW_{k,\infty}$. We work with the pair $(A_0,\alpha)$ with
$A_0:=\Lambda^0/p\Lambda^0\cong \F_p[\![T]\!]$ and $\alpha=T$.

Fix an integer $r\geq 2$ if $p$ is odd and $r\geq 3$ if $p=2$. As in Section \ref{sec:omega} let  $\mathfrak{X}_\infty$ be the $T$-adic formal scheme
$\mathfrak{X}_\infty:=Y_{\F_p} \otimes A_0$ and let $\fX_{r,\infty}$ be the $T$-adic formal scheme over $\mathfrak{X}_\infty$ representing the functor associating
to every $A_0$-algebra $T$-adically complete $R$ the set of equivalence classes of pairs $(f, \eta)$, where $f\colon {\rm Spf}(R)\lra \mathfrak{X}_\infty$ and
$\eta\in {\rm H}^0\bigl({\rm Spf}(R), f^\ast(\omega^{(1-p)p^{r+1}})\bigr)$ such that
$$
\eta\cdot \mathrm{Hdg}^{p^{r+1}}=\alpha.
$$

Thanks to \cite[\S 4.3]{halo_spectral} for every  $n$ we have a natural formal scheme $\fIG_{n,r,\infty}\lra \fX_{r,\infty}$ given as the normalization of the Igusa
tower of level $n$ over the adic fiber of $ \fX_{r,\infty} $. By loc. cit. we have a canonical subgroup $H_n$ over $\fIG_{n,r,\infty}$ and a section $\psi_n\colon
\Z/p^n\Z \to H_n^\vee$ of its Cartier dual, which is an isomorphism over the ordinary locus of $\fX_{r,\infty}$.

Let $\fIG_{\infty,r,\infty}$ be the projective limit $\lim_n \fIG_{n,r,\infty}$ in the category of $T$-adic formal schemes. Thanks to \cite[Prop. 4.2]{halo_spectral} we
have a canonical section $\psi\colon \Q_p/\Z_p\to \mathrm{colim} H_n^\vee$. Proceeding as in Section \ref{sec:new} we have a sheaf $\Hsharp$ and an exact sequence
$$
0\lra \Omega_E\lra \barHsharp\lra \udelta^p\omega_E^{-1}\lra 0.
$$
with $\Omega_E$ an invertible sheaf over $\fIG_{\infty,n,\infty}$, endowed with a canonical generator $\gamma$ of $\Omega_E$ defined as the image of $1$ via  the
map $\Z_p \to \lim_n H_n^\vee$ provided by $\psi$, the limit of the maps ${\rm dlog}\colon H_n^\vee \to \omega_{H_n}$ and the isomorphism $\omega_E \to \lim_n \omega_{H_n}$ defined by the inclusions $H_n\subset E$ using (\ref{eq:dlog})..

\begin{definition}\label{def:V0charp} We define the formal scheme
$$
\pi\colon \bV_0(\barHsharp,\gamma) \lra \bV_0(\Omega_E,s)\lra \fIG_{\infty,r,\infty}\lra \fX_{r,\infty},
$$
requiring that for every formal scheme $\rho\colon S \to \fIG_{\infty,r,\infty}$ the $S$-valued points of $\bV_0(\barHsharp,\gamma)$
(resp.~$\bV_0(\Omega_E,\gamma)$) over $\rho$ are the $\cO_S$-linear homomorphisms $v\colon \rho^\ast\bigl(\barHsharp \bigr)\to \cO_S$ (resp.~$v\colon
\rho^\ast\bigl(\Omega_E \bigr)\to \cO_S$) such that $v(\gamma)=1$.
\end{definition}

Notice that in this case the map $\bV_0(\Omega_E,\gamma)\lra \fIG_{\infty,r,\infty}$ is an isomorphism as $\gamma$ is a generator of $\Omega_E$. Denote by $f$ the morphism from these formal schemes to
$\fX_{r,\infty}$. As in Definition \ref{def:newomegak} we set  $$\fw_\infty^{k,0}:=f_{\ast}\bigl(\cO_{\bV_0(\Omega_E, s)}  \bigr)\bigl[k^0\bigr],$$where $k$ is the universal
weight and $k^0:=k k_f^{-1}$. As $\bV_0(\Omega_E,\gamma)\cong \fIG_{\infty,r,\infty}$ this coincides with the sheaf defined in \cite[Thm. 4.1]{halo_spectral} in terms of the structure sheaf  of $ \fIG_{\infty,r,\infty}$.
Twisting by the sheaf $\fw^{k_f}_\infty$ as in Definition \ref{defi:wnrI}, which is invertible by \cite[\S 4.4.2]{halo_spectral}, we get an invertible sheaf
$\fcw_\infty^k$ over $\fX_{r,\infty}$.

Proceeding as in Section \ref{sec:wk} we have an action of $\Z_p^\ast$ on $\bV_0(\barHsharp,\gamma)$ and one  defines sheaves
$\bW^0_{k,\infty}$ and $\bW_{k,\infty}$ as in Definition \ref{def:omegak} with filtrations  $\Fil_\bullet\bW^0_{k,\infty}$ and $\Fil_\bullet\bW_{k,\infty}$. We will see in \S \ref{sec:ordinaryandqexp} that if we invert $T$, or equivalently if we restrict to the ordinary locus, the  filtration is canonically split.

\begin{theorem}\label{thm:bWkinfty} The following hold

\begin{itemize}

\item[i.] $\Fil_h\bW^0_{k,\infty}$ and $\Fil_h\bW_{k,\infty}$ are locally free $\cO_{\fX_{r,\infty}}$-modules;

\item[ii.] $\bW^0_k$ and $\bW_k$ are the $\alpha$-adic completions of $\lim_h \Fil_h\bW^0_{k,\infty}$, respectively $\lim_h \Fil_h\bW_{k,\infty}$.

\item[iii.] $\Fil_0\bW^0_{k,\infty}\cong \fw^{k,0}_\infty$ and $ \mathrm{Gr}_h\bW^0_k \cong \fw^{k,0}_\infty \otimes_{\cO_{\mathfrak{X}_{r,I}}}{\rm Hdg}_E^{h}\omega_{E}^{-2h}$.

\item[iv.] $\Fil_0\bW_{k,\infty}\cong \fcw^{k}_\infty$ and $ \mathrm{Gr}_h\bW_{k,\infty} \cong \fcw^{k}_\infty \otimes_{\cO_{\mathfrak{X}_{r,I}}}
{\rm Hdg}_E^h\omega_{E}^{-2h}$.

\item[v.] The sheaves $\bW^0_{k,\infty}$ and $\Fil_h\bW^0_{k,\infty}$ are the base changes to $\fX_{r,\infty}$ of the sheaves $\bW^0_{k,[p,\infty]}$ and
$\Fil_h\bW^0_{k,[p,\infty]}$   over $\fX_{r,[p,\infty]}$ of  Theorem \ref{thm:descentbWk}, respectively.

\end{itemize}
\end{theorem}

\begin{proof} Let $\rho\colon S=\Spf(R)\lra \fX_{r,\infty}$ be an affine open formal subscheme
such that $\rho^\ast(\omega_E)$ is a free $R$-module with $R$-basis element $\omega$.  Let $S_\infty:=\Spf(R_\infty)$ the corresponding open of
$\fIG_{\infty,r,\infty}$ over $S$. Write $\rho^\ast\bigl(\barHsharp\bigr)=R_\infty s \oplus  R_\infty e$ (here $s$ is the canonical section of $\Omega_E$ over $S$
defined above and $e$ is a generator of $\udelta\omega_E^{-1}$ over $\Spf(R_\infty)$). In this case $\rho^\ast\Bigl(\pi_{\ast}\bigl(\cO_{\bV_0(\barHsharp,
\gamma)}\bigr)\Bigr)=R_\infty\langle Y\rangle$, with $Y$ is the dual of the generator $e$ of $ \udelta^p\omega_E^{-1}\vert_S$. In particular if we let
$S_1:=\Spf(R_1)$ be the inverse image of $S$ in $\fIG_{1,r,\infty}$ and we choose $e$ the generator $\delta^p \omega^{-1}$ of $\udelta^p \omega_E^{-1}$ over $S_1$,
then we have the following analogue of Lemma \ref{lemma:localwkappa}:
$$\bW^0_{k,\infty}(S)=\rho^\ast\Bigl(\pi_{\ast}\bigl(\cO_{\bV_0(\barHsharp, \gamma)}\bigr)[k^0]\Bigr)=s^k\vert_S R\langle Y'\rangle$$
where $s^k$ is the generator of $\fw_\infty^{k,0}\vert_S$ defined by $s$ via \cite[Thm. 4.1]{halo_spectral} and $Y'=\frac{Y}{u}$ with $u\in R_1$ such that $\lambda
\ast u= \frac{\lambda \ast \udelta}{\udelta} u$ for every $\lambda \in (\Z/p\Z)^\ast$. The filtration $\Fil_\bullet \bW^0_{k,\infty}(S)$ is the $Y'$-adic
filtration. Using this local description Claims (i)--(iv) of the Theorem follow.

We next sketch the proof of  Claim (v). Write $I:=[p,\infty]$. One introduces auxiliary objects; consider the anticanonical tower $h_r\colon \fX_{\infty,I} \to
\fX_{r,I}$ and the Igusa tower $\fIG_{\infty,\infty,[p,\infty]}\to \fX_{\infty,I}$ as $T$-adic formal schemes.  Over $\fIG_{\infty,\infty,[p,\infty]}$ the
pull--back of $\Omega_E$ admits a canonical generator $\mathrm{HT}^{\rm un}$, see \cite[\S 6.5]{halo_spectral}. This allows to define $\bV_0^\infty(\Hsharp,
\mathrm{HT}^{\rm un})$ over $\fIG_{\infty,\infty,[p,\infty]}$ as in Definition \ref{def:V0charp} and hence sheaves $\bW_{k,[p,\infty]}^{\rm perf,0}$ over
$\fX_{\infty,I}$. Arguing as in \cite[Prop. 6.4]{halo_spectral} one gets that $\bW_{k,[p,\infty]}^{\rm perf,0}$ is endowed with a filtration $\Fil_\bullet
\bW_{k,[p,\infty]}^{\rm perf,0}$ by locally free sheaves such that $\mathrm{Gr}_h \bW_{k,[p,\infty]}^{\rm perf,0}\cong \fw_I^{\rm perf}
\otimes_{\cO_{\mathfrak{X}_{\infty,I}}} {\rm Hdg}_E^{h}\omega_{E}^{-2h}$. Here $\fw_I^{\rm perf,0}=f_\ast\bigl(\cO_{\bV_0^\infty(\Omega_E, \mathrm{HT}^{\rm
un})}\bigr)[(k^0){-1}]$ can be identified with $h_r^\ast(\fw_k)$ by \cite[Prop. 6.6]{halo_spectral}.

Note that for every closed interval $J\subset [p,\infty)$ and every integer $n$ adapted to $J$ we have a natural commutative diagram
$$\begin{matrix} \bV_0^\infty(\Hsharp, \mathrm{HT}^{\rm un}) & \lra & \bV_0(\Hsharp, s) \cr \big\downarrow & &\big\downarrow \cr
\fIG_{\infty,\infty,J} & \lra & \fIG_{n,r, J} \cr \big\downarrow & &\big\downarrow \cr \fX_{\infty,J} & \stackrel{h_r}{\lra} & \fX_{r,I}\cr
\end{matrix}$$where $\bV_0(\Hsharp, s)$ is as in Section \ref{sec:wk}. This defines a morphism $h_r^\ast\bigl(\bW^0_{k,J}\bigr)\lra \bW_{k,J}^{\rm perf,0}$ that
respects filtrations and induces the isomorphism $$h_r^\ast(\fw^{k,0})  \otimes_{\cO_{\mathfrak{X}_{\infty,I}}} {\rm Hdg}_E^h\omega_{E}^{-2h} \cong \fw_J^{\rm
perf,0} \otimes_{\cO_{\mathfrak{X}_{\infty,I}}} {\rm Hdg}_E^h\omega_{E}^{-2h}$$on graded pieces and, hence, it is an isomorphism, also on the filtrations. Arguing as
in the end of the proof of \cite[Thm. 6.4]{halo_spectral} one concludes that the sheaf $\bW_{k,[p,\infty]}^{\rm perf,0}$ descends to the sheaf
$\bW^0_{k,[p,\infty]}$ over $\fX_{r,[p,\infty]}$ defined in Definition \ref{def:omegak} (for $r\geq 3$ if $p$ is odd and $r\geq 5$ if $p=2$).

By construction we also have a commutative diagram $$\begin{matrix} \bV_0^\infty(\Hsharp, \mathrm{HT}^{\rm un})_\infty & \lra & \bV_0(\barHsharp, \gamma) \cr
\big\downarrow & &\big\downarrow \cr \fIG_{\infty,\infty,\infty} & \lra & \fIG_{\infty,r,\infty} \cr \big\downarrow & &\big\downarrow \cr \fX_{\infty,\infty} &
\stackrel{h}{\lra} & \fX_{r,\infty}\cr
\end{matrix}$$where $ \bV_0^\infty(\Hsharp, \mathrm{HT}^{\rm un})_\infty $ is the restriction of $\bV_0^\infty(\Hsharp, \mathrm{HT}^{\rm un})$
to $\fIG_{\infty,\infty,\infty}$ and $\bV_0(\barHsharp, \gamma) $ is as defined in \ref{def:V0charp}. Note that $h^\ast\bigl(\fw^{k,0}_I \bigr)\cong \fw_I^{\rm
perf,0}$ by \cite[Prop. 6.8]{halo_spectral}.

This commutative diagram provides a morphism from $h^\ast\bigl( \bW^0_{k,\infty}\bigr)$ to the restriction $\bW_\infty^{\rm perf,0}$ of $\bW_{k,[p,\infty]}^{\rm
perf,0}$ to $\fX_{\infty,\infty}$, that respects the filtrations and induces an isomorphism on graded pieces thanks to the cited result of \cite{halo_spectral}.
Hence it is an isomorphism. On the other hand we know that $\bW_\infty^{\rm perf,0}$  descends to the restriction of $\bW^0_{k,[p,\infty]}$ to $\fX_{r,\infty}$. By
the uniqueness of the descent -- in this case defined by taking $\Z_p^\ast$-invariant --  the claim follows.

\end{proof}

\subsection{The Gauss-Manin connection on  $\bW_k$.}\label{sec:GMwk}

Let $r$, $n$, $I$, $(\Lambda_I^0,\alpha)$ be as in the previous
sections (see Definition \ref{defi:wnrI})  with the property that $n\ge 2$ and $I\subset [0,\infty)$. The restriction of $k$
to $1+p^n\Z_p$ is analytic so there is $u_I\in p^{1-n}\Lambda_I^0$ such
that $t^k:=k(t)=\exp\bigl(u_I\log(t)\bigr)$ for all $t\in
1+p^n\Z_p$.

Consider the morphism of adic spaces $\mathcal{IG}_{n,r,I}'\lra \mathcal{IG}_{n,r,I} $ defined by the trivializations $E[p^n]^\vee\cong (\Z/p^n\Z)^2$ compatible with
the trivializations $H_n^\vee\cong \Z/p^n \Z$. Let $\fIG_{n,r,I}'\lra \fIG_{n,r,I} $ be the normalization as in \S\ref{sec:omega} and let $h_n\colon \fIG_{n,r,I}\to \fX_{r,I}$ be the natural morphism.  It then follows from Proposition \ref{prop:nablasharp}  and from Lemma \ref{lem:Griffiths} that over $\fIG_{n,r,I}'$ the sheaf
$\bW^0_k$ admits an integrable connection relatively to $\Lambda_I^0$ for which $\Fil_\bullet \bW^0_k$ satisfies Griffiths' tranversality.

\begin{theorem}
\label{theorem:griffith} The connection on the pull-back of $\bW^0_k$ over $\fIG_{n,r,I}'$  descends to an integrable connection
$$\nabla_k\colon h_n^\ast\bigl(\bW^0_k\bigr) \to
 h_n^\ast\bigl(\bW^0_k\bigr)\widehat{\otimes}_{\cO_{\mathfrak{IG}_{n,r,I}}} \Omega^1_{\mathfrak{IG}_{n,r,I}/\Lambda_I^0}[1/\alpha]$$over
$\mathfrak{IG}_{n,r,I}$ for which $ h_n^\ast\bigl(\Fil_\bullet\bW^0_k\bigr)$ satisfies Griffiths' tranversality. In particular it induces a connection
$$\nabla_k\colon \bW^0_k \to
\bW^0_k\widehat{\otimes}_{\cO_{\mathfrak{X}_{r,I}}} \Omega^1_{\fX_{r,I}/\Lambda_I^0}[1/\alpha]$$such that the induced $\cO_{\fX_{r,I}}$-linear map on th $h$ graded piece

$$
\mathrm{Gr}_h(\nabla_k)\colon \mathrm{Gr}_h(\bW^0_k)[1/\alpha]\lra \mathrm{Gr}_{h+1}(\bW^0_k)\otimes \Omega^1_{\fX_{r,I}/\Lambda_I^0}[1/\alpha]
$$is an isomorphism times $u_I-h$ and, in particular, it is an isomorphism if and only if $u_I-h$ is invertible in $\Lambda_I^0[1/\alpha]$.

It also induces a connection $\nabla_k\colon \bW_k[1/\alpha] \to \bW_k\widehat{\otimes}_{\cO_{\mathfrak{X}_{r,I}}}
\Omega^1_{\mathfrak{X}_{r,I}/\Lambda_I^0}[1/\alpha]$ that satisfies Griffiths' tranversality and such that the induced map on the $h$ graded piece is an isomorphism times
$u_I-h$.

If $k\in \N$ is an integral weight, the identification ${\rm Sym}^k\bigl({\rm H}_E\bigr)[1/\alpha]\vert_{\fX_{r,I}}={\rm Fil}_k(\bW_k)[1/\alpha]$ of Theorem
\ref{thm:descentbWk} is compatible with the connections, considering on ${\rm Sym}^k\bigl({\rm H}_E\bigr)$ the Gauss-Manin connection.

\end{theorem}

\begin{proof} The proof of the first part is local on $\mathfrak{X}_{r,I}$ and will follow from the computations of \S\ref{section:explicitnablak}.
The statement for the descent to $\fX_{r,I}$ after inverting $\alpha$ follows from Lemma \ref{lemma:CokerOmega} taking $(\Z/p^n\Z)^\ast$-invariants.

For the second part of the Theorem recall from Definition \ref{def:omegak} that $\bW_{k}= \bW^0_{k}\otimes_{\cO_{\fX_{r,I}}} \fw^{k_f}$ and
$\fw^{k_f}=\left(g_{i,\ast} \bigl(\cO_{\mathfrak{IG}_{i,r,I}}\bigr)\otimes_{\Lambda^0} \Lambda\right) \bigl[k_{I,f}^{-1}\bigr] $ by Definition \ref{defi:wnrI}. As
$\Omega^1_{\mathfrak{IG}_{i,r,I}/\fX_{r,I}} $ is annihilated by a power of $\alpha$, the universal derivation $g_{i,\ast} \bigl(\cO_{\mathfrak{IG}_{i,r,I}}\otimes_{\Lambda^0} \Lambda\bigr) \to g_{i,\ast} \bigl(\Omega^1_{\mathfrak{IG}_{i,r,I}/\Lambda_I^0}\otimes_{\Lambda^0} \Lambda\bigr)$ defines
a connection on $\fw^{k_f}[1/\alpha]$. This connection and the connection on $\bW^0_k$ induce a connection on the tensor product $\bW_{k}[1/\alpha]$.

For the third part the local expression of the connections is described in \S \ref{section:explicitnablak} and directly implies that the given identification is
compatible with the connections.

\end{proof}

The multiplication structure on $\pi_\ast\bigl(\cO_{\bV_0(\Hsharp)}\bigr)$ induces a multiplication $\bW^0_k \otimes_{\cO_{\mathfrak{X}_{r,I}}} \bW^0_2 \to
\bW^0_{k+2}$. Since $\Fil_0 \bW^0_2=\Omega^{\otimes 2}_E$ we have a morphism $\bW^0_k\otimes_{\cO_{\mathfrak{X}_{r,I}}} \Omega_E^{\otimes 2} \to \bW^0_{k+2}$ which
is easily checked, using Lemma \ref{lemma:localwkappa}, to be an isomorphism, preserving the filtrations. We have an identification
$\Omega^1_{\mathfrak{X}/\Lambda_I^0}\cong \omega_E^{\otimes 2}$ via Kodaira-Spencer.  Thanks to Lemma \ref{lemma:CokerOmega} we also have a positive integer $c_n$,
depending on $n$, such that $\mathrm{Hdg}^{c_n}$ annihilates $\Omega^1_{\fIG_{n,r,I}/\fX} $, i.e., $\mathrm{Hdg}^{c_n} \Omega^1_{\fIG_{n,r,I}/\Lambda_I^0}$ is
contained in the pull-back of $\Omega^1_{\fX/\Lambda_I^0}$ to $\fIG_{rn,r,I}$. In conclusion replacing $c_n$ with $c_n+3+c_i$, with $i=1$ for $p$ odd and $i=2$ for
$p=2$ and using the explict formula for the connection over $\fIG_{n,r,I}$ provided in (\ref{eq:nablak}),  we can write the Gauss-Manin connections as morphisms:

$$\bW^0_k \lra \Bigl(\frac{1}{p^{n-1}\mathrm{Hdg}^{c_n}} \Bigr) \cdot \bW^0_{k+2} $$and $$\bW_k \lra \Bigl(\frac{1}{p^{n-1}\mathrm{Hdg}^{c_n}} \Bigr) \cdot \bW_{k+2},$$
here the factor $p^{1-n}$ comes from the fact that $u\in p^{1-n} \Lambda_I^0$.

\begin{remark} One could refine Theorem \ref{theorem:griffith}
in order to control the denominators $c_n$ of $\mathrm{Hdg}$, and hence of $\alpha$, appearing in the connection of $\bW^0_k$ over $\fX_{r,I}$ in terms of the
integer $n$, adapted to $I$. Unfortunately due to Lemma \ref{lemma:CokerOmega} and the more detailed analysis of the inverse different of $\mathfrak{IG}_{n,r,I}\to
\fX_{r,I} $ in Lemma \ref{lemma:kerloc} such powers grow as $p^n$. In particular if we take the limit over intervals $[p,p^h]$ for $h\to \infty$ we find a
connection with unbounded denominators in $\alpha=T$.

The conclusion is that the connection $\nabla$ can not be iterated over the whole weight space, including $\infty$, using the methods of Section
\ref{sec:IterateManin}.

\end{remark}

\subsubsection{Explicit, local calculation of the connection $\nabla_k$.}
\label{section:explicitnablak}

Let $\rho\colon S=\Spf(R)\lra \fIG'_{n,r,I}$ be a morphism of
formal schemes over $\Spf(\Lambda_I^0)$. Assume that the composite
of $\rho$ with the projection to the modular curve $\mathfrak{X}$
factors through some open affine neighborhood of $\mathfrak{X}$
over which $\mathrm{H}_E$ is free with bases $\{\omega,\eta\}$
where $\omega$ spans $\omega_E$. Let $\delta $ be the generator
$\Delta(E/R, \omega)$ of $\rho^\ast(\udelta)$ of Remark
\ref{rmk:delta}. By definition of $\Hsharp$, the $R$-modules $\rho^\ast\bigl(\mathrm{H}_E\bigr)$
and  $\rho^\ast\bigl(\mathrm{H}_E^\sharp\bigr)$ are free of rank
$2$ with bases $\{\rho^\ast(\omega),\rho^\ast(\eta)\}$ and
$\{f,e\vert f:=\delta \omega, e:=\delta^p \eta\}$ respectively.  We
also deduce that  $\rho^\ast(\ubeta_n)$ is a principal ideal of
$R$ with generator $\beta_n$ and that the given $R$-basis $\{f,
e\}$ of $\rho^\ast\bigl(\mathrm{H}_E^\sharp\bigr)$ satisfies
$f(\mbox{mod }\beta_n R)=\rho^\ast\bigl({\rm dlog}(P_n)\bigr)$.

Let $\cP^{(1)}_{R/\Lambda_I^0}\subset
\Spf(R\widehat{\otimes}_{\Lambda_I^0} R)$ be the closed immersion
defined by the square of the ideal ${\rm I}(\Delta)$ associated to
the diagonal embedding $\Delta\colon S\hookrightarrow
S\times_{\Lambda_I^0} S$. Thanks to Proposition
\ref{prop:nablasharp} the $R$-module
$\rho^\ast\bigl(\mathrm{H}_E^\sharp\bigr)$ admits an integrable
connection $\nabla^\sharp$ that can be expressed via
Grothendieck's formalism (see in \S\ref{sec:fvvconenction}) as an
isomorphism $\epsilon^\sharp\colon
j_2^\ast\bigl(\rho^\ast(\mathrm{H}_E^\sharp)\bigr) \cong
j_1^\ast\bigl(\rho^\ast(\mathrm{H}_E^\sharp)\bigr)$. Let
$$A:=\left(
\begin{array}{cc} a & b \\ c & d
\end{array} \right) \in \mathrm{GL}_2\bigl(\cP^{(1)}_{R/\Lambda_I^0}\bigr)$$be the inverse of the
matrix of $\epsilon^\sharp$ with respect to the basis $\{f\otimes 1, e\otimes 1\}$ of $j_2^\ast\bigl(\rho^\ast(\mathrm{H}_E^\sharp)\bigr)$ and $\{1\otimes f,
1\otimes e\}$ of $j_1^\ast\bigl(\rho^\ast(\mathrm{H}_E^\sharp)\bigr)$.

\begin{lemma}
\label{lemma:propM}
We have

\begin{itemize}

\item[a)] $a=1+a_0$, $d=1+d_0$ with $a_0,b,c,d_0\in {\rm
I}(\Delta)$ and so $a_0^2=b^2=c^2=d_0^2=0$ in
$\cP^{(1)}_{R/\Lambda_I^0}$.

\item[b)] interpreting $a_0$, $b$, $c$, $d_0\in {\rm
I}(\Delta)/{\rm I}(\Delta)^2\cong \Omega^1_{R/\Lambda_I^0}$ we
have that $a_0$, $b$, $c$, $d_0\in \frac{1}{\mathrm{Hdg}} \cdot
\rho^\ast\bigl(\Omega^1_{\mathfrak{X}/\Z_p}\bigr)$ and $-\mathrm{Hdg}\cdot  c$ is  the
Kodaira-Spencer  differential $\mathrm{KS}(\omega,\eta)$
associated to the local basis $\{\omega, \eta\}$ of
$\mathrm{H}_E$.
\end{itemize}

\end{lemma}

\begin{proof}
For a) as $A\bigl(\mbox{ mod }{\rm I}(\Delta)\bigr)=\mathrm{Id}$
we have that $a_0=a-1,b,c,d_0=d-1\in {\rm I}(\Delta)$. Moreover
${\rm I}(\Delta)^2=0$ in $\cP^{(1)}_{R/\Lambda_I^0}$.

For b) recall that the connection $\nabla^\sharp$ is uniquely
determined by the Gauss-Manin connection $\nabla$ on
$\rho^\ast\bigl(\mathrm{H}_E\bigr) $ via the inclusions $\Hsharp \subset \mathrm{H}_E$. Also $f=\delta
\rho^\ast(\omega)$ and $e=\delta^p \rho^\ast(\eta)$ with $\omega$ a
generator of $\omega_E$ over some open affine subscheme $U\subset
\mathfrak{X}$ and $\eta$ a  generator of the quotient
$\mathrm{H}_E/\omega_E=\omega_E^\ast$ over $U$. In particular
$\delta^{p-1}=\widetilde{\mathrm{Ha}}(E/R,\omega)=\rho^\ast(u)$
for a  section $u\in {\rm H}^0\bigl(U,\cO_{U}\bigr)$ so
that $$d \rho^\ast(u)=d \delta^{p-1}=(p-1) \delta^{p-2} d
\delta=(p-1) \delta^{p-1} {\rm dlog}(\delta)=(p-1) \rho^\ast(u)
{\rm dlog}(\delta).$$ Hence, ${\rm dlog}(\delta)=(p-1)^{-1} {\rm
dlog}\bigl(\rho^\ast(u)\bigr)\in \frac{1}{\mathrm{Hdg}} \cdot
\rho^\ast\bigl(\Omega^1_{\mathfrak{X}/\Z_p}\bigr)$. The
Kodaira-Spencer isomorphism $${\rm KS}\colon \omega_E \to
\omega_E^\ast \otimes_{\cO_\mathfrak{X}}
\Omega^1_{\mathfrak{X}/\Z_p}$$obtained by restricting $\nabla$ to
$\omega_E\subset \mathrm{H}_E$ and then taking the projection onto
$\bigl(\mathrm{H}_E/\omega_E\bigr)\otimes_{\cO_\mathfrak{X}}
\Omega^1_{\mathfrak{X}/\Z_p}$ provides a basis element
$\Theta:={\rm KS}(\omega,\eta)$ of $\Omega^1_{\mathfrak{X}/\Z_p}$
over $U$ characterized by the property that ${\rm
KS}(\omega)=\eta\otimes {\rm KS}(\omega,\eta)$. Write the
connection
$$
\begin{array}{ccccccccc}
\nabla(\omega)&=& m\, \omega\otimes \Theta &+& \eta \otimes \Theta \\
\nabla(\eta)&=& q \, \omega\otimes \Theta &+& r \, \eta\otimes \Theta
\end{array}
$$
with $m$, $q$, $r\in {\rm H}^0\bigl(U,\cO_\mathfrak{X}\bigr)$. Therefore we have, omitting $\rho^\ast$ for simplicity:
$$
\begin{array}{cccccccccc}
\nabla^\sharp(f)&=&\nabla(\delta\omega)&=& \delta \nabla(\omega)+ \delta \omega \otimes d\log(\delta) & = & \bigl(m+ \frac{d u}{(p-1)u}\bigr) f \otimes \Theta &+& \frac{1}{\delta^{p-1}} e
\otimes \Theta\\
\nabla^\sharp(e)&=&\nabla(\delta^p \eta)&=& \delta^p \nabla(\eta)+p \delta^p \eta \otimes d\log(\delta) & = & \delta^{p-1}q  f \otimes \Theta &+& \bigl(r+ p \frac{d u}{(p-1)u}\bigr) e
\otimes \Theta. \end{array}
$$
This proves the first statement and shows that $-\delta^{p-1} c=\Theta={\rm KS}(\omega,\eta)$, implying also the second statement.

\end{proof}

{\bf Proof of Theorem \ref{theorem:griffith}.} \enspace Let now $k\colon \Z_p^\ast\lra \Lambda_I^\ast, \bW^0_k,\nabla_k$ be as in the previous section.  Recall from
Lemma \ref{lemma:localwkappa} that
$$(g_n\circ \rho)^\ast(\bW^0_k)=\Bigl\{\sum_{n=0}^\infty a_nV^n(1+\beta_nZ)^{k}\ \vert  \ a_n\in R\mbox{ with } a_n\rightarrow 0\mbox{ and } V=\frac{Y}{1+\beta_nZ} \Bigr\}.$$
Moreover
$$j_i^\ast\bigl(\rho^\ast(\bW^0_k)  \bigr)=\Bigl\{\sum_{m=0}^\infty a_mV^m(1+\beta_nZ)^{k}\ \vert
\ a_m\in j_i^\ast(R)=\cP^{(1)}_{R/\Lambda_I^0}\mbox{ for each } m\ge 0, \mbox{ with }
a_m\rightarrow 0 \Bigr\}$$for $i=1$, $2$. Therefore $\epsilon_k$
is given by the action of the matrix
$A=\left(
\begin{array}{cc} a & b \\ c & d
\end{array} \right)$ on
$V^m(1+\beta_nZ)^{k}$, for $m\ge 0$. More precisely
$$
\epsilon_k\bigl(V^m(1+\beta_nZ)^{k})\bigr)=A\cdot\bigl(V^m(1+\beta_nZ)^{k}  \bigr)
=(a+cV)^k\Bigl(\frac{b+dV}{a+cV}  \Bigr)^m(1+\beta_nZ)^{k}=
$$
$$
=(a+cV)^{k-m}(b+dV)^m(1+\beta_nZ)^{k}.
$$
Let us recall that given $k$ there is a positive integer $n$ and an element $u\in p^{1-n} \Lambda_I^0$ such that $t^k:=k(t)=\exp\bigl(u\log(t)\bigr)$ for all $t\in 1+p^n\Z_p^\ast$. Using Lemma
\ref{lemma:propM} we can write: $(a+cV)^{k-m}=\exp\Bigl((u-m)\log\bigl(1+(a_0+cV)\bigr)\Bigr)=(u-m)(a_0+cV)$. On the other hand we have
$\bigl(b+dV\bigr)^m=\bigl(V+(b+d_0)V\bigr)^m=V^m+mV^{m-1}(b+d_0V)$, and therefore
$$
\epsilon_k\bigl(V^m(1+\beta_nZ)^k\bigr)=\Bigl(\bigl(1+md_0+(u-m)a_0   \bigr)V^m+mbV^{m-1}+(u-m)c V^{m+1}\Bigr)(1+\beta_nZ)^k.
$$
Thus we have

\begin{equation}\label{eq:nablak}\begin{array}{c} \nabla_k\bigl(V^m(1+\beta_nZ)^{k}\bigr)=\epsilon_k\bigl(V^m(1+\beta_nZ)^{k}\bigr)-V^m(1+\beta_nZ)^{k}= \cr
=\Bigl(mV^m\otimes d_0+(u-m)V^m\otimes a_0+mV^{m-1}\otimes
b+(u-m)V^{m+1}\otimes c\Bigr)\bigl((1+\beta_nZ)\otimes
1\bigr)^{k}\in \cr
 \in p^{1-n}(1+\beta_nZ)^{k}R\langle V\rangle\otimes_R\Omega^1_{R/\Lambda_I^0}=p^{1-n}\rho^\ast(\bW^0_k)\otimes_R\Omega^1_{R/\Lambda_I^0}.\end{array}
\end{equation}

Here the factor $p^{1-n}$ comes from the fact that $u\in p^{1-n} \Lambda_I^0$. In particular
$\nabla_k\bigl(V^m(1+\beta_nZ)^{k}\bigr)=(u-m)V^{m+1}(1+\beta_nZ)^{k}\otimes
c$ modulo $Y^m$. Since the map
$\rho^\ast\bigl(\Omega^1_{\mathfrak{X}/\Z_p}\bigr)\to
\Omega^1_{R/\Lambda_I^0}$   is an isomorphism after inverting
$\alpha$ due to by Lemma \ref{lemma:CokerOmega}, the second claim
of Theorem \ref{theorem:griffith} is proven as $\mathrm{Hdg} c$ is a generator
of $\rho^\ast\bigl(\Omega^1_{\mathfrak{X}/\Z_p}\bigr)$ due to
Lemma \ref{lemma:propM}.

\subsection{$q$-Expansions of sections of $\bW_k$ and nearly overconvergent modular
forms.}\label{sec:ordinaryandqexp}

Given a formal scheme $\mathfrak{S}\to \fX$ we will denote $\mathfrak{S}^{\rm ord}\subset \mathfrak{S}$ to be the open formal subscheme defined by the inverse image
of the ordinary locus of $\fX$. In particular $\fIG_{n,r,I}^{\rm ord}$ is the $n$-th layer of the Igusa tower of $\fX^{\rm ord}$. Over $\fIG_{n,I}^{\rm ord}$ we
have $\Hsharp=\mathrm{H}_E=\omega_E \oplus \omega_E^{-1}$ as the Hodge filtration splits canonically, via the so called {\it unit root decomposition}: one has a lift of Frobenius on $ \fX^{\rm ord}$ and the universal semiabelian scheme $E$ defined by taking the quotient by the canonical subgroup $H_1$ and $\omega_E^{-1}$ is identified with the submodule of $\mathrm{H}_E$ on which such isogeny is an isomorphism. In particular we have a morphism $\bV_0(\Hsharp, s)^{\rm
ord}\lra \bV(\omega_E^\ast)^{\rm ord}$ by \S\ref{sec:vbfil} and the induced morphism
$$\bV_0(\Hsharp, s)^{\rm ord}\lra \bV_0\bigl(\omega_E, s\bigr)\times_{\fIG_{n,I}^{\rm ord}}
\bV\bigl(\omega_E^{-1}\bigr)$$is an isomorphism of formal schemes. Recall that we have divided the universal weight $k\colon \Z_p^\ast\to \Lambda^\ast$ into the
product $k^0 \cdot k_f$ where  $k_f$ is the finite part and $k^0\colon \Z_p^\ast \to (\Lambda^0)^\ast$.

Note that over $\fIG_{n,r,I}^{\rm ord}$  the image fo the universal section of $H_n^\vee$ defines via the map $d\log$ a basis element $s$ of $\omega_E/p^n \omega_E$. In particular,  as we are assuming that $k$ restricted to $1+p^n\Z_p$ is analytic, 
and if $\pi\colon \bV\bigl(\omega_E,s\bigr)^{\rm ord} \to \fX^{\rm ord}$ is the canonical projection, then the global sections of $\omega_{E,\fX^{\rm
ord}}^{k^0}:=\pi_\ast\bigl(\cO_{\bV\bigl(\omega_E,s\bigr)^{\rm ord}} \bigr)[k^0]$ over $\fX^{\rm ord}$ coincide with Katz's $p$-adic modular forms of weight $k^0$.
The space of Katz's $p$-adic modular forms of weight $k$ is then obtained by taking the global sections of  the tensor product $\omega_{E,\fX^{\rm
ord}}^{k}:=\omega_{E,\fX^{\rm ord}}^{k^0}\otimes_{\cO_{\fX^{\rm ord}}} \fw^{k_{I,f}}\vert_{\fX^{\rm ord}}$ (see  Definition \ref{defi:wnrI} for $\fw^{k_{I,f}}$).
Denote by $\bW_k^{\rm ord,0}$, resp.~$\bW_k^{\rm ord}$ the space $\bW^0_k\vert_{\fX^{\rm ord}}$, resp.~$\bW_k\vert_{\fX^{\rm ord}}$. We  obtain a canonical
decomposition

\begin{equation}\label{eq:FilnWkord}
\bW^0_k\vert_{\fX^{\rm ord}}\cong \omega_{E,\fX^{\rm ord}}^{k^0}\widehat{\otimes}_{\cO_{\fX^{\rm ord}}} \mathrm{Sym}\bigl(\omega_E^{-2}\bigr),\qquad
\bW_k\vert_{\fX^{\rm ord}}\cong \omega_{E,\fX^{\rm ord}}^{k}\widehat{\otimes}_{\cO_{\fX^{\rm ord}}} \mathrm{Sym}\bigl(\omega_E^{-2}\bigr),
\end{equation}

where $\mathrm{Sym}\bigl(\omega_E^{-2}\bigr)$ is the symmetric
algebra and $\widehat{\otimes}$ is the $\alpha$-adic completed
tensor product. In particular we get morphisms
$$\bW^0_k \stackrel{\rho}{\lra} \bW_k^{\rm ord,0}
\stackrel{\Phi}{\lra}\omega_{E,\fX^{\rm ord}}^{k^0 },$$and upon twisting with $\fw^{k_{I,f}}$
$$\bW_k \lra \bW_k^{\rm ord} \lra \omega_{E,\fX^{\rm ord}}^k
$$which provide a splitting of the first step of the filtration $\Fil_0 \bW^0_k$, resp.~$\Fil_0 \bW^k$ and that,  upon taking global sections, defines a  projection from the global sections of $\bW_k$ to the weight $k$ $p$-adic modular forms of Katz.

\begin{definition}\label{def:qexp} Using the $q$-expansion map for
Katz $p$-adic modular forms at a given unramified cusp we obtain the ``$q$-expansion map" which is the composition of the following morphisms: $${\rm
H}^0\bigl(\fX_{r,I},\bW_k\bigr) {\lra} {\rm H}^0\bigl(\fX^{\rm ord},\bW_k^{\rm ord}\bigr) \lra {\rm H}^0\bigl(\fX^{\rm ord},\omega_{E,\fX^{\rm ord}}^k\bigr) \lra
\Lambda_I(\!(q)\!).$$\end{definition}

We can now give the definition of nearly overconvergent modular forms of weight $k$.

\begin{definition}\label{def:overconv}
Let $g$ be a Katz $p$-adic modular form of weight $k$. We say that $g$ is {\bf nearly overconvergent} if there exists an $r$  compatible with the interval $I$
determined by $k$ such that  $g$ is in the image of ${\rm H}^0\bigl(\fX_{r,I},\bW_k\bigr)$ or equivalently if its $q$-expansion in $\Lambda_I[\![q]\!]$ is the
$q$-expansion of an element of ${\rm H}^0\bigl(\fX_{r,I},\bW_k\bigr)$.
\end{definition}

\begin{remark}\label{rmk:overconv}
Several authors have already introduced the notion of nearly overconvergent modular forms of finite degree,  notably \cite{HX}, \cite{darmon_rotger} and especially \cite{UNO} and \cite{zheng}. 
Their definitions provide alternative sheaf theoretic constructions of the  sheaves ${\rm Fil}_\bullet\bW_k$  over $\cX_{r,[0,\infty)}$ but neither did they work with the whole of $\bW_k$ formally (i.e. integrally) nor did they define the connection on it. As it will become clear later,
cf.~Theorem \ref{theorem:mostgeneral} and Proposition \ref{prop:nablas},  the definition of the whole $\bW_k$ is necessary if one wants to $p$-adically interpolate powers of the
Gauss-Manin connection. This is necessary in order to define triple product $L$-functions.

\end{remark}

We make the $q$-expansion map more explicit by working with the Tate curve.  Consider the Tate curve $E={\rm Tate}(q^N)$ over $\Spf(R)$ with
$R=\Lambda_I^0(\!(q)\!)$ and fix basis $\bigl(\omega_{\rm can}, \eta_{\rm can}:=\nabla(\partial)(\omega_{\rm can})\bigr)$ of ${\rm H}_E$, where $\partial$ is the
derivation dual to $\displaystyle {\rm KS}(\omega_{\rm can}^2)=\frac{dq}{q}$, i.e., $\displaystyle \partial:=q\frac{d}{dq}$. Let us remark that the canonical
subgroup $H_{E,n}$ of order $p$ of $E$ is isomorphic to $\mu_{p^n}$ and therefore its dual is isomorphic to $\Z/p^n\Z$, i.e., $\fIG_{n,r,I}$ over $\Spf(R)$ is
isomorphic to $\Spf(R)$. Hence if we denote by $\bW^0_k(q)$ the module $\bW^0_k$ for the Tate curve, we have a description of this $R$-module using the given basis
as described in Section \S\ref{section:explicitnablak}:  $\bW^0_k(q)=R\langle V\rangle (1+pZ)^{k}$ and, if we set $V_{k,n}(q):=Y^n(1+pZ)^{k-n}$, then
$\Fil_h \bW^0_k(q)=\sum_{i=0}^h R V_{k,i}(q)$.   The $q$-expansion map corresponds to the projection $\bW^0_k(q)\to R$ sending $\sum_i a_i V_{k,i}(q)\mapsto a_0$
and similarly twisting with $\fw^{k_{I,f}}$.

\subsection{The $U$-operator}\label{sec:upoperator}

Considering the morphisms $p_1$, $p_2\colon \fX_{r+1,I} \to \fX_{r,I}$ defined  on the universal elliptic curve by  $E\mapsto E$ and $E \mapsto E':=E/H_1$.  Over $\fIG_{1,r+1,I}$ we have the isogeny $\lambda\colon E'\to E$, dual to the projection $E\to E'$.

\begin{proposition}\label{prop:cU}  The isogeny $\lambda$ defines morphisms of $\cO_{\fX_{r,I}}$-modules
$$\cU\colon p_{2,\ast} p_1^\ast\bigl(\bW^0_k\bigr)  \to p_{2,\ast} p_2^\ast\bigl(\bW^0_k\bigr) $$and
$$\cU\colon p_{2,\ast} p_1^\ast\bigl(\bW_k\bigr)  \to p_{2,\ast} p_2^\ast\bigl(\bW_k\bigr)$$ which commute with the Gauss-Manin connections  $\nabla_k$ of Theorem \ref{theorem:griffith} and preserve the filtrations
defined in Theorem \ref{thm:descentbWk}. Futhermore the induced map on the $m$-graded pieces of the filtration is $0$ modulo $\alpha^{[m/p]}$ with $\alpha=p$ if $I\subset [0,1]$ and $\alpha=T$ if $I\subset [1,\infty]$ and where $[m/p]$ the integral part of $m/p$.

\end{proposition}
\begin{proof} 
Assume first that $I\subset [0,\infty)$. Consider the morphisms $p_1$, $p_2\colon\fIG_{n+1,r+1,I} \to \fIG_{n,r,I}$ defined as above  on the universal elliptic curve by $E\mapsto E$ and $E \mapsto E':=E/H_1$
respectively. Over $\fIG_{n+1,r+1,I}$ the isogeny $\lambda\colon E'\to E$ induces a morphism of the canonical subgroups of level $n$ of $E'$ and $E$, which is an isomorphism on analytic fibers. It follows from Lemma \ref{lemma:fsharp} that the map induced by $\lambda$ on de Rham cohomology induces a morphism 
$\lambda^\sharp \colon {\rm H}_E^\sharp \lra {\rm H}_{E'}^\sharp$ which provides an isomorphism $f^\ast\colon \Omega_E\cong \Omega_{E'}$ and identifies marked sections.
Thanks to Proposition \ref{prop:functfilbV0},  we get a morphism $$\cU\colon p_{2,\ast} p_1^\ast\bigl(\bW^0_k\bigr)  \to p_{2,\ast} p_2^\ast\bigl(\bW^0_k\bigr)$$over $\fIG_{n+1,r+1,I}$, preserving the filtration and commuting with Gauss-Manin connections.

Let $\tau:=p/{\rm Hdg}_E^{p+1}=p/\udelta_E^{p^2-1} $. It follows from Lemma \ref{lemma:fsharp} that the map 
$\lambda^\sharp \colon {\rm H}_E^\sharp \lra {\rm H}_{E'}^\sharp$ gives an isomorphism $\lambda^\ast\colon \Omega_E\cong \Omega_{E'}$ and  identifies
${\rm H}_E^\sharp/\Omega_E={\rm Hdg}(E)^{\frac{p}{p-1}} \omega_E^\vee$ with $\tau \cdot  {\rm H}_{E'}^\sharp/\Omega_{E'}=\tau \cdot  {\rm Hdg}(E')^{\frac{p}{p-1}} \omega_{E'}^\vee$. 
Using  the description of the map on graded pieces provided in Proposition \ref{prop:functfilbV0} we conclude that $\cU$  on the $m$-graded piece of  $\Fil_\bullet \bW_k^0$ defines a map ${\rm Gr}_m \cU$ which is zero modulo $\tau^m$. By construction $\alpha/ {\rm Hdg}_E^{p^{r+2}} $ is a well defined section of  $ \fX_{r+1,I}$ and $p/\alpha\in \Lambda_I$. Since $r\geq 1$ by assumption, then $\tau^{pm} \subset  (\alpha/  {\rm Hdg}_E^{p^2} )^{pm}\subset \alpha^{m}  (\alpha/  {\rm Hdg}_E^{p^3} )^{m}$ we conclude that $\tau^{pm}$ is in the ideal generated by $\alpha^m$. 

The descent Theorem \ref{thm:descentbWk} provides the descent of  $\cU$ for $\bW^0_k$  to $\fIG_{1,r,I}$ with the claimed properties.
By twisting $\bW_k^0$ and its filtration with the sheaf $\fw^{k_{I,f}}$ as in Definition \ref{def:omegak} we get $\bW_k$ and its filtration and the claim for $\bW_k$ follows as well, considering the $\cU$ corespondence for $\fw^{k_{I,f}}$. The construction of $\cU$ extends also to the case that $\infty \in I$ by passing to limits as in Theorem \ref{thm:descentbWk} and we get a morphism preserving the filtration, which is zero on the $m$-graded piece modulo $\alpha^{[m/p]}$.

\end{proof}

As the morphism $p_2\colon \fX_{r+1,I} \to \fX_{r,I}$ is finite and flat of degree $p$ by \cite[Prop. 3.3]{halo_spectral} there is a well defined trace map with respect to $\cO_{\fX_{r,I}} \to p_{2,\ast}\bigl(\cO_{\fX_{r+1,I}}\bigr)$ and we get the definition of the operator
$U$ on global sections of $\bW_k$.

$$U\colon {\rm H}^0\bigl(\fX_{r,I},\bW_k\bigr) \stackrel{\cU \circ p_1^\ast}{\lra} {\rm H}^0\bigl(\fX_{r,I},p_{2,\ast} p_2^\ast\bigl(\bW_k\bigr)\bigr)
\stackrel{\frac{1}{p} {\rm Tr}_{p_2}}{\lra}  {\rm H}^0\bigl(\fX_{r,I},\bW_k\bigr)[p^{-1}].$$

\begin{proposition}\label{prop:upfiltration} Assume that $I\subset [0,1]$ and $\alpha=p$ or that $I\subset [1,\infty]$ and $\alpha=T$.  Then 
$U\bigl({\rm H}^0(\fX_{r,I}, \bW_k\bigr)\subset \frac{1}{\alpha} {\rm H}^0\bigl(\fX_{r,I},\bW_k  \bigr)$ and
 $U$ induces a map on $\displaystyle {\rm H}^0\bigl(\fX_{r,I}, \bW_k/{\rm Fil}_m(\bW_k)\bigr)$ which  is $0$  modulo $\alpha^{[m/p]-1} $ for $m\geq p$.

Moreover if $k\in \N$ is an integral weight, the identification ${\rm Sym}^k\bigl({\rm H}_E\bigr)[p^{-1}]\vert_{\fX_{r,I}}={\rm Fil}_k(\bW_k)[p^{-1}]$ of
Theorem \ref{thm:descentbWk} is compatible with the $U$ operators defined on the global sections ${\rm H}^0(\fX_{r,I}, \ -\ )$ of the two sheaves.

\end{proposition}
\begin{proof} 
The first part follows directly from Proposition \ref{prop:cU}  for $I\subset [0,1]$ and $\alpha=p$. In the case that $I\subset [1,\infty]$ and $\alpha=T$ it follows from loc.~cit.~and the result
  $  \frac{1}{p} {\rm Tr}_{p_2}\left( p_{2,\ast} \bigl(\cO_{\fX_{r+1,I}}\bigr)\right)\subset  \frac{1}{T}\cO_{\fX_{r,I}}$ proven in \cite[lemme 6.1 \& Cor. 6.2]{halo_spectral}. 
The last claim of the proposition is clear as the
$U$-operator on ${\rm H}^0\bigl(\fX_{r,I}, {\rm Sym}^k\bigl({\rm H}_E\bigr)\bigr)$ is defined in the same way using the universal isogeny $\lambda\colon E'\to E$.

\end{proof}

Using the proposition we get the following result on slope decompositions with respect to the $U$-operator (in the sense of \cite[\S 4]{ash_stevens}) and passing to the analytic adic space $\cX_{r,I}$ of $\fX_{r,I}$ we have:

\begin{corollary}\label{cor:FredholmUp}  
The operator $U$  on $ {\rm H}^0\Bigl(\cX_{r,I}, \bW_{k} \Bigr)$
admits a Fredholm determinant $P_I(k,X)\in\Lambda_I[\![X]\!]$ and
for every non-negative rational $h$ the group $ {\rm
H}^0\Bigl(\cX_{r,I}, \bW_{k} \Bigr)$ admits a slope
$h$-decomposition.

For every $n\in \N$ also the groups $ {\rm
H}^0\Bigl(\cX_{r,I},{\rm Fil}_n \bW_{k} \Bigr)$ admits a Fredholm
determinant  $P_I^n(k,X)\in\Lambda_I[\![X]\!]$  and a slope
$h$-decomposition. The series $P_I^n(k,X)$ is the product
$$P_I^n(k,X):=\prod_{i=0}^n \mathcal{P}_I(k-2i,p^iX),$$
where $\mathcal{P}_I(k-2i,X)$ is the Fredholm determinant of $U$
on ${\rm H}^0\Bigl(\cX_{r,I},\fw^{k-2i}\Bigr)$. Finally, the
inclusion $ {\rm H}^0\Bigl(\cX_{r,I},{\rm Fil}_n \bW_{k}
\Bigr)^{\leq h}\subset {\rm H}^0\Bigl(\cX_{r,I}, \bW_{k}
\Bigr)^{\leq h}$ is an equality for $n$ large enough.

\end{corollary}
\begin{proof}
Since ${\rm Fil}_n(\bW_{k})$ is coherent and $U$ is compact, the
usual discussion on slope decompositions applies to the groups
${\rm H}^0\bigl(\cX_{r,I}, {\rm Fil}_n(\bW_{k}) \bigr)$, i.e.,
given a finite slope $h\ge 0$ we have, locally on the weight
space, a slope $h$ decomposition
$$
{\rm H}^0\bigl(\cX_{r,I}, {\rm Fil}_n(\bW_{k})   \bigr)={\rm
H}^0\bigl(\cX_{r,I}, {\rm Fil}_n(\bW_{k})   \bigr)^{\le h}\oplus
{\rm H}^0\bigl(\cX_{r,I}, {\rm Fil}_n(\bW_{k})   \bigr)^{>h}.
$$

Thanks to Proposition \ref{prop:upfiltration} the $U$ operator
on the quotient ${\rm H}^0\bigl(\fX_{r,I}, \bW_{k}/{\rm
Fil}_n(\bW_{k}) \bigr)$ is divisible by $p^{h+1}$ for $n$ large
enough. It follows that  ${\rm H}^0\bigl(\cX_{r,I}, \bW_{k}/{\rm
Fil}_n(\bW_{k}) \bigr)$ also admits a slope $h$-decomposition and
in fact ${\rm H}^0\bigl(\cX_{r,I}, \bW_{k}/{\rm Fil}_n(\bW_{k})
\bigr)^{\leq h}=0$. Finally notice that $\mathrm{Gr}_{i}\bW_k
\cong \fw^{k-2i} $ thanks to Theorem \ref{thm:descentbWk}. The
claimed factorization $P_I^n(k,X):=\prod_{i=0}^n
\mathcal{P}_I(k-2i,p^iX)$ follows as in \cite[\S 3.4.2]{UNO}.
\end{proof}

\subsection{The $V$ operator and $p$-depletion on overconvergent modular forms}\label{sec:depletion}

\medskip
\noindent In section \S 2 of \cite{coleman} R.~Coleman defines the $V$ operator on overconvergent modular forms of integer weight and the goal of this paragraph is
to recall his definition in our setting and so make it work on integral, overconvergent modular forms of arbitrary weight.

Let, in the notations of the beginning of this section, $E$ be an elliptic curve defining a point of $\fX_{r+1,I}$. In particular, $E$ has a canonical subgroup
$H_{n+1}$ of order $p^{n+1}$  and we let $\pi\colon E \lra E':=E/H_1$ denote the natural isogeny. We remark that $E'$ defines a point on $\fX_{r,I}$ and has a canonical subgroub $H_n'=H_{n+1}/H_1$. This morphism $\Phi\colon \fX_{r+1,I}\to \fX_{r,I}$ naturally lifts to a morphism  $\Phi\colon \fIG_{n+1,r+1,I} \to \fIG_{n,r,I}$ as a generic trivialization of $H_{n+1}$ provides a generic trivialization of $H_n'$. Let
$\pi^\vee\colon E'\lra E$ be the dual isogeny; then  $\pi^\vee$ defines a morphism ${\rm
H}_{E',n}\lra {\rm H}_{E,n}$ which is an isomorphism if we invert $\alpha$. We are in the setting of \S \ref{sec:functoriality} and therefore  $(\pi^\vee)^\ast$ induces a morphism ${\rm H}_E^\sharp\to {\rm H}_{E'}^\sharp=\Phi^\ast\bigl({\rm H}_E^\sharp\bigr)$ which defines an isomorphism $(\pi^\vee)^\ast\colon
\Omega_E\lra \Omega_{E'}=\Phi^\ast\bigl(\Omega_E\bigr)$ over $\fIG_{n+1,r+1,I}$. By Corollary \ref{prop:functfilbV0} this  gives a morphism $\bW_k \lra \Phi^\ast\bigl(\bW_k\bigr)$ over $\fX_{r+1,I}$ that provides  an isomorphism $(\pi^\vee)^\ast\colon \fw^k \lra \Phi^\ast\bigl(\fw^k\bigr)$ . Define the operator  $$V\colon
{\rm H}^0(\fX_{r,I}, \mathfrak{w}^k) \lra {\rm H}^0(\fX_{r+1,I}, \mathfrak{w}^k), \quad V(\gamma):=\left((\pi^\vee)^\ast\right)^{-1}\bigl(\Phi^\ast(\gamma)\bigr) .$$Its expression on $q$-expansions is:
$$
V(\sum_{n=0}^\infty a_nq^n)=\sum_{n-0}^\infty a_n q^{pn}.
$$

It follows that $U\circ V={\rm Id}_{{\rm H}^0(\fX_{r,I}, \mathfrak{w}^k)}$, as this is so on $q$-expansions.

\begin{definition}
\label{def:pdepletion}
Let $f\in {\rm H}^0(\fX_{r+1,I}, \mathfrak{w}^k)$. We denote by $f^{[p]}:=f-V(U(f))\in {\rm H}^0(\fX_{r+1,I}, \mathfrak{w}^k)$ and call $f^{[p]}$ the $p$-depletion of $f$.
\end{definition}

\begin{remark}
1) If $f\in {\rm H}^0(\fX_{r+1,I}, \mathfrak{w}^k)$, then $U(f^{[p]})=0$.

2) If the $q$-expansion of $f$ is $\displaystyle f(q)=\sum_{n=0}^\infty a_n q^n$ then the $q$-expansion of its
$p$-depletion is $\displaystyle f^{[p]}(q)=\sum_{n\in \N, (n,p)=1} a_n q^n$.

\end{remark}

\subsection{Twists by finite characters}\label{sec:twists}

Let $n$ be a positive integer and fix a primitive $n$-th root of unity $\zeta\in \overline{\Q}_p$.
Let $\chi\colon (\Z/p^n\Z)^\ast\to
\bigl(\Lambda_I[\zeta]\bigr)^\ast$ be a character. The aim of this section is to prove the following:

\begin{proposition}
\label{prop:thetachi} There exists  a unique morphism,  called  {\it the twist  by $\chi$} and denoted $\theta^\chi$ or $\nabla^\chi$:
$$\theta^\chi \colon {\rm H}^0\bigl(\fX_{r,I},\bW_k\bigr) \lra
{\rm H}^0\bigl(\fX_{r+n,I}, \bW_{k+2\chi}\bigr),$$that preserves the filtration $\Fil_\bullet \bW_k$ and the Gauss-Manin connection and such that the induced map on $q$-expansions $\bW_k(q)$, using the
notation of Section \ref{sec:ordinaryandqexp}, is $$\theta^\chi\bigl(\sum_{i=0}^\infty a_i(q) V_{k,i}(q)\bigr) =\sum_{i=0}^\infty \chi\bigl(a_i(q)\bigr)
V_{k+2\chi,i}, \quad \theta^\chi\bigl(\sum_n c_n q^n\bigr)=\sum_n \chi(n) c_n q^n.$$
\end{proposition}

The requirement on $q$-expansions provides the uniqueness. We need to prove that such an operator exists. We first construct it on $\fIG_{2n,r+n,I}$.
Consider the morphism $$t\colon\mathcal{IG}_{2n,r,I} \lra \mathcal{IG}_{n,r,I},$$defined  as follows (remark that we first work on the adic spaces, i.e. $\alpha$ is inverted). We send the
universal generalized elliptic curve $E$ to $ E':=E/H_n$, where $H_n\subset E$ is the canonical subgroup of level $n$. Notice that, denoting $H_{2n}\subset E[p^{2n}]$ the
canonical subgroup of level $2n$, then $H_n':=H_{2n}/H_n \subset E'$ is the canonical subgroup of level $n$ of $E'$. Furthermore if $\gamma$ is the universal section
of $H_{2n}^\vee$ then $\gamma':=p^n \gamma $ defines a section of $H_n^{',\vee}=H_{2n}^\vee[p^n] \subset H_{2n}^\vee$ that generates $H_n^{',\vee}$.  The
morphism $t$ sends $(E,\gamma)\mapsto (E',\gamma')$ with the $\Gamma_1(N)$-level structure on $E'$ defined via the projection $\pi \colon E\to E'$.

Let $\lambda \colon E'\to E$ be the dual isogeny. Then $\lambda$
defines an isomorphism of canoical subgroups $H_n'\cong H_n$. If we set $H_n'':=\Ker(\lambda)$, the $p^n$-torsion of $E'$ decomposes
as
$$E'[p^n]=H_n' \times H_n'',$$as group schemes over
$\mathcal{IG}_{2n,r,I}$, and the Weil pairing induces an isomorphism $H_n':=\bigl(H_n''\bigr)^\vee$. The universal section $\gamma'$ defines  isomorphisms of group schemes
$$s\colon \Z/p^n\Z \to H_n'', \quad s^\vee\colon H_n' \to \mu_{p^n}$$ over
$\mathcal{IG}_{2n,r,I}$ (the second morphism is obtained by duality). Assume that $K$ contains a primitive $p^n$-th root of unity $\zeta$. The choice of $\zeta$ identifies
$\mathrm{Hom}\bigl(\Z/p^n\Z,\mu_{p^n}\bigr) \cong \Z/p^n\Z$: an element $j\in \Z/p^n\Z$ corresponds to the homomorphism $\Z/p^n\Z \to \mu_{p^n}$ sending $1\mapsto
\zeta^j$. We then get a bijection $$\eta\colon \mathrm{Hom}\bigl(H_n'',H_n'\bigr) \to
\mathrm{Hom}\bigl(\Z/p^n\Z,\mu_{p^n}\bigr)\cong \Z/p^n\Z, \quad g\mapsto s^\vee \circ g \circ s$$

\begin{lemma}\label{lemma:firstthetachi} 
Given $\alpha\in (\Z/p^n\Z)^\ast$, if we
let $[\alpha] s$ be the multiplication of $s$ by $\alpha$, the induced map $[\alpha]\eta\colon \mathrm{Hom}\bigl(H_n'',H_n'\bigr) \to \Z/p^n\Z$ is $\alpha^2 \eta$.
\end{lemma}
\begin{proof} For every
$g\in \mathrm{Hom}\bigl(H_n'',H_n'\bigr)$,  we have
$([\alpha]s)^\vee \circ g \circ ([\alpha]s)=\alpha^2 s^\vee \circ
g \circ s$.

\end{proof}

Thanks to the Lemma \ref{lemma:firstthetachi} for every $j\in \Z/p^n\Z$ we get a map $\rho_j\colon H_n''\to H_n'$, inducing the morphism
$\Z/p^n\Z\to \mu_{p^n}$ given by sending $1\mapsto \zeta^j$ (identifying $\Z/p^n\Z\cong H_n''$ via $s$ and $H_n'\cong \mu_{p^n}$ via $s^\vee$). We then let $$H_{\rho_j}:=\bigl(\rho_j \times \mathrm{Id}\bigr)(H_n'') \subset H_n'\times
H_n''=E'[p^n]$$be the closed subgroup scheme given by the image of $\rho_j \times \mathrm{Id}$. Define $$t_j\colon\mathcal{IG}_{2n,r+n,I} \lra \mathcal{IG}_{n,r,I},$$the map given as follows.  Notice that the image of $H_n'$ via the projection map  $\lambda_j\colon E' \to E_j':=E'/H_{\rho_j}$ defines the canonical subgroup $H_{n,j}'\subset E_j'[p^n]$ of order $n$  so that the trivialization $\gamma'\colon \Z/p^n\Z \to (H_n')^\vee$ defines a trivialization $\gamma_j'\colon \Z/p^n\Z \to (H_{n,j}')^\vee$. Then $t_j(E,\gamma)=(E_j',\gamma_j')$ with $\Gamma_1(N)$-level structure on $E_j'$ induced by the one on $E'$. We let $$t\colon \fIG_{2n,r+n,I} \to \fIG_{n,r,I} \qquad t_j\colon \fIG_{2n,r+n,I} \to \fIG_{n,r,I}$$to be the morphisms defined by $t$ and $t_j$, upon taking normalizations. Over $\fIG_{2n,r+n,I} $ we have the isogenies $$E \stackrel{\lambda}{\longleftarrow} E' \stackrel{\lambda_j}{\longrightarrow} E_j',$$where $E$ is the universal ellitpic curve over $\fIG_{2n,r+n,I}$. Moroever by construction $\lambda$ and $\lambda_j$ map the canonical subgroup of $E'$ to the canonical subgroups of $E$ and $E_j'$ respectively, compatibly with the universal sections $\gamma'$, $\gamma$ and $\gamma_j$.  It follows from Lemma \ref{lemma:fsharp} that $\lambda$ and $\lambda_j$ induce morphisms $${\rm H}_{E}^\sharp \stackrel{\lambda^\sharp}{\longrightarrow}{\rm H}_{E'}^\sharp \stackrel{\lambda_j^\sharp}{\longleftarrow}{\rm H}_{E_j'}^\sharp$$with the same image.  In particular, we get isomorphsms $f_j\colon {\rm H}_{E_j'}^\sharp\lra {\rm H}_{E}^\sharp$ as submodules of ${\rm H}_{E'}^\sharp$. Using Proposition \ref{prop:functfilbV0} we finally get morphisms
$$f_j^\ast\colon t_j^\ast\bigl(\bW_k\bigr) \lra \bW_k$$over $\fIG_{2n,r+n,I}$  that preserves the filtration $\Fil_\bullet \bW_k$ defined in Theorem \ref{thm:descentbWk} and the Gauss-Manin connection $\nabla_k$ of Theorem \ref{theorem:griffith}.

\begin{lemma} Let $g_{\chi}:=\sum_{j\in (\Z/p^n\Z)^\ast} \chi(j) \zeta^j$ be the Gauss
sum associated to $\chi$. The map
 $$ \theta^\chi:=\frac{g_{\chi}}{p^n} \cdot \Bigl(\sum_{j\in (\Z/p^n\Z)^\ast} \chi(j)^{-1} f_{j}^\ast \circ t_j^\ast\Bigr)\colon {\rm H}^0\bigl(\fX_{r,I},\bW_k\bigr) \lra
{\rm H}^0\bigl(\fX_{r+n,I}, \bW_{k+2\chi}\bigr)$$has the properties claimed in Proposition \ref{prop:thetachi}.

\end{lemma}
\begin{proof}
We  first check the assertion on the weights, i.e., that $\theta^\chi$ goes from $\bW_k$ to $\bW_{k+2\chi}$.  Take $\alpha\in \Z_p^\ast$. Given $s\in {\rm H}^0\bigl(\fX_{r,I},\bW_k\bigr)$ we have $[\alpha] t_j^\ast(s) =
k(\alpha) t_j^\ast(s)$ by definition. Thanks to Lemma \ref{lemma:firstthetachi} we also have $[\alpha] f_j =f_{\alpha^2 j}$.  Then

$$\begin{aligned}[]   [\alpha] \bigl( \sum_{j\in (\Z/p^n\Z)^\ast} \chi(j)^{-1} f_{j}^\ast(t_j^\ast(s)\bigr) & =\sum_{j\in (\Z/p^n\Z)^\ast} \chi(j)^{-1} \bigl([\alpha] (f_{j}^\ast t_j^\ast) (s)\bigr)=\cr & = \sum_{j\in (\Z/p^n\Z)^\ast} \chi(j)^{-1}
k(\alpha) f_{\alpha^{2} j}^\ast\bigl(t_{\alpha^{2}
j}^\ast(s)\bigr)= \cr & =\sum_{j\in
(\Z/p^n\Z)^\ast} \chi(\alpha^2 j)^{-1} \chi(\alpha)^{2} k(\alpha)
f_{\alpha^{2} j}^\ast\bigl(t_{\alpha^{2}
j}^\ast(s)\bigr) = \cr & =(k+2\chi)(\alpha) \bigl(
\sum_{j\in (\Z/p^n\Z)^\ast} \chi(\alpha^{2} j)^{-1} f_{\alpha^{2}
j}^\ast(t_{\alpha^{2}
j}^\ast(s))\bigr) =\cr & =(k+2\chi)(\alpha) \bigl( \sum_{j\in
(\Z/p^n\Z)^\ast} \chi(j)^{-1} f_{j}^\ast(t_j^\ast(s))\bigr)
.\cr\end{aligned}$$The compatibility with filtrations and Gauss-Manin connection is clear. The assertion on $q$-expansions of modular
forms follows from the proof of \cite[Prop. III.3.17(b)]{Koblitz}.
See also \cite[Lemma 3.3]{Loeffler}.

\end{proof}

\subsection{De Rham cohomology with coefficients in $\bW_k$ and
the overconvergent projection.}

Let $r$, $I\subset [0,\infty)$ and the universal weight  be as in the previous sections. As we will use the universal weight and classical weights as well, at least for this section we write   $\bf k$ for the universal weight to avoid confusion. Let us regard $\displaystyle \bW_{{\bf k}}^\bullet\colon
\bW_{{\bf k}}\stackrel{\nabla_{\bf k}}{\lra}\bW_{{\bf k}+2}$ as a de Rham complex of sheaves on the adic space $\cX_{r,I}$ and denote by ${\rm H}_{\rm dR}^i\bigl(\cX_{r,I},
\bW_{{\bf k}}^\bullet  \bigr)$  the $i$-th hypercohomology group of the de Rham complex $\bW_{{\bf k}}^\bullet$. We observe that because $p$ is a unit in $\cO_{\mathcal{X}_{r,I}}$, the connection
$\nabla_{\bf k}$ does not have poles so that $\displaystyle \bW_{{\bf k}}^\bullet$ and ${\rm H}_{\rm dR}^i\bigl(\cX_{r,I}, \bW_{{\bf k}}^\bullet  \bigr)$ are well defined.

Let us recall that the  sheaf $\bW_{\bf k}$ has a natural filtration preserved by $\nabla_{\bf k}$ therefore we have the following commutative diagram of sheaves on $\cX_{r,I}$ with exact rows:
$$
\begin{array}{ccccccccccc}
0&\lra&{\rm Fil}_n(\bW_{{\bf k}})&\lra&\bW_{{\bf k}}&\lra&\bW_{{\bf k}}/{\rm Fil}_n(\bW_{{\bf k}})&\lra&0\\
&&\downarrow \nabla_{\bf k}&&\downarrow \nabla_{\bf k}&&\downarrow \nabla_{\bf k}\\
0&\lra&{\rm Fil}_{n+1}(\bW_{{\bf k}+2})&\lra&\bW_{{\bf k}+2}&\lra&\bW_{{\bf k}+2}/{\rm Fil}_{n+1}(\bW_{{\bf k}+2})&\lra&0
\end{array}
$$
We denote by ${\rm Fil}_n^\bullet(\bW_{{\bf k}})$ and respectively by $\Bigl(\bW_{{\bf k}}/{\rm Fil}(\bW_{{\bf k}})  \Bigr)^\bullet$ the first, respectively the last, column of the
above diagram.

With these notations we have an exact sequence of de Rham complexes on $\cX_{r,I}$:
 \begin{equation}\label{eq:ast}  0\lra {\rm Fil}_n^\bullet(\bW_{{\bf k}})\lra
\bW_{{\bf k}}^\bullet\lra \Bigl(\bW_{{\bf k}}/{\rm Fil}(\bW_{{\bf k}})  \Bigr)^\bullet\lra 0,
\end{equation}
which gives a long exact sequence of hypercohomology groups

\begin{equation}\label{eq:astast}  0\lra {\rm H}_{\rm dR}^0\bigl(\cX_{r,I}, {\rm Fil}_n^\bullet(\bW_{{\bf k}})\bigr)\lra {\rm H}_{\rm dR}^0\bigl(\cX_{r,I}, \bW_{{\bf k}}^\bullet\bigr)\lra {\rm H}^0_{\rm
dR}\Bigl(\cX_{r,I},   \Bigl(\bW_{{\bf k}}/{\rm Fil}_n(\bW_{{\bf k}})  \Bigr)^\bullet\Bigr)\lra
\end{equation}
$$
\lra {\rm H}_{\rm dR}^1\bigl(\cX_{r,I}, {\rm Fil}_n^\bullet(\bW_{{\bf k}})\bigr)\lra {\rm H}_{\rm dR}^1\bigl(\cX_{r,I}, \bW_{{\bf k}}^\bullet\bigr)\lra {\rm H}^1_{\rm
dR}\Bigl(\cX_{r,I},   \Bigl(\bW_{{\bf k}}/{\rm Fil}_n(\bW_{{\bf k}})  \Bigr)^\bullet\Bigr)\lra \ldots
$$

Moreover, let us recall that the sheaves ${\rm Fil}_m(\bW_{{\bf k}})$ for $m=n$, $n+1$ are coherent and as $\cX_{r,I}$ is a Stein adic space (an affinoid in this case)
the hypercohomology of the complex ${\rm Fil}_n^\bullet(\bW_{{\bf k}})$ is simply calculated as the cohomology of the complex of global sections, i.e. for all $i\geq 0$
we have

\begin{equation}\label{eq:hypedR} {\rm H}^i_{\rm dR}\bigl(\cX_{r,I}, {\rm
Fil}_n^\bullet(\bW_{{\bf k}}) \bigr)={\rm H}^i\Bigl({\rm H}^0\bigl(\cX_{r,I}, {\rm Fil}_n(\bW_{{\bf k}})\bigr)\stackrel{\nabla_{\bf k}}{\lra}{\rm H}^0\bigl(\cX_{r,I}, {\rm
Fil}_{n+1}(\bW_{{\bf k}+2})\bigr)\Bigr).
\end{equation}

\begin{lemma}\label{lemma:H1drFiln} We have an exact sequence, with morphisms equivariant for the action of $U$,

$$0 \lra {\rm H}^0\bigl(\cX_{r,I}, \fw^{{\bf k}+2} \bigr) \lra  {\rm H}^1_{\rm dR}\bigl(\cX_{r,I}, {\rm Fil}_n^\bullet(\bW_{{\bf k}})  \bigr) \lra
\oplus_{i=0}^{n} {\rm H}^0\bigl(\cX_{r,I}, j_{i,\ast} \bigl(\omega_E\bigr)^{-i} \bigr) \lra 0,$$where $\fw^{{\bf k}+2}$ is the universal sheaf of Definition
\ref{defi:wnrI}, the first arrow is induced by the inclusion $\fw^{{\bf k}+2}={\rm Fil}_0(\bW_{{\bf k}+2}) \subset \bW_{{\bf k}+2}$, $j_i$ is the closed immersion
$\cX_{r,I}\times_{\cW_I} \Q_p \subset \cX_{r,I}$ defined by the $\Q_p$-valued point ${\bf k}=i$ of $\cW_I$, $\omega_E$ is the sheaf of invariant differentials
of the universal elliptic curve over $\cX_{r,I}\times_{\cW_I} \Q_p$ and the action of $U$ on ${\rm H}^0\bigl(\cX_{r,I}, j_{i,\ast} \bigl(\omega_E\bigr)^{-i} \bigr)$ is divided by $p^{i+1}$. Moreover, the $\Lambda_I$-torsion of  ${\rm H}^1_{\rm dR}\bigl(\cX_{r,I}, {\rm Fil}_n^\bullet(\bW_{{\bf k}})
\bigr)$ is identified with  ${\rm H}^0\bigl(\cX_{r,I}, j_{i,\ast} \bigl( \overline{\omega}\bigr)^0 \bigr)$ and if we denote by ${\rm H}^1_{\rm
dR}\bigl(\cX_{r,I}, {\rm Fil}_n^\bullet(\bW_{{\bf k}}) \bigr)^{\rm tf}$ the torsion free part, we have an exact sequence, with morphisms equivariant for the action of $U$,

$$0 \lra {\rm H}^0\bigl(\cX_{r,I}, \fw^{{\bf k}+2} \bigr) \lra  {\rm H}^1_{\rm dR}\bigl(\cX_{r,I}, {\rm Fil}_n^\bullet(\bW_{{\bf k}})  \bigr)^{\rm tf} \lra
\oplus_{i=0}^{n} \theta^{i+1}\Bigl({\rm H}^0\bigl(\cX_{r,I}, j_{i,\ast} \bigl( {\omega}_E\bigr)^{-i} \bigr) \Bigr)\lra 0,$$
where $\theta^{i+1} \colon {\rm H}^0\bigl(\cX_{r,I}, j_{i,\ast} \bigl( {\omega}_E\bigr)^{-i} \bigr) \to {\rm H}^0\bigl(\cX_{r,I}, j_{i,\ast} \bigl({\omega}_E\bigr)^{i+2} \bigr)$
is the theta operator defined in \cite[Prop. 4.3]{coleman} and we consider on ${\rm H}^0\bigl(\cX_{r,I}, j_{i,\ast} \bigl( {\omega}_E\bigr)^{-i} \bigr)$ the action of $U$ divided by $p^{i+1}$.

\end{lemma}
\begin{proof}  Theorem \ref{theorem:griffith} and the identification
$\mathrm{Gr}_{i+1}\bW_{\bf k} \cong \fw^{{\bf k}-2i-2} $ of Theorem \ref{thm:descentbWk} imply that $\nabla_{\bf k} \colon {\rm Fil}_n(\bW_{{\bf k}}) \to {\rm Fil}_{n+1}(\bW_{{\bf k}+2})$
induces an isomorphism times  the multiplication by ${\bf k}-n$ map $$\fw^{{\bf k}-2n}\cong \mathrm{Gr}_n(\bW_{{\bf k}}) \to \mathrm{Gr}_{n+1}(\bW_{{\bf k}+2})\cong \fw^{{\bf k}-2n}.$$This map
is injective and the cokernel is identified with $\fw^{{\bf k}-2n}/({\bf k}-n) \fw^{{\bf k}-2n}\cong {\omega}_E^{-n}$. The first claim then follows proceeding by induction on
$n$, using for $n=0$ the identification $\fw^{\bf k}={\rm Fil}_0(\bW_{{\bf k}}) $.

Since ${\rm H}^0\bigl(\cX_{r,I}, \fw^{{\bf k}+2} \bigr)$ is torsion free and ${\rm H}^0\bigl(\cX_{r,I}, j_{i,\ast} \bigl( \omega_E\bigr)^{-i} \bigr)$ is
annihilated by multiplication by ${\bf k}-i$, it follows from the first part of the lemma that  the torsion part of ${\rm H}^1_{\rm dR}\bigl(\cX_{r,I}, {\rm
Fil}_n^\bullet(\bW_{{\bf k}}) \bigr)$ is the sum of the kernels of multiplication by ${\bf k}-i$ for $i=0,\ldots,n$. Fix such an $i$. Consider the following diagram with exact
rows:

$$\begin{matrix}
0 \lra & {\rm H}^0\bigl(\cX_{r,I}, {\rm Fil}_n(\bW_{{\bf k}})\bigr) & \stackrel{\nabla_{\bf k}}{\lra} & {\rm H}^0\bigl(\cX_{r,I}, {\rm Fil}_{n+1}(\bW_{{\bf k}+2})\bigr) & \lra & {\rm
H}_{\rm dR}^1\bigl(\cX_{r,I}, {\rm Fil}_{n}^\bullet(\bW_{{\bf k}})\bigr) & \lra 0 \cr & \downarrow \cdot ({\bf k}-i) & & \downarrow \cdot ({\bf k}-i) & & \downarrow \cdot ({\bf k}-i) \cr 0
\lra & {\rm H}^0\bigl(\cX_{r,I}, {\rm Fil}_n(\bW_{{\bf k}})\bigr) & \stackrel{\nabla_{\bf k}}{\lra} & {\rm H}^0\bigl(\cX_{r,I}, {\rm Fil}_{n+1}(\bW_{{\bf k}+2})\bigr) & \lra & {\rm
H}_{\rm dR}^1\bigl(\cX_{r,I}, {\rm Fil}_{n}^\bullet(\bW_{{\bf k}})\bigr) & \lra 0 \cr
\end{matrix}$$

The rows are exact as ${\rm H}^1\bigl(\cX_{r,I}, {\rm Fil}_n(\bW_{{\bf k}})\bigr) =0$: indeed $\cX_{r,I}$ is affinoid and ${\rm Fil}_n(\bW_{{\bf k}})$ is a coherent $\cO_{\cX_{r,I}}$-module.
Since multiplication by ${\bf k}-i$ is injective on ${\rm Fil}_n^\bullet(\bW_{{\bf k}})$ and ${\rm Fil}_{n+1}^\bullet(\bW_{{\bf k}+2})$ and hence on their global sections, using the
snake lemma we see that the kernel of multiplication by $k-i$ on $ {\rm H}_{\rm dR}^1\bigl(\cX_{r,I}, {\rm Fil}_n^\bullet(\bW_{{\bf k}})\bigr)$ is identified with the
kernel of the complex
$$\nabla\colon {\rm H}^0\bigl(\cX_{r,I}, {\rm Fil}_{n}(\bW_{{\bf k}})/({\bf k}-i)\bigr) \lra  {\rm H}^0\bigl(\cX_{r,I}, {\rm Fil}_{n+1}(\bW_{{\bf k}+2})/({\bf k}-i)\bigr).$$Using that
$\nabla$ induces an isomorphism on graded pieces except for ${\rm Fil}_{i}(\bW_{{\bf k}})/({\bf k}-i)$, this complex is quasi-isomorphic (i.e., the homology groups of the two
complexes are isomorphic) to the sub-complex $$\nabla\colon {\rm H}^0\bigl(\cX_{r,I}, {\rm Fil}_{i}(\bW_{{\bf k}})/({\bf k}-i)\bigr) \lra  {\rm H}^0\bigl(\cX_{r,I}, {\rm
Fil}_{i+1}(\bW_{{\bf k}+2})/({\bf k}-i)\bigr)
$$and, similarly, it is quasi-isomorphic to  the quotient complex $$\nabla\colon {\rm H}^0\bigl(\cX_{r,I}, {\rm Gr}_{i}(\bW_{{\bf k}})/({\bf k}-i)\bigr) \lra  {\rm
H}^0\bigl(\cX_{r,I}, {\rm Fil}_{i+1}(\bW_{{\bf k}+2})/({\bf k}-i)\bigr)/ \nabla {\rm H}^0\bigl(\cX_{r,I}, {\rm Fil}_{i-1}(\bW_{{\bf k}})/({\bf k}-i)\bigr).$$As $\nabla$ induces an
isomorphism $\nabla \colon {\rm H}^0\bigl(\cX_{r,I}, {\rm Fil}_{i-1}(\bW_{{\bf k}})/({\bf k}-i)\bigr)\cong {\rm H}^0\bigl(\cX_{r,I}, \bigl({\rm Fil}_{i}(\bW_{{\bf k}+2})/{\rm
Fil}_{0}(\bW_{{\bf k}+2}\bigr))/({\bf k}-i)\bigr)$ and the image of ${\rm H}^0\bigl(\cX_{r,I}, {\rm Fil}_{i}(\bW_{{\bf k}})/({\bf k}-i)\bigr)$ lies in ${\rm H}^0\bigl(\cX_{r,I}, {\rm
Fil}_{i}(\bW_{{\bf k}+2})/({\bf k}-i)\bigr)$, using the identification $j_{i,\ast} \bigl({\omega}_E\bigr)^{-i}  \cong  {\rm Gr}_{i}(\bW_{{\bf k}})/({\bf k}-i)$ and $j_{i,\ast}
\bigl({\omega}_E\bigr)^{i+2} \cong  {\rm \Fil}_{0}(\bW_{{\bf k}+2})/({\bf k}-i)$, we may identify the kernel of such quotient complex with the kernel of the induced map
$$ {\rm H}^0\bigl(\cX_{r,I}, j_{i,\ast} \bigl({\omega}_E\bigr)^{-i} \bigr)  \lra
 {\rm H}^0\bigl(\cX_{r,I}, j_{i,\ast} \bigl({\omega}_E\bigr)^{i+2}\bigr).$$This is identified with $\theta^{i+1}$, by the results of Coleman  and
it is injective, as it is injective on $q$-expansions,  except for $i=0$ in which case the kernel coincides with ${\rm H}^0\bigl(\cX_{r,I}, j_{i,\ast}
\bigl({\omega}_E\bigr)^0 \bigr)$. See \cite[Prop. 4.3]{coleman}. The twisted action of $U$ in the statement comes from the equality $U \circ \theta^{i+1}=p^{i+1} \theta^{i+1} \circ \theta$ proven in loc.~cit. The claim follows.

\end{proof}

If $C^\bullet$ denotes any one of the complexes in the exact sequence (\ref{eq:ast}), the discussion in Section \ref{sec:upoperator} implies that we have compact
$U$-operators on each one of the groups ${\rm H}_{\rm dR}^i\bigl(\cX_{r,I}, C^\bullet\bigr)$, for $i\ge 0$.

\begin{lemma}\label{lemma:slopeH1dr} For $h\ge 0$ and $n\in \N$ the groups $ {\rm H}^i_{\rm dR}\Bigl(\cX_{r,I},
{\rm Fil}_n(\bW_{{\bf k}})^\bullet \Bigr)$ and ${\rm H}_{\rm dR}^i\bigl(\cX_{r,I}, \bW_{{\bf k}}^\bullet\bigr)$ have  slope $h$-decompositions for every $i$ (in the sense of \cite[\S 4]{ash_stevens}).   Moreover, for $n$ large enough, the exact
sequence (\ref{eq:ast}) induces an isomorphism

$$
{\rm H}_{\rm dR}^i\bigl(\cX_{r,I}, {\rm Fil}_n^\bullet(\bW_{{\bf k}})\bigr)^{\le h}\cong {\rm H}_{\rm dR}^i\bigl(\cX_{r,I}, \bW_{{\bf k}}^\bullet\bigr)^{\le h},
$$
for all $i\geq 0$.

\end{lemma}
\begin{proof} Corollary \ref{cor:FredholmUp}
implies that the groups ${\rm H}_{\rm dR}^i\bigl(\cX_{r,I}, {\rm
Fil}_n^\bullet(\bW_{{\bf k}})\bigr)$ have slope decompositions, i.e.,
given a finite slope $h\ge 0,$ locally on the weight space (i.e.,
we might have to change the interval $I$ but our notations will
not mark this change)  we have the slope decomposition:

$${\rm H}_{\rm dR}^i\bigl(\cX_{r,I}, {\rm Fil}_n^\bullet(\bW_{{\bf k}})\bigr)={\rm H}^i_{\rm dR}\bigl(\cX_{r,I}, {\rm Fil}_n^\bullet(\bW_{{\bf k}})
 \bigr)^{\le h}\oplus {\rm H}_{\rm dR}^i\bigl(\cX_{r,I}, {\rm Fil}_n^\bullet(\bW_{{\bf k}})\bigr)^{>h}.$$

Arguing as in Corollary \ref{cor:FredholmUp} again we also have
that the groups $ {\rm H}^i_{\rm dR}\Bigl(\cX_{r,I},
\Bigl(\bW_{{\bf k}}/{\rm Fil}_n(\bW_{{\bf k}})\Bigr)^\bullet \Bigr)$ have
slope $h$-decompositions for all $i\ge 0$ and  in fact
$${\rm H}^i_{\rm dR}\Bigl(\cX_{r,I}, \Bigl(\bW_{{\bf k}}/{\rm Fil}_n(\bW_{{\bf k}}) \Bigr)^\bullet\Bigr)^{\le h}=0.$$

Therefore the long exact sequence (\ref{eq:astast}) and the considerations above imply the claim.

\end{proof}

We summarize the results of Lemma \ref{lemma:H1drFiln} and of Lemma \ref{lemma:slopeH1dr} in the following 

\begin{theorem}\label{thm:slopedecomp} Given a finite slope $h\geq 0$, locally on the weight space, the groups
${\rm H}_{\rm dR}^i\bigl(\cX_{r,I}, \bW_{{\bf k}}^\bullet \bigr)$ have slope $h$-decompositions. Moreover for $n$ large enough we get exact sequences:

$$0 \lra {\rm H}^0\bigl(\cX_{r,I}, \fw^{{\bf k}+2} \bigr)^{\le h}  \lra  {\rm H}^1_{\rm dR}\bigl(\cX_{r,I},  \bW_{{\bf k}}^\bullet  \bigr)^{\le h} \lra
\oplus_{i=0}^{n} {\rm H}^0\bigl(\cX_{r,I}, j_{i,\ast} \bigl( \overline{\omega}\bigr)^{-i} \bigr)^{\le \frac{h}{p^{i+1}}} \lra 0$$and

$$0 \lra {\rm H}^0\bigl(\cX_{r,I}, \fw^{{\bf k}+2} \bigr)^{\le h}  \lra  {\rm H}^1_{\rm dR}\bigl(\cX_{r,I},  \bW_{{\bf k}}^\bullet  \bigr)^{\le h,\mathrm{tf}} \lra
\oplus_{i=0}^{n} \theta^{i+1}\Bigl({\rm H}^0\bigl(\cX_{r,I}, j_{i,\ast} \bigl( \overline{\omega}\bigr)^{-i} \bigr)\Bigr)^{\le h} \lra 0.$$

\end{theorem}

In particular, take a  rank $1$  point $\rho\colon  \mathrm{Spa}(K,\cO_K)\to \cW_I$ and  denote by $\cX_{r,K}$, $\fw^{\bf k}_K$,  the base change of
$\cX_{r,I}$,  $\fw^{\bf k}$  respectively. We immediately get:

\begin{corollary}\label{cor:special}
If $\rho$ corresponds to a weight different from the classical weights $0,\ldots,n$ we have that
$$\rho^\ast\Bigl( {\rm H}^1_{\rm dR}\bigl(\cX_{r,I},  \bW_{{\bf k}}^\bullet  \bigr)^{\le h}\Bigr)\cong  {\rm H}^0\bigl(\cX_{r, K}, \fw^{{\bf k}+2}_K \bigr)^{\le h}.$$If
$\rho$ corresponds to the weight $k=i$ for some $0\leq i\leq n$ then we have an exact sequence $$0 \lra \frac{\Bigl({\rm H}^0\bigl(\cX_{r,K},
{\omega}_E^{i+2} \bigr)^{\le h}}{\Bigl(\theta^{i+1}{\rm H}^0\bigl(\cX_{r,K}, {{\omega}_E}^{-i} \bigr)\Bigr)^{\le h}} \lra \rho^\ast\Bigl( {\rm
H}^1_{\rm dR}\bigl(\cX_{r,I},  \bW_{i}^\bullet  \bigr)^{\le h,\mathrm{tf}}\Bigr) \lra \Bigl(\theta^{i+1}{\rm H}^0\bigl(\cX_{r,K}, {{\omega}_E}^{-i}
\bigr)\Bigr)^{\le h} \lra 0. $$

\end{corollary}

\begin{proof} Base change $\cX_{r,I}$, $\fw^{\bf k}$, $\bW_{{\bf k}}^\bullet$ to $K$.
Then $\rho$ is defined by the quotient
$\Lambda_I\widehat{\otimes}_{\Z_p} K/
t\Lambda_I\widehat{\otimes}_{\Z_p} K\cong K$ where $t$ is a
regular element of $\Lambda_I\widehat{\otimes}_{\Z_p} K$. Since
multiplication by $t$ is injective on $\fw^{\bf k} $ and on
$\bW_{{\bf k}}^\bullet$ and taking slope decomposition is an exact
operation, the Corollary follows applying the snake lemma to the
multiplication by $t$ to the sequences in Theorem
\ref{thm:slopedecomp}.

\end{proof}

We also have the following Definition inspired by \cite[\S
3.5]{UNO}:

\begin{definition}
\label{def:holomorphicproj} With the notation above, we denote by
$$H^\dagger_n \colon  {\rm H}^1_{\rm dR}\bigl(\cX_{r,I}, \Fil_n^\bullet(\bW_{{\bf k}})  \bigr)\otimes_{\Lambda_I} \Lambda_I\bigl[\prod_{i=0}^n (u_{\bf k}-i)^{-1}\bigr]
\cong  {\rm H}^0\bigl(\cX_{r,I}, \fw^{{\bf k}+2}\bigr)\otimes_{\Lambda_I} \Lambda_I\bigl[\prod_{i=0}^n (u_{\bf k}-i)^{-1}\bigr]$$the isomorphism induced by the inclusion ${\rm
H}^0\bigl(\cX_{r,I}, \fw^{{\bf k}+2} \bigr)  \lra  {\rm H}^1_{\rm dR}\bigl(\cX_{r,I}, \Fil_n^\bullet \bW_{{\bf k}}\bigr)$ of Lemma \ref{lemma:H1drFiln}. Similarly we
define $$H^\dagger \colon  {\rm H}^1_{\rm dR}\bigl(\cX_{r,I}, \bW_{{\bf k}}^\bullet \bigr)^{\le h}\otimes_{\Lambda_I} \Lambda_I\bigl[\prod_{i=0}^n (u_{\bf k}-i)^{-1}\bigr]
\cong  {\rm H}^0\bigl(\cX_{r,I}, \fw^{{\bf k}+2}\bigr)^{\le h}\otimes_{\Lambda_I} \Lambda_I\bigl[\prod_{i=0}^{n_h} (u_{\bf k}-i)^{-1}\bigr]$$as the isomorphism provided via
Theorem \ref{thm:slopedecomp} (here the integer $n_h$ depends on $h$). We call such maps the {\sl overconvergent projections in families}.

Note that for every $\rho\colon  \mathrm{Spa}(K,\cO_K)\to \cW_I$ as above such that the image of $u_{\bf k}-i$ is non-zero in $K$ for $i=0,\ldots,n$, the maps
$\rho^\ast\bigl(H^\dagger_n\bigr)$ and $\rho^\ast\bigl(H^\dagger\bigr)$ are well defined and provide the isomorphism of Corollary \ref{cor:special} upon identifying
${\rm H}^0\bigl(\cX_{r, K}, \fw^{{\bf k}+2}_K \bigr) \cong \rho^\ast\Bigl( {\rm H}^0\bigl(\cX_{r,I}, \fw^{{\bf k}+2}\bigr)\Bigr)$ (and similarly if one considers $(\le
h)$-slope decompositions).

\end{definition}

\subsection{The overconvergent projection and the Gauss-Manin connection on $q$-expansions.}

Let us recall that we have fixed a pair $I$, $r$ consisting of a closed interval $I\subset [0,\infty )$ and an integer $r>0$ adapted to $I$. Consider the Tate curve
$E={\rm Tate}(q^N)$ over $\Spf(R)$ with $R=\Lambda_I^0(\!(q)\!)$ and fix  a basis $\bigl(\omega_{\rm can}, \eta_{\rm can}:=\nabla(\partial)(\omega_{\rm can})\bigr)$
of ${\rm H}_E$ as in \S \ref{sec:ordinaryandqexp}. Using this basis the matrix of the connection $\nabla$ on ${\rm H}_E$ is given by
$$\left(
\begin{array}{cc} 0 & 0 \\ \frac{dq}{q} & 0
\end{array} \right). $$
Write $\bW^0_k(q)=R\langle V\rangle (1+pZ)^{k}$ and  set $V_{k,n}:=Y^n(1+pZ)^{k-n}$ as in loc.~cit. We have

\begin{equation}\label{eq:nablapartial}
\nabla_k(a V_{k,h})=\partial(a)V_{k+2,h}+a(u_k-h)V_{k+2,h+1} \qquad \forall h\ge 0,\end{equation} where let us recall that $u_k\in p^{1-s} \Lambda_I^0$ is such
that $k(t)={\rm exp}(u_k\log(t))$ for all $t\in 1+p^s\Z_p$ for $s \gg 0$. One immediately gets the following:

\begin{proposition}\label{prop:Hdagger(q)} Consider an element
$\gamma\in {\rm H}^0\bigl(\cX_{r,I}, \Fil_{n+1}(\bW^0_{k+2}) \bigr)$ with class $[\gamma]\in {\rm H}^1_{\rm dR}\bigl(\cX_{r,I}, \Fil_n^\bullet(\bW^0_{k})  \bigr)$
via (\ref{eq:hypedR}). Let $\gamma(q)=\sum_{i=0}^{n+1} \gamma_i(q) V_{k+2,i}$ be its evaluation at the Tate curve. Then the $q$-expansion of $\displaystyle H^\dagger_n([\gamma])$
is $\displaystyle \sum_{i=0}^{n+1} \frac{\partial^i\gamma_i(q)}{(u_k-i+1) (u_k-i+2) \cdots u_k} $.
\end{proposition}

We also have the following formula describing the iterations of $\nabla_k$. For simplicity we omit the subscript $k$ and write simply $\nabla$ for the connection.

\begin{lemma}
\label{lemma:formulanablaNg} Let $g(q)\in R$ and $N\ge 1$ and
write $\nabla^N\bigl(g(q)V_{k,h}
\bigr):=\sum_{j=0}^Na_{N,k,h,j}\partial^{N-j}\bigl(g(q)\bigr)V_{k+2N,j+h}$,
with $a_{N,k,h,j}\in R$. We then we have $a_{N,k,h,0}=1$ and for
$j\geq 1$ we have
$$
a_{N,k,h,j}=\left(
\begin{array}{cc} N \\ j\end{array} \right)\frac{(u_k-h+N-1)\cdots (u_k-h+1)(u_k-h)}{(u_k-h+N-1-j)\cdots (u_k-h+1)(u_k-h)}
=\left(
\begin{array}{cc} N \\ j
\end{array} \right)\prod_{i=0}^{j-1}(u_k-h+N-1-i).
$$
In particular, if $u_k\in p\Lambda_I^0$, then $a_{N,k,h,j}\in \Lambda_I$ for all $0\le j\le N$ and $a_{N,k,h,j}\in p\Lambda_I$ if $N=p$ and $j\geq 1$.

\end{lemma}

\begin{proof} We first prove the formula for $a_{N,k,h,j}$ by induction on $N$. For $N=1$ the statement is clear using (\ref{eq:nablapartial}).
Assume the statement true for $N=n$. For $j=0$ or $j=n+1$ the
statement is also clear. So we assume $0<j< n+1$. It follows once
more from (\ref{eq:nablapartial}) that
$a_{n+1,k,h,j}=a_{n,k,h,j}+(u_k+2n-h-j+1) a_{n,k,h,j-1}$. In
particular, $a_{n+1,k,h,j}\in \Lambda_I\subset R^{\partial=0}$.
Moreover we compute
$$ a_{n+1,k,h,j}= \left(
\begin{array}{cc} n \\ j\end{array} \right)\frac{(u_k-h+n-1)\cdots (u_k-h)}{(u_k-h+n-1-j)\cdots (u_k-h)}+ $$

$$ + (u_k-h+2n-j+1) \left(
\begin{array}{cc} n \\ j-1\end{array} \right)\frac{(u_k-h +n-1)\cdots (u_k-h)}{(u_k-h+n-j)\cdots (u_k-h)}=  $$

$$=\frac{(u_k-h+n-1)\cdots (u_k-h)}{(u_k-h+n-j)\cdots (u_k-h)} \left((u_k-h+n-j)\left( \begin{array}{cc} n \\ j\end{array} \right) +  (u_k-h+2n-j+1)
\left(\begin{array}{cc} n \\ j-1\end{array} \right) \right)=$$

$$=\frac{(u_k-h+n-1)\cdots (u_k-h)n!\left( (u_k-h+n-j)(n+1-j)+(u_k-h+2n-j+1) j  \right) }{(u_k-h+n-j)\cdots (u_k-h) j! (n+1-j)!} =$$

$$=\frac{(u_k-h+n-1)\cdots (u_k-h)n!(n+1) (u_k-h+n)}{(u_k-h+n-j)\cdots (u_k-h)j! (n+1-j)!} =
\frac{(u_k-h+n)\cdots (u_k-h)(n+1)!}{(u_k-h+n-j)\cdots (u_k-h)j!
(n+1-j)!},$$as claimed. The last two claims of the Lemma are clear
as $p$ divides $\left(\begin{array}{cc} p \\ j
\end{array} \right)$ for $0<j <p$. For $j=p$ there exists an integer $i$ with $0\leq i \leq p-1$ such that
 and $-h+N-1-i\equiv 0$ modulo $p$ and then $p$ divides $\prod_{i=0}^{p-1}(u_k-h+N-1-i)$
 as $u_k\in p\Lambda_I^0$.

\end{proof}

\begin{remark}\label{rmk:Nablasonqexp} The formula in Lemma \ref{lemma:formulanablaNg}
suggests that for an arbitrary locally analytic weight $s\colon
\Z_p^\ast\to \Lambda_{I_s}^\ast$ one should define
$$\nabla^s\bigl(g(q)V_{k,h}\bigr):=\sum_{j=0} \left(
\begin{array}{cc} u_s \\ j
\end{array} \right)\prod_{i=0}^{j-1}(u_k+u_s-h-1-i)
\partial^{s-j}\bigl(g(q)\bigr)V_{k+2s,j+h}.$$Here $u_s\in \Lambda_{I_s}^0[p^{-1}]$ is such
that $s(t)={\rm exp}(u_s\log(t))$ for all $t\in 1+p^a\Z_p$ for $a
\gg 0$ and $\left(
\begin{array}{cc} u_s \\ j
\end{array} \right)=\frac{u_s \cdot (u_s-1) \cdots (u_s-j+1)}{j!}$. In
particular, in order not to have unbounded denominators in $p$ we
must have that $u_s\in \Lambda_{I_s}$ and  $u_k\in \Lambda_I^0$
and there should be some divisibility by $p$. We will see that
these conditions are also sufficient in order to define $\nabla^s$
for overconvergent families in such a way that  the formula above
on $q$-expansions is satisfied.

\end{remark}

\begin{lemma}
\label{lemma:formulanablaZ} For every positive integers $u$ and $h$, consider the element $1+pZ\in \bW^0_0(q)$.  We then have
$$\nabla^u_0\left(
\frac{((1+pZ)^{2(p-1)}-1)^{ph}}{p^h}\right)= \sum_{j\geq
\mathrm{max}(h-u, 0) }^{h} P_{u,h,j}(1+pZ)
\frac{((1+pZ)^{2(p-1)}-1)^{pj}}{p^j} V_{0,u}$$with
$P_{u,h,j}(T)\in \Z[T]$ a polynomial with coefficients in $\Z$,
divisible by $p$ if $u\geq p$.

\end{lemma}

\begin{proof} Recall that $V_{k+s,n}=(1+pZ)^s V_{k,n}$. For simplicity we omit
the subscript in $\nabla$. We use the formula
(\ref{eq:nablapartial}) that gives
$$\nabla\bigl((1+pZ)^H\bigr)=\nabla\bigl(V_{H,0}\bigr)= H V_{H+2,1}=H (1+pZ)^{H+2} V_{0,1}.$$
Hence $\nabla\bigl(
(1+pZ)^{2(p-1)}-1)^{ph}\bigr)= 2 p h (p-1)
\bigl((1+pZ)^{2(p-1)}-1\bigr)^{ph-1} (1+pZ)^{2(p-1)+2} V_{0,1}$.
As
$\bigl((1+pZ)^{2(p-1)}-1\bigr)^{ph-1}=\bigl((1+pZ)^{2(p-1)}-1\bigr)^{p(h-1)}
\bigl((1+pZ)^{2(p-1)}-1\bigr)^{p-1}$ we get that $$\displaystyle \nabla\bigl(
\frac{((1+pZ)^{2(p-1)}-1)^{ph}}{p^h}\bigr)=2 h Q(1+pZ)
\frac{((1+pZ)^{2(p-1)}-1)^{p(h-1)}}{p^{h-1}} V_{0,1},$$ where $Q(T)$ is a
polynomial with coefficients in $\Z$. Proceeding inductively on
$u$ the first claim follows. 

We prove the second statement. For $p=2$ we have divisibility applying $\nabla$ once and the claim is clear. Assume that $p\geq 3$.  It suffices to deal with the case that $u=p$. Notice that $\nabla^p(f g)=\nabla^p(f) g+ f \nabla^p(g) 
+\sum_{s=1}^{p-1} \left(\begin{matrix} p \cr
s\cr\end{matrix}\right) \nabla^s(f) \nabla^{p-s}(g)$. Thus taking $f=\frac{((1+pZ)^{2(p-1)}-1)^{p}}{p}$ and $g=f^i=\frac{((1+pZ)^{2(p-1)}-1)^{pi}}{p^i}$ and proceedng by induction on $i$ and using the first part for the contribution of $\nabla^s(f) \nabla^{p-s}(g)$,  one is reduced to prove the claim for $f$. 
Write  $$ \nabla\left(f\right)= 2 (p-1)
\bigl((1+pZ)^{2(p-1)}-1\bigr)^{p-1} (1+pZ)^{2} V_{0,1} + 2 p (p-1)
\frac{((1+pZ)^{2(p-1)}-1)^{p}}{p}  (1+pZ)^{2} V_{0,1}.$$Recall that $$\nabla^s\bigl((1+pZ)^H V_{2,1} \bigr)=(H+1) \cdots (H+s)
V_{H+2s+2,s+1}=(H+1) \cdots (H+s) (1+pZ)^{H+2s+2} V_{0,s+1}$$thanks to
formula (\ref{eq:nablapartial}). In particular, if $s=p-1$ and $H$ is prime to $p$ then $ (H+1) \cdots
(H+s)$ is divisible by $p$. We conclude that $\nabla^{p-1}$ of $
\bigl((1+pZ)^{2(p-1)}-1\bigr)^{p-1} (1+pZ)^{2} V_{0,1}=\bigl((1+pZ)^{2(p-1)}-1\bigr)^{p-1} V_{2,1}$ is divisible by $p$ as all the exponents of $(1+pZ)$ appearing  in $\bigl((1+pZ)^{2(p-1)}-1\bigr)^{p-1}$  are prime to $p$.  The second claim follows.

\end{proof}

In particular let $g(q)=\sum_{n=0}^\infty a_nq^n \in
\Lambda_I[\![q]\!]$ be the $q$-expansion of a  $p$-adic modular
form $g$ of weight $k$ and assume that $U \bigl(g(q)\bigr)=0$
that is $a_n=0$ if $p$ divides $n$. Let $c$ be a positive integer
such that $p^{c-1} u_k\in  \Lambda_I^0$.

\begin{proposition}\label{prop:iteratenabla}
For every positive integer $N$ we may write $$\Bigl( \frac{\bigl(\nabla^{p-1}-\mathrm{Id}\bigr)^{Np}}{p^N} \Bigr) \bigl(g(q)V_{k,0}\bigr)=\sum_{r=0}^{(p-1)Np}
\sum_{h\in\N} p^{N-(c+1)r-h} \frac{\bigl( (1+pZ)^{2(p-1)}-1\bigr)^{hp}}{p^h} g_{r,h}^{(N)} V_{k,r}$$with $g_{r,h}^{(N)} \in R^{U=0}[1+pZ]$ a polynomial in $1+pZ$
with coefficients in $R^{U=0}$. If we assume that $u_k\in p\Lambda_I^0$, then $p^{N-2r-h} g_{r,h}^{(N)} \in R^{U=0}[1+pZ]$ for every $r$ and $h$, i.e.,
$p^{2r+h-N}$ divides $g_{r,h}^{(N)}$ whenever $2r+h-N\geq 0$.

\end{proposition}
\begin{proof} We first compute $\bigl(\nabla^{p-1}-\mathrm{Id}\bigr)^H\bigl(g(q) V_{k,n}\bigr)$
for every positive integer $H$:

$$\bigl(\nabla^{p-1}-\mathrm{Id}\bigr)^H\bigl(g(q) V_{k,n}\bigr)=\sum_{s=0}^H \left(\begin{matrix} H \cr
s\cr\end{matrix}\right)(-1)^{H-s}
\nabla^{(p-1)s}\bigl(g(q)V_{k,n}\bigr) =$$

$$=\sum_{s=0}^H \sum_{j=0}^{(p-1)s} \left(\begin{matrix} H \cr
s\cr\end{matrix}\right)(-1)^{H-s} a_{(p-1)s,k,n,j}
\partial^{(p-1)s-j}\bigl(g(q)\bigr) V_{k+2(p-1)s,n+j}=$$

$$=\sum_{s=1}^H \sum_{j=1}^{(p-1)s} \left(\begin{matrix} H \cr
s\cr\end{matrix}\right)(-1)^{H-s} a_{(p-1)s,k,n,j}
\partial^{(p-1)s-j}\bigl(g(q)\bigr) (1+pZ)^{2(p-1)s} V_{k,n+j} +$$
$$+\sum_{s=0}^H  \left(\begin{matrix} H \cr
s\cr\end{matrix}\right)(-1)^{H-s}
\partial^{(p-1)s}\bigl(g(q)\bigr) (1+pZ)^{2(p-1)s} V_{k,n}=$$

$$\sum_{s=1}^H \sum_{j=1}^{(p-1)s}  \left(\begin{matrix} H \cr
s\cr\end{matrix}\right)(-1)^{H-s} a_{(p-1)s,k,n,j}
\partial^{(p-1)s-j}\bigl(g(q)\bigr) (1+pZ)^{2(p-1)s} V_{k,n+j}
+$$

$$+\sum_{s=1}^H  \left(\begin{matrix} H \cr
s\cr\end{matrix}\right)(-1)^{H-s}
\left(\partial^{(p-1)s}(g(q))-g(q)\right) (1+pZ)^{2(p-1)s} V_{k,n}
+ \bigl((1+pZ)^{2(p-1)}-1\bigr)^H V_{k,n}.
$$

\

{\it Base step $N=1$:} We prove the Lemma for $N=1$ using the previous computation with $H=p$.

For $1\leq s \leq p-1$ we have that $\partial^{(p-1)s}(g(q))-g(q)\in p R^{U=0}$ so that $\Bigl(\begin{matrix} p \cr s\cr\end{matrix}\Bigr)
\bigl(\partial^{(p-1)s}(g(q))-g(q)\bigr)\in p^2 R^{U=0}$. On the other hand also $\partial^{(p-1)p}(g(q))-g(q)\in p^2 R^{U=0}$. Considering the term
$\bigl((1+pZ)^{2(p-1)}-1\bigr)^p$, the first part of the claim is proven for $h=0$ or $h=1$ and $r=0$. Recall from Lemma \ref{lemma:formulanablaNg} that
$a_{(p-1)s,k,n,j}$ is a polynomial with coefficients in $\Z$ in $u_k$ of degree $j$ so that by assumption $p^{j(c-1)}a_{(p-1)s,k,n,j}\in \Lambda_I^0$. The first
part of the claim then follows also for the terms with $r\geq 1$.

We prove the second part. For $j\geq 1$ we have $2j-1\geq 1$ so that $p^{2j-1}\partial^{(p-1)s-j}\bigl(g(q)\bigr) (1+pZ)^{2(p-1)s}\in R^{U=0}[1+pZ]$. It follows
from Lemma \ref{lemma:formulanablaNg} that $\left(\begin{matrix} p \cr s\cr\end{matrix}\right) a_{(p-1)s,k,n,j} \in p \Lambda_I$: in fact for $1\leq s \leq p-1$ the
binomial coefficient $\left(\begin{matrix} p \cr s\cr\end{matrix}\right)$ is divisible by $p$, for $s=p$ and $j$ prime to $p$ then $a_{(p-1)p,k,n,j}$ has a factor
$\left(\begin{matrix} p (p-1) \cr j\cr\end{matrix}\right)$ which is divisible by $p$ and for $j$ divisible by $p$ then $a_{(p-1)p,k,n,j}$ has a factor
$\prod_{i=0}^{p-1}(u_k+p(p-1)-1-i) $ divisible by $p$. This proves the statement for $N=1$.

\

{\it Inductive step $N\Longrightarrow N+1$:} It suffices to prove the following:

\

CLAIM: Let $g_{r,h}^{(N)}\in R^{U=0}[1+pZ]$ and suppose that $p^{N-2r-h} g_{r,h}^{(N)}\in R^{U=0}[1+pZ]$ in case $u_k\in p\Lambda_I^0$. Then

$$\frac{(\nabla^{p-1}-\mathrm{Id})^p}{p} \left(p^{N-(c+1)r-h} \frac{\bigl( (1+pZ)^{2(p-1)}-1\bigr)^{ph}}{p^h} g_{r,h}^{(N)} V_{k,r}\right)=$$
$$=\sum_j \sum_v p^{N+1-(c+1)(r+j)-v} \frac{\bigl( (1+pZ)^{2(p-1)}-1\bigr)^{pv}}{p^v} g_{j,v}^{(N+1)} V_{k,r+j} $$ with $g_{j,v}^{(N+1)}\in R^{U=0}[1+pZ]$.
Furthermore, if we assume that $u_k\in p\Lambda_I^0$  then $p^{N+1-2(r+j)-v}g_{j,v}^{(N+1)}\in R^{U=0}[1+pZ]$ for every $j$.

\

We compute $\frac{(\nabla^{p-1}-\mathrm{Id})^p}{p} \left(p^{N-(c+1)r-2h} \bigl( (1+pZ)^{2(p-1)}-1\bigr)^{ph} g_{r,h}^{(N)} V_{k,r}\right)$ as the sum of two terms:

\begin{equation}\label{eq:A}
p^{N-(c+1)r-h} \frac{\bigl( (1+pZ)^{2(p-1)}-1\bigr)^{ph}}{p^h} \frac{(\nabla^{p-1}-\mathrm{Id})^p}{p} \bigl( g_{r,h}^{(N)} V_{k,r}\bigr)
\end{equation}

and

\begin{equation}\label{eq:B}
\sum_{s=1}^p \left(\begin{matrix} p \cr s\cr\end{matrix}\right) (-1)^{p-s} p^{N-(c+1)r-h-1} \sum_{u=1}^{s(p-1)} \left(\begin{matrix} s(p-1) \cr
u\cr\end{matrix}\right) \nabla^u\left(\frac{ \bigl( (1+pZ)^{2(p-1)}-1\bigr)^{ph}}{p^h}\right) \nabla^{s(p-1)-u} \bigl( g_{r,h}^{(N)} V_{k,r}\bigr).
\end{equation}

We start with the contribution given by (\ref{eq:A}). As $(1+pZ)^s V_{k,n}=V_{k+s,n}$, the computation at the beginning of the proof shows that if we write
$$\frac{(\nabla^{p-1}-1)^p}{p} \bigl( g_{r,h}^{(N)} V_{k,r}\bigr)=\sum_j p^{1-(c+1)j} g_{r,j}^{(N+1)} V_{k,r+j}$$then $g_{r,j}^{(N+1)}\in R^{U=0}[1+pZ]$
for $j\geq 1$ and $g_{r,0}^{(N+1)}$  is the sum of a term  $\alpha$ with $\alpha\in R^{U=0}[1+pZ]$ and a term $\frac{\bigl( (1+pZ)^{2(p-1)}-1\bigr)^p}{p^2}\beta$
with $\beta\in R^{U=0}[1+pZ]$. Those terms multiplied by $p^{N-(c+1)r-2h} \bigl( (1+pZ)^{2(p-1)}-1\bigr)^{ph}$  satisfy the claim:

We start with the terms $j\geq 1$ and we use that $N-(c+1)r-h \geq N+1-(c+1)(r+j)-h$. Then  $p^{N-(c+1)r-h} \frac{\bigl( (1+pZ)^{2(p-1)}-1\bigr)^{ph}}{p^h}
g_{r,j}^{(N+1)}$ is equal to $$ p^{N+1-(c+1)(r+j)-h} \frac{\bigl( (1+pZ)^{2(p-1)}-1\bigr)^{ph}}{p^h} \bigl(p^{(N-(c+1)r-h) -(N+1-(c+1)(r+j)-h)}
g_{r,j}^{(N+1)}\bigr)$$with $p^{(N-(c+1)r-h) -(N+1-(c+1)(r+j)-h)} g_{r,j}^{(N+1)}\in R^{U=0}[1+pZ ]$. Assuming that $u_k\in p\Lambda_I^0$ we also have
$$p^{N+1-2(r+j)-h} p^{(N-2r-h) -(N+1-2(r+j)-h)} g_{r,j}^{(N+1)}= p^{N-2r-h} g_{r,j}^{(N+1)} \in R^{U=0}[1+pZ ]$$
using the inductive hypothesis that $p^{N-2r-h} g_{r,h}^{(N)}\in R^{U=0}[1+pZ ]$.

Consider next the contribution for $j=0$, i.e., $p^{N+1-(c+1)r-h} \frac{\bigl( (1+pZ)^{2(p-1)}-1\bigr)^{ph}}{p^h} g_{r,0}^{(N+1)}$. It is the sum of two terms. The
first is $p^{N+1-(c+1)r-h} \frac{\bigl( (1+pZ)^{2(p-1)}-1\bigr)^{ph}}{p^h} \cdot \alpha$. If we assume that $u_k\in p\Lambda_I^0$ then
$p^{N+1-2r-h}\alpha=p^{N-2r-h} (p\alpha)\in R^{U=0}[1+pZ]$ by the hypothesis that $p^{N-2r-h} g_{r,h}^{(N)}\in R^{U=0}[1+pZ ]$.

On the other hand $$p^{N-(c+1)r-h} \frac{\bigl( (1+pZ)^{2(p-1)}-1\bigr)^{ph}}{p^h} \frac{\bigl( (1+pZ)^{2(p-1)}-1\bigr)^p}{p} \beta=$$ $$=p^{N+1-(c+1)r-(h+1)}
\frac{\bigl( (1+pZ)^{2(p-1)}-1\bigr)^{p(h+1)}}{p^{h+1}}\beta.$$If $u_k\in p\Lambda_I^0$ then $p^{N-2r-h} g_{r,h}^{(N)}\in R^{U=0}[1+pZ ]$ so that
$p^{N+1-2r-(h+1)}\beta=p^{N-2r-h}\beta \in R^{U=0}[1+pZ]$.

\

Consider next the contribution of the terms in (\ref{eq:B}), namely $$\left(\begin{matrix} p \cr s\cr\end{matrix}\right) p^{N-(c+1)r-h-1} \left(\begin{matrix}
s(p-1) \cr u\cr\end{matrix}\right) \nabla^u\left( \frac{\bigl( (1+pZ)^{2(p-1)}-1\bigr)^{ph}}{p^h}\right) \nabla^{s(p-1)-u} \bigl( g_{r,h}^{(N)} V_{k,r}\bigr)$$for
$1\leq s\leq p$ and $1\leq u\leq s(p-1)$ and write $$\nabla^{s(p-1)-u} \bigl( g_{r,h}^{(N)} V_{k,r}\bigr)=\sum_{j\geq 0} \alpha_{r,j}^{(N+1)} V_{k,r+j} .$$ It
follows from Lemma \ref{lemma:formulanablaZ} that $\nabla^u\bigl( \bigl( (1+pZ)^{2(p-1)}-1\bigr)^{ph} p^{-h}\bigr)$ can be written as a sum $\displaystyle
\sum_{i\geq \mathrm{max}(h-u,0)}^{h} P_{u,h,i}(1+pZ) \bigl( \frac{((1+pZ)^{2(p-1)}-1)^{pi}}{p^i}\bigr) V_{0,u}$ with $P_{u,h,j}(T)$ a polynomial with coefficients
in $\Z$. Hence we need to analyze the expression

$$\sum_{j\geq 0}\sum_{i\geq \mathrm{max}(h-u,0)}^{h} \left(\begin{matrix} p \cr s\cr\end{matrix}\right)
p^{N-(c+1)r-h-1} \left(\begin{matrix} s(p-1) \cr u\cr\end{matrix}\right) P_{u,h,i}(1+pZ) \frac{((1+pZ)^{2(p-1)}-1)^{pi}}{p^i}\alpha_{r,j}^{(N+1)} V_{k,r+j+u}$$

As $1\leq s$,  if $s\leq p-1$ then $\left(\begin{matrix} p \cr
s\cr\end{matrix}\right)$ is divisible by $p$. If $s=p$ and $u$ is
coprime to $p$ then $ \left(\begin{matrix} p(p-1) \cr
u\cr\end{matrix}\right)$ is divisible by $p$. If $s=p$ and $u$ is
divisible by $p$ then the polynomials $P_{u,h,j}(T)$ are divisible
by $p$. In all these cases $\left(\begin{matrix} p \cr
s\cr\end{matrix}\right)  \left(\begin{matrix} s(p-1) \cr
u\cr\end{matrix}\right) P_{u,h,i}(1+pZ) $ is divisible by $p$.
Write this as $p \beta_{s,u,h,i}$ with $\beta_{s,u,h,i} \in
\Z[1+pZ]$.

As $j\geq 0$, $c+1\geq 2$, $u\geq 1$ and $i\geq h-u$, we also have $N-(c+1)r-h\geq N-(c+1)(r+j)-h \geq N-(c+1)(r+j+u)-(h-u)+u \geq N+1-(c+1)(r+j+u)-i$. Hence
$$\left(\begin{matrix} p \cr s\cr\end{matrix}\right)
p^{N-(c+1)r-h-1} \left(\begin{matrix} s(p-1) \cr u\cr\end{matrix}\right) P_{u,h,i}(1+pZ) \frac{((1+pZ)^{2(p-1)}-1)^{pi}}{p^i} \alpha_{r,j}^{(N+1)}=$$
$$= \beta_{s,u,h,i}    p^{N+1-(c+1)(r+j+u)-i}
\frac{((1+pZ)^{2(p-1)}-1)^{pi}}{p^i} \bigl(p^{(N-(c+1)r-h) -(N+1-(c+1)(r+j+u)-i)}\alpha_{r,j}^{(N+1)}\bigr)
$$and $\bigl(p^{(N-(c+1)r-h)
-(N+1-(c+1)(r+j+u)-i)}\alpha_{r,j}^{(N+1)}\bigr)\in R^{U=0}[1+pZ]$.

Furthermore, assuming that $u_k\in p\Lambda_I^0$ and  that $p^{N-2r-h} \alpha_{r,j}^{(N+1)}\in R^{U=0}[1+pZ]$ by inductive hypothesis, we have
$p^{N+1-2(r+j+u)-i}  \bigl(p^{(N-2r-h) -(N+1-2(r+j+u)-i)}\alpha_{r,j}^{(N+1)}\bigr)= p^{N-2r-h} \alpha_{r,j}^{(N+1)}\in R^{U=0}[1+pZ]$. This proves the inductive
step and the Claim follows.

\end{proof}

Write $\fw^k(q)$ for the sheaf $\fw^k$ evaluated on the Tate curve. We consider it as a submodule of $\bW_k(q)$ using the identification $\fw^k(q)=\Fil_0 \bW_k(q)$.
Recall that $\bW_k(q):=\bW^0_k(q) \otimes_R \fw^{k_f}(q)$, where $\fw^{k_f}$ is the evaluation at the Tate curve of the coherent sheaf
$\bigl(g_{i,\ast}\bigl(\cO_{\mathfrak{IG}_{i,I}^{\rm ord}}\bigr)\otimes_{\Lambda^0_I} \Lambda_I\bigr) \bigl[k_{I,f}^{-1}\bigr] $ where $i=1$ for $p$ odd and $i=2$
for $p=2$.

\begin{corollary}\label{cor:iteratenabla}
Let $g(q)\in \fw^k(q)$ with $U\bigl(g(q)\bigr)=0$ and let $c$ be a positive integer such that $p^{c-1} u_k\in \Lambda_I^0$. Then for every positive integer $N$ we
have
$$\Bigl( {\bigl(\nabla^{p-1}-\mathrm{Id}\bigr)^{Np}} \Bigr) \bigl(g(q)\bigr)\in \sum_{n=0}^{(p-1)p N}p^{2N-(c+1)n} \fw^{k_f}(q)[Z] V_{k,n}.$$Moreover, if $u_k\in
p\Lambda_I^0$ then
$$\Bigl( {\bigl(\nabla^{p-1}-\mathrm{Id}\bigr)^{Np}} \Bigr) \bigl(g(q)\bigr)\in p^N \cdot \Bigl(\sum_{n=0}^{(p-1)p N} \fw^{k_f}(q)[Z] V_{k,n}\Bigr).$$

\end{corollary}
\begin{proof}  The Igusa tower $\mathfrak{IG}_{i,I}^{\rm ord}$ over $\Spf(R)$ becomes a disjoint union of copies of $\Spf(R)$,
permuted transitively by the group $G_i=(\Z/q\Z)^\ast$. We denote it by $\mathfrak{IG}(q)=\Spf(R')$. The connection on $\bW_k(q)$ is the composite of the connection
on $\bW^0(q)$ and of the connection on $\fw^{k_f}(q)$ defined by usual the derivation on $R'$ and, hence, by the derivation on $R$. It suffices to prove the claim
replacing $\fw^{k_f}(q)$ with $R'':=R'\otimes_{\Lambda_I^0} \Lambda_I$ which is a finite and free $R$-module, i.e., we work with $\bW^0_k(q) \otimes_R R''$. Fix a
basis $\{e_j\}_j$, with $j$ varying in a set of indices $J$, so that $\bW^0_k(q) \otimes_R R''=\oplus_j \bW^0_k(q) e_j$. Taken $g(q)\in \fw^k(q)\subset \bW_k(q)$ we
can decompose it as a sum $\sum_j g_j(q) e_j$ with $g_j(q)\in \fw^{k,0}(q) \subset \bW_k^0(q)$ and $\nabla(g(q))=\sum_j \nabla(g_j(q)) e_j$.  The assumption
$U(g(q))=0$ is equivalent to require that $U\bigl(g_j(q)\bigr)=0$ for every $j$. The statement then follows from Proposition \ref{prop:iteratenabla}.

\end{proof}

\section{$p$-Adic iterations of the Gauss Manin
connection.}\label{sec:IterateManin}

Let us fix closed intervals $I_s=[p^a, p^b]$ and $I=I_k=[p^c,p^d]$ with $a\le b,c\le d, a,b,c,d\in \N$ and an integer $r$ adapted to $I_s$ and $I_k$. The main topic of this chapter, in view of
applications to the construction of $p$-adic $L$-functions attached to triple products of finite slope families of eigenforms in the next chapter, is the following:
given weights $k\colon \Z_p^\ast\lra \Lambda_{I}^\ast$ and $s\colon \Z_p^\ast\lra \Lambda_{I_s}^\ast$ define the operator ``$(\nabla_k)^s$".

To see what this should be we'll  first look at $q$-expansions. Let $g(q)=\sum_{n=0}^\infty a_nq^n \in \Lambda_{I}[\![q]\!]$ be the $q$-expansion of a  $p$-adic
modular form $g$ of weight $k$. Then the $q$-expansion of $\nabla_k(g)$ is $\partial(\sum_{n=0}^\infty a_nq^n)\omega_{\rm can}^2=(\sum_{n=1}^\infty
na_nq^n)\omega_{\rm can}^2$ seen as the $q$-expansion of a $p$-adic modular form of weight $k+2$. Here $\displaystyle \partial:=q\frac{d}{d q}$. Let
$g^{[p]}(q):=\sum_{n=0, (p,n)=1}^\infty a_nq^n$ be the $p$-depletion of $g(q)$. Seeing the weight $s$ as a continuous homomorphism $s\colon \Z_p^\ast\lra
\Lambda_{I_s}^\ast$, we define the operator $\partial^s$ on $p$-depleted $q$-expansions by:
$$
\partial^s\bigl(g^{[p]}(q)\bigr):=\sum_{n=1, (n,p)=1}^\infty a_n s(n) q^n.
$$
It can be seen easily that $g^{[p]}(q)$ is the $q$-expansion of a $p$-adic modular form of weight $k$ which lies in the kernel of the $U$-operator and that
$\partial^s\bigl(g^{[p]}(q)\bigr)$, thus defined is the $q$-expansion of a $p$-adic modular form of weight $k+2s\colon \Z_p^\ast\to \Lambda_{I} \widehat{\otimes}
\Lambda_{I_s}$.

Therefore  we'd expect that $\nabla_k^s$ were a differential operator defined on ${\rm H}^0(\fX_{r,I}, \bW^0_k)^{U=0}$ with values in ${\rm H}^0(\fX_{r,I}\otimes
\Lambda_{I_s} , \bW^0_{k+2s})$, but unfortunately things are not as simple as this.

The first problem is that $\nabla_k$, seen as a connection on the sheaf $\bW^0_k$ over $\fX_{r,I}$, has poles along ${\rm Hdg}$;  see section \ref{sec:GMwk}. This
makes it difficult to iterate it.

The second problem is that the definition of $\partial^s$ on $q$-expansions given above is not algebraic enough and what we would like to interpolate is not
$\partial$ but the whole connection $\nabla$. We then incur in the problem discussed in Remark \ref{rmk:Nablasonqexp}.

To remedy this let us suppose that the weight $s$ has the property: there is $u_s\in \Lambda_{I_s}$ such that for every $t\in \Z_p^\ast$, $s(t)={\rm
exp}\bigl(u_s{\rm log}(t) \bigr)$. In particular $s\vert_{\mu_{p-1}}=1$ and $s\vert_{1+p\Z_p}$ is analytic. Then let us remark that the operator
$\partial^{p-1}-{\rm Id}$ on $p$-depleted $q$-expansions is divisible by $p$, i.e., if $g(q)=\sum_{n=1,(n,p)=1}^\infty a_nq^n\in \Lambda_I[\![q]\!]^{U=0}$, then
$\partial^{p-1}(g(q))-g(q)=\sum_{n=1,(n,p)=1}^\infty a_n\bigl(n^{p-1}-1\bigr)q^n\in p\Lambda_{I}[\![q]\!]$. So if we put:
$$
\delta(g):={\rm exp}\Bigl(\frac{u_s}{p-1}\log(\partial^{p-1})\Bigr)(g)
$$
then the definition makes sense and moreover we have that $\delta(g)$,  on $p$-depleted $q$-expansions, equals the previously defined $\partial^s(g)$. Our strategy
to define $\nabla_k^s$ in general is based on the following assumption:

\begin{assumption}\label{ass10}  $k$ and $s$ are weights satisfying the condition:
 $k=\chi\cdot k_0\cdot v$ and $s=\chi'\cdot s_0\cdot w$ where:

\begin{itemize}

\item[a)] $\chi$, $\chi'$ are finite order  characters of $\Z_p^\ast$ and $\chi$ is even;

\item[b)] $k_0$, $s_0$ are integer weights such that $k_0$ is even modulo $p$, i.e., there are integers $a$, $b$ with $a$ even modulo $p$ such that $k_0(t)= t^a$,
$s_0(t)=t^b$ for all $t\in \Z_p^\ast$.

\item[c)] $v$, $w\colon\Z_p^\ast \lra \Lambda_I^\ast$ are weights such that there exist $u_v\in p\Lambda_I$, $u_w\in q\Lambda_I$ satisfying: $v(t)={\rm
exp}\bigl(u_v\log(t)\bigr)$ and $w(t)={\rm exp}\bigl(u_w\log(t)  \bigr)$ for all $t\in \Z_p^\ast$.
\end{itemize}

\end{assumption}

We recall that $q=p$ if $p\geq 3$ and $q=4$ if $p=2$ and that a finite order character $\chi\colon \Z_p^\ast\lra \cO_K^\ast$ is called {\bf even} if there is a
finite field extension $K\subset L$ and a character $\varepsilon\colon \Z_p^\ast\lra \cO_L^\ast$ such that $\varepsilon^2=\chi$.

\begin{remark} Assume that $p$ is odd. Then $\chi\colon \Z_p^\ast\lra \cO_K^\ast$ has the form
$\chi=\epsilon\cdot \tau$, where $\epsilon|_{1+p\Z_p}=1$ and $\tau|_{(\Z/p\Z)^\ast}=1$, i.e. $\epsilon=\omega^i$ with $i$ a positive divisor of $p-1$, while $\tau$
is a character of order a power of $p$. Here we have denoted by $\omega$ the Teichm\"uller character composed with reduction modulo $p$. Let us remark that
$\epsilon$, $i$ and $\tau$ are uniquely determined by $\chi$.  Then, the character $\chi$ is even if and only if $\epsilon$ is even, i.e., if and only if $i$ is
even. In this case  the field $L$  may be taken $L=K$.
\end{remark}

For all $g\in {\rm H}^0\bigl(\fX_{r,I}, \bW^0_k \bigr)^{U=0}$, if $k$ and $s$ satisfy Assumption \ref{ass10} we set
$$
(\nabla_k)^s(g):={\rm exp}\Bigl(\frac{u_s}{(p-1)}{\rm log}\bigl(\nabla_k^{(p-1)}\bigr)  \Bigr)(g)
$$
and claim that this makes sense and it is the desired section of $\bW_{k+2s}$.

The rest of this section is devoted to the implementation of this strategy. Let $r$ and $I$ be as at the beginning of this section and let $n$ be an integer adapted
to $I$.  The main result is the following:

\begin{theorem}
\label{theorem:mostgeneral} Let $K$, $k$ and $s$ be a finite extension of $\Q_p$ and respectively a pair of weights satisfying the Assumptions \ref{ass10} such that
$K$ contains the images of the finite parts of $k$, $s$. Let  $g\in {\rm H}^0(\fX_{r,I}, \bW_k)^{U=0}$. Then there exist positive integers $b$ and $\gamma$
depending on $r$, $n$ and $p$ and an element $\nabla_k^s(g)$ of $ {\rm Hdg}^{-\gamma}{\rm H}^0(\fX_{b,I}, \bW_{k+2s})$ such that on $q$-expansions, in the notations
of \S \ref{sec:ordinaryandqexp},  if $g(q)=\sum_h g_h(q) V_{k,h}$ then
$$\nabla_k^s\bigl(g\bigr)(q):=\sum_h \sum_{j=0} \left(
\begin{array}{cc} u_s \\ j
\end{array} \right)\prod_{i=0}^{j-1}(u_k+u_s-h-1-i)
\partial^{s-j}\bigl(g_h(q)\bigr)V_{k+2s,j+h};$$here $\left(
\begin{array}{cc} u_s \\ j
\end{array} \right)=\frac{u_s \cdot (u_s-1) \cdots (u_s-j+1)}{j!}$ and if $g_h(q)=\sum_{n, p\not\vert n} a_{h,n} q^n$ then
$$\partial^{s-j}\bigl(g_h(q)\bigr)=\sum_{n, p\not\vert n} s(n) n^{-j} a_{h,n} q^n.$$

\end{theorem}
\begin{proof} We show how to reduce the proof of the Theorem to the case that $k$ and $s$ as in Assumption \ref{ass10}
which in addition satisfy $\chi=\chi'=1$ and $k_0=s_0=1$ (i.e. $a=b=0$) (we'll call these ``strict assumptions"). The proof in this case is postponed to \S
\ref{sec:proofmostgeneral}.

{\it Case I.} \enspace First of all assume  that  $k(t)={\rm exp}\bigl((a+u_v)\log(t)\bigr)$ and $s(t)={\rm exp}\bigl((b+u_w)\log(t) \bigr)$ for all $t\in
\Z_p^\ast$, where $a$, $b\in \Z$ with $a$ even modulo $p$, $u_v\in p\Lambda_I$ and $u_w\in q\Lambda_I$. Let then $\alpha,\beta\in \Z$ be integers such that:

i) $p \vert (a+2\alpha)$ and $\alpha>0$

and

ii) $q \vert \beta$ and $\beta-\alpha+b>0$.

Then let us remark that we can write formally:
$$
(\nabla_k)^s(g):= (\nabla_{u_v+a+2\alpha+2u_w-2\beta})^{\beta-\alpha+b}  \Bigl((\nabla_{u_v+a+2\alpha})^{u_w-\beta}\bigl((\nabla_k)^\alpha(g)\bigr)\Bigr).
$$
Remark that everything written on the right hand side makes sense either because the weights satisfy the strict assumptions or because the exponent of $\nabla$ is a
positive integer.

Here we wrote the weights additively, i.e. $u_v+a+2\alpha$ is the weight sending $t\in \Z_p^\ast$ to ${\exp}\bigl((u_v+a+2\alpha)\log(t)\bigr)=k(t)\cdot
t^{2\alpha}$ etc. We leave to the reader to prove that one obtains the expected formula on $q$-expansions  using Lemma \ref{lemma:formulanablaNg} and the assumption
that the Theorem holds for $(\nabla_{u_v+a+2\alpha})^{u_w-\beta}$.

\

{\it Case II.} \enspace We next consider weights of the form $k=k' \chi$ and $s=s' \chi'$ with $k'$ and $s'$ weights with trivial character of the type considered
in Case I and  $\chi$ and $\chi'$ characters such that $\chi$ even. Let $L$ be a finite extension of $K$ and $\varepsilon\colon \Z_p^\ast\lra \cO_L^\ast$ a finite
character such that $\chi=\varepsilon^2$. Thanks to Proposition \ref{prop:thetachi}  we have an element
$$
\theta^{\varepsilon^{-1}}(g)\in {\rm H}^0(\fX_{r,I}, \bW_{k-2\tau}\otimes_K L)^{U=0}={\rm H}^0(\fX_{r,I}, \bW_{k'}\otimes_K L)^{U=0}.
$$Define:

$$\nabla_k^s(g):=\theta^{\varepsilon \, \chi}\Bigl(\nabla_{k'}^{s'}\bigl(\theta^{\varepsilon^{-1}}(g)\bigr)  \Bigr)
\in {\rm H}^0(\fX_{r,I}, \bW_{k+2s}\otimes_K L)^{U=0}.$$ We leave to the reader to check that one gets the required formula  on $q$-expansions using Proposition
\ref{prop:thetachi}. One deduces from this that in fact $\nabla_k^s(g)\in {\rm H}^0(\fX_{r,I}, \bW_{k+2s})^{U=0}$, concluding the proof.

\end{proof}

Following a suggestion of Eric Urban define:

\begin{definition}\label{def:nablachi} If $s=\chi'$ is a finite character and $a$ is a positive integer and $f\in{\rm H}^0(\fX_{b,I}, \fw^{k})$ set
$ \nabla_k^{a+s}(f):=\theta^{\chi'}\bigl(\nabla_k^a(f)\bigr)$.

\end{definition}

\begin{remark} Notice that $\nabla_k^{a+s}(f)$ is a modular form of weight $k+2a+2\chi'$ that coincides with $\nabla_k^a\bigl(\theta^{\chi'}(f)\bigr)$
as can be checked on $q$-expansions using Lemma \ref{lemma:formulanablaNg} and Proposition \ref{prop:thetachi}.

\end{remark}

With the notations of the Theorem \ref{theorem:mostgeneral} and Assumptions \ref{ass10}, take $\cO_K$-valued points $\alpha$ (resp.~$\beta$) of $\Lambda_I$ so that
the induced weight $k_0+v$ (resp.~$s_0+w$) specialize to classical weights $v_0\in\N$ and $w_0\in\N$. Let $g_{\alpha}$ be the specialization of $g$ at $\alpha$. One
deduces from the formula on $q$-expansions and using the notations of Definition \ref{def:nablachi} the following:

\begin{corollary}\label{cor:specializas}  The specialization of
$(\nabla_k)^s(g)$ at $\alpha$ and $\beta$ is $\nabla_{v_0+\chi}^{w_0+\chi'}(g_\alpha)$.

\end{corollary}

\subsection{The proof of Theorem \ref{theorem:mostgeneral}}\label{sec:proofmostgeneral}

From now on until the end of this section we will assume that the weights $k$ and $s$ satisfy the strict assumptions:

\begin{assumption}
\label{ass}
There are $u_k\in p\Lambda_I$ and $u_s\in q\Lambda_I$ such that $k(t)={\rm exp}(u_k\log(t))$
and $s(t)={\rm exp}(u_s\log(t))$ for all $t\in \Z_p^\ast$.
\end{assumption}

Our goal for the rest of the section is to define for all $g\in {\rm H}^0\bigl(\fX_{r,I}, \bW^0_k \bigr)^{U=0}$
$$
(\nabla_k)^s(g)={\rm exp}\Bigl(\frac{u_s}{(p-1)}{\rm
log}\bigl(\nabla_k^{(p-1)}\bigr)  \Bigr)(g).
$$

\bigskip\noindent
We start with the following

\begin{definition}\label{def:oW'}
Let us denote $\bW^{0,'}_k\subset \bW^0:=\pi_\ast\bigl(\cO_{\bV_0({\rm H}^\#_E,s)}\bigr)$ the subsheaf of $\bW^0$ defined by $\bW^{0,'}_k:=\sum_{n\in\Z}\bW^0_{k+2n}\subset
\bW^0$. We let $\bW'_k\subset \bW$ be the sheaves obtained from $\bW^{0,'}_k\subset \bW^0$ by twisting by $\fw^{k_f}$ (see Definition \ref{defi:wnrI}).

We also define the differential operator $\nabla\colon \bW'\lra \Bigl(\frac{1}{{\rm Hdg}^{c_n}}\Bigr)\cdot \bW'$ by $\nabla\vert_{\bW_{k+2n}}:=\nabla_{k+2n}\colon
\bW_{k+2n}\lra \bW_{k+2n+2}\subset \bW'_k$.
\end{definition}

The fact that inside $\bW$ we have $\bW_{k+2n}\cap\bW_{k+2n'}=\{0\}$ for $n\neq n'$ implies that $\nabla$ is well defined on $\bW'$. We start with the following
result:

\begin{proposition}
\label{prop:congord} Under the Assumption \ref{ass} above for
every $g\in {\rm H}^0(\fX_{r,I}^{\rm ord}, \oW_k)^{U=0}$ and
every positive integer $N$ we have
$$\bigl(\nabla^{p-1}-\mathrm{Id}\bigr)^{Np} \bigl(g\bigr)\in
p^N{\rm H}^0(\fX_{r,I}^{\rm ord}, \bW)\cap {\rm
H}^0\bigl(\fX_{r,I}^{\rm ord}, \bW'_k\bigr).
$$The same result applies if we replace $\fX_{r,I}$ with any layer
of the Igusa tower $\fIG_{n,r,I}$.
\end{proposition}

\begin{proof} This follows from Corollary
\ref{cor:iteratenabla} and the fact that the evaluation at the Tate curve provides an injective map ${\rm H}^0\bigl(\fX_{r,I}^{\rm ord}, \bW^0/p\bW^0\bigr)\to
\bW^0(q)/p\bW^0(q)$.

The second claim follows as the map $\fIG_{n,r,I}^{\rm ord}\to
\fX_{r,I}^{\rm ord}$ is finite \'etale.

\end{proof}

\begin{proposition} \label{prop:surconv}
Let $s$ be a non-negative integer. Then there exists a positive integer $b\geq r$ depending on $r$, $n$ and $s$ with the following property: for every $w\in {\rm
Hdg}^{-s} {\rm H}^0(\fX_{r,I}, \bW)$ such that $ w\vert_{\fX_{r,I}^{\rm ord}}\in {\rm H}^0(\fX_{r,I}^{\rm ord}, p^j \bW) $ we have $w\in {\rm
H}^0\bigl(\fX_{b,I},p^{[j/2]} \bW\bigr)$.
\end{proposition}

\begin{proof}
Recall that  $\bW:= \bW^0\otimes_{\cO_{\fX_{r,I}}} \fw^{k_f}$ where $\fw^{k_f}$ is the coherent $\cO_{\fX_{r,I}}$-module $\bigl(g_{i,\ast}
\bigl(\cO_{\mathfrak{IG}_{i,r}}\bigr)\otimes_{\Lambda^0} \Lambda\bigr) \bigl[k_{f}^{-1}\bigr] $ of Definition \ref{defi:wnrI}. It follows from Theorem
\ref{thm:descentbWk} that tensoring $\bW^0\otimes_{\cO_{\fX_{r,I}}}$ with a coherent $\cO_{\fX_{r,I}}$-modules is exact so that $\bW_k$ is also $\bW_k=\left(
\bW^0_{k}\otimes_{\cO_{\fX_{r,I}}} \bigl(g_{i,\ast} \bigl(\cO_{\mathfrak{IG}_{i,r}}\bigr)\otimes_{\Lambda^0} \Lambda\bigr)\right)\bigl[k_{f}^{-1}\bigr] $. Thus it
suffices to prove the statement locally on $\fX_{r,I}$ and replacing $\bW$ with $\bW\otimes_{\cO_{\fX_{r,I}}} \mathcal{F}$ with $\mathcal{F}:=\bigl(g_{i,\ast}
\bigl(\cO_{\mathfrak{IG}_{i,r}}\bigr)\otimes_{\Lambda^0} \Lambda\bigr)$.

Consider first the case of $\bW^0$. Let $V={\rm Spf}(S) \subset \fX_{r,I}$ and let $U={\rm Spf}(R)\subset \fIG_{n,r,I}$ be its inverse image.  For $V$ small enough
$\bW^0(U)/p^j \bW^0(U) \cong R/p^j R[ Z,Y] $ is a free $R/p^jR$-module and $\oW(U)/p^j \oW(U) \cong R^{\rm ord}/p^jR^{\rm ord}  R\bigl[ Z^{\rm ord},Y^{\rm
ord}\bigr] $ is a free $R^{\rm ord}/p^jR^{\rm ord}$-module. We may choose the variables so that the restriction map $\bW^0(U)/p^j \bW^0(U) \lra \oW(U)/p^j \oW(U)$
sends $Y\mapsto Y^{\rm ord}$ and $Z\mapsto p^n \beta_n^{-1} Z^{\rm ord}=\udelta^{p^n-1} Z^{\rm ord}$. It follows from Lemma \ref{lemma:kerloc}  that the kernel of
$\bW^0(U)/p^j \bW^0(U) \lra \oW(U)/p^j \oW(U) $ is annihilated by $\udelta^{i p^{r+1} (p-1) + p^n-p} $. This implies that the kernel of $\bW^0(V)/p^j \bW^0(V) \lra
\oW(V)/p^j \oW(V) $ is annihilated by $\udelta^{j p^{r+1} (p-1) + p^n-p} $ since $p^j\bW^0(V)=\bigl(p^j\bW^0(U)\bigr) \cap \bW^0(V)$ by construction. The same then
applies for the kernel of $\bW^0(U)/p^j \bW^0(U) \lra \bW^{0,\rm ord}(U)/p^j \bW^{0,\rm ord}(U) $ as  $p^j \bW^0(U)=\bigl(p^j \bW^0(V)\bigr) \cap \bW^0(U)$ by
construction.

As explained in the proof of Lemma \ref{lemma:kerloc}, the morphism $g_1\colon \mathfrak{IG}_{1,r}\to \fX_{r,I}$ is flat and $\mathfrak{IG}_{2,r} \to
\mathfrak{IG}_{1,r}$ can be factored via a flat morphism $ \mathfrak{IG}_{i,r}'\to \mathfrak{IG}_{1,r} $  with $\udelta^{p^i-p} \cO_{\fIG_{i,r,I}} \subset
\cO_{\fIG_{i,r,I}'}$. Since $\Lambda_I$ is a finite and free $\Lambda_I^0$-module, the tensor product of the pushforward of the structure sheaf of $
\mathfrak{IG}_{i,r}' \to \fX_{r,I}$ via $\Lambda_I^0\to \Lambda_I$ defines a finite and flat $\cO_{\fX_{r,I}}$-module $\mathcal{G}$ such that  $\udelta^{p^i-p}
\mathcal{F}\subset \mathcal{G} \subset \mathcal{F}$.

This implies that the kernel of $\bigl(\bW^0(V)/p^j \bW^0(V)\bigr)\otimes_R \mathcal{F}(V) \lra \bigl(\oW(V)/p^j \oW(V)\bigr)\otimes_R \mathcal{F}(V)  $ is
annihilated by $\udelta^{j p^{r+1} (p-1) + p^n-p + p^i-p} $. We also conclude that $\udelta^{j p^{r+1} (p-1)   + p^n-p + p^i-p} w \in p^j \bW(V)$ as $\udelta$ is
invertible in $R^{\rm ord}$

Then passing to $\fX_{b,I} $ with $b$ such that $(p-1)p^{b+1} \geq 2 p^{r} (p-1)  + 2(p^n-p) + 2(p^i-p) $ and considering the open $V_b:={\rm Spf}(S_b)$
corresponding to $V$ we have that $a =p \udelta^{-(p-1) p^{b+1}}\in S_b$ and

$$
\frac{p^2}{\udelta^{2 p^{r+1} (p-1)   + 2(p^n-p) + 2(p^i-p)}}=p a \bigl(\udelta^{(p-1) p^{b+1} - 2 p^{r+1} (p-1)  - 2(p^n-p)-2(p^i-p)} \bigr)\in p S_b.
$$
Hence, if we denote by $w'$ the image of $w$ in ${\rm H}^0\bigl(V_b, \bW\bigr)$,  we have that $w'  \in p^{[j/2]} \bW(V_b)$ as claimed.

\end{proof}

Consider now the connection  $\nabla\colon \bW'\lra \Bigl(\frac{1}{{\rm Hdg}^{c_n}}\Bigr)\cdot \bW'$ over $\fX_{r,I}$ defined in Section \ref{sec:GMwk}. We have the
following key result:

\begin{corollary}
\label{cor:mainiterations} There exists an integer $b$ depending on $r$ and $n$ such that for every $g\in {\rm H}^0(\fX_{r,I}, \bW_k)^{U=0}$ and every positive
integer $N$ we have
$${\rm Hdg}^{c_n(p-1)^2} \bigl(\nabla^{p-1}-\mathrm{Id}\bigr)^{N}(g)\subset
p^{[N/2p]}{\rm H}^0\bigl(\fX_{b,I}, \bW\bigr)\cap {\rm H}^0\bigl(\fX_{b,I}, \bW'_k\bigr)$$and there exists a positive integer $\gamma$ (depending on $r$, $n$ and
$p$) such that, given positive integers $h$ and $j_1,\ldots,j_h$ with $N=j_1+\cdots + j_h$, then

$${\rm Hdg}^{\gamma} \frac{p^{h'}}{h!} \Bigl(\prod_{a=1}^h \frac{\bigl(\nabla^{(p-1)}-{\rm
Id}\bigr)^{j_a}}{j_a} \Bigr)(g)\subset p^{[N/2p^2]}{\rm H}^0\bigl(\fX_{b,I}, \bW\bigr)\cap {\rm H}^0\bigl(\fX_{b,I}, \bW'_k\bigr) ;$$here $h'=h$ if $p\neq 2$ and
$h'=2h$ if $p=2$.

\end{corollary}
\begin{proof}
Note that $${\rm Hdg}^{c_nN(p-1)} \bigl(\nabla^{p-1}-\mathrm{Id}\bigr)^{N}\colon  \bW' \lra \bW'\subset \bW
$$is well defined, i.e., it does not have poles.
Write $N=p[N/p]+N_0$ the division with reminder of $N$ by $p$.
Then

$${\rm Hdg}^{c_nN(p-1)} \bigl(\nabla^{p-1}-\mathrm{Id}\bigr)^{N}(g)=
{\rm Hdg}^{c_nN_0(p-1)} {\rm Hdg}^{c_np(p-1)[N/p]}
\bigl(\nabla^{p-1}-\mathrm{Id}\bigr)^{[N/p]}
\bigl(\bigl(\nabla^{p-1}-\mathrm{Id}\bigr)^{N_0}(g)  \bigr)$$

We then deduce from  Proposition \ref{prop:congord} that the
restriction to the ordinary locus belongs to $p^{[N/p]} {\rm
H}^0\bigl(\fX_{r,I}^{\rm ord}, \bW\bigr)$. Thanks to Proposition
\ref{prop:surconv} there exists $b$, depending on $n$, $r$ and $p$
such that
$${\rm Hdg}^{c_nN_0(p-1)}
\bigl(\nabla^{p-1}-\mathrm{Id}\bigr)^{[N/p]} \bigl(\bigl(\nabla^{p-1}-\mathrm{Id}\bigr)^{N_0}(g)  \bigr)\in p^{[N/2p]} {\rm H}^0\bigl(\fX_{b,I}, \bW\bigr).$$As
$N_0\leq p-1$, the first claim follows.

We prove the second claim. Write $N=p[N/p]+N_0$ Also in this case $$ {\rm Hdg}^{c_np(p-1)[N/p]+ c_n (p-1)N_0} \bigl(\nabla^{p-1}-\mathrm{Id}\bigr)^{p[N/p]} \circ
\bigl(\nabla^{p-1}-\mathrm{Id}\bigr)^{N_0}(g) \colon \bW' \lra \bW'$$over $\fX_{r,I}$ is integral. Over the ordinary locus the image of $g$ is zero modulo
$p^{[N/p]}$ thanks to Proposition \ref{prop:congord}. Arguing as in the proof of Proposition \ref{prop:surconv} there exists a linear function $\ell(X)=\alpha X +
\beta$ with $\alpha$ and $\beta$ positive integers depending on $r$, $n$ and $p$ such that
$${\rm Hdg}^{\ell([N/p])}
\bigl(\nabla^{p-1}-\mathrm{Id}\bigr)^{p[N/p]}\bigl(\bigl(\nabla^{p-1}-\mathrm{Id}\bigr)^{N_0}(g)\bigr)\in p^{[N/p]}{\rm H}^0\bigl(\fX_{r,I}, \bW\bigr)$$and hence
${\rm Hdg}^{\ell([N/p])} w\in p^{[N/p^2]}{\rm H}^0\bigl(\fX_{r,I}, \bW\bigr)$with  $$ w:=\frac{p^h}{h! j_1\cdots j_h}
\bigl(\nabla^{p-1}-\mathrm{Id}\bigr)^{p[N/p]}\bigl(\bigl(\nabla^{p-1}-\mathrm{Id}\bigr)^{N_0}(g)\bigr)
$$thanks to Lemma \ref{lemma:estimate}. Replacing $\ell([N/p])$ with
$\ell'([N/p^2]):=p \alpha [N/p^2]+ \gamma$ with $\gamma:=(p-1)
\alpha+ \beta$, noticing that $\ell'([N/p^2])\geq \ell([N/p])$ and
arguing as the proof of Proposition \ref{prop:surconv} we find a
positive integer $b$ depending on $\alpha$ and $\beta$, and hence
on $r$ and $n$, such that ${\rm Hdg}^{\gamma} w\in
p^{[N/2p^2]}{\rm H}^0\bigl(\fX_{b,I}, \bW\bigr)$, concluding the
proof of the Corollary.

\end{proof}

\begin{lemma}\label{lemma:estimate} Let $j_1,\ldots,j_h$ be positive integers and write $N=j_1 +\cdots + j_h$. Then we have
$\delta h+ \frac{N}{p}-\sum_{i=1}^h v_p(j_i)-v_p(h!)\geq \frac{N}{p^2}$ with $\delta=1$ if $p\geq 3$ and $\delta=2$ if $p=2$.
\end{lemma}
\begin{proof}
Write $h=h_0 + \cdots + h_t p^t$ for the $p$-adic expansion of $h$. Then the $p$-adic valuation $v_p$ of $h!$ is
$$v_p(h!)=\frac{h-(h_0+\cdots + h_t)}{p-1}\leq
\frac{h}{p-1}.$$It suffices to prove that $\delta h+ \frac{N}{p}-\sum_i v_p(j_i)-\frac{h}{p-1}\geq \frac{N}{p^2}$ with $\delta=1$ if $p\geq 3$ and $\delta=2$ for
$p=2$. As $N=j_1 +\cdots + j_h$ it suffices to prove the claim for $h=1$, i.e., that for every positive integer $j$
$$\delta+ \frac{(p-1)}{p^2} j \geq  v_p(j)+ \frac{1}{p-1}.$$

If $v_p(j)=0$ this is holds for any $\delta\geq 1$. Else write $j=\gamma p^r$ with $p$ not dividing $\gamma$ and $r\geq 1$, then the inequality becomes  $$\delta+
\frac{(p-1)}{p} \gamma p^{r-1} \geq r+ \frac{1}{p-1} .$$It suffices to prove it for $\gamma=1$.  If $p\geq 3$,  we can take $\delta=1$ as $ p^{r-1} \geq r+1$ for
$r\geq 1$. If $p=2$,  we can take $\delta=2$ as $2^{r-1}\geq r$ for every $r\geq 1$. This concludes the proof of the Claim.

\end{proof}

\begin{proposition}
\label{prop:nablas} The notations are  as in Corollary \ref{cor:mainiterations} and let $s\colon\Z_p^\ast\lra (\Lambda_{I_s}^0)^\ast$ that satisfies the Assumption
\ref{ass}. Then, there exist  positive integers $\gamma$, $b$ depending on $r$, $n$ and $p$ such that  for every $g\in {\rm H}^0(\fX_{r,I}, \bW_k)^{U=0}$, the
sequences

$$A(g,s)_n:=\sum_{j=1}^{n+1}\frac{(-1)^{j-1}}{j} \bigl(\nabla^{p-1}-{\rm
Id}\bigr)^j(g)$$and, if we write $H_{i,n}$ for the set of $i$-uple
$(j_1,\ldots,j_i)$ of positive integers having $j_1+\cdots +
j_i\leq n+1$,
$$
B(g,s)_n:=\sum_{i=0}^n\frac{1}{i!} \frac{u_s^i}{(p-1)^i}
\Bigl(\sum_{(j_1,\ldots,j_i)\in  H_{i,n}} \bigl( \prod_{a=1}^i
\frac{(-1)^{j_a-1}}{j_a} \bigr) \bigl( \nabla^{p-1}-{\rm
Id}\bigr)^{j_1+\cdots + j_i}\Bigr)(g), \mbox{ for } n\ge 0
$$
converge in ${\rm Hdg}^{-\gamma} {\rm
H}^0(\fX_{b,I}\widehat{\otimes} \Lambda_{I_s}, \bW)$. Moreover if
we denote the limits
$${\rm lim}_{n\rightarrow \infty} A(g,s)_n=:{\rm
log}(\nabla_k^{p-1})(g)$$ and
$${\rm lim}_{n\rightarrow \infty} B(g,s)_{n}=:{\rm exp}\Bigl(\frac{u_s}{(p-1)}{\rm log}(\nabla^{p-1}_k)\Bigr)(g)=:\nabla^s_k(g),$$
we have that $\nabla^s_k(g)\in {\rm Hdg}^{-\gamma} {\rm H}^0(\fX_{b, I}\widehat{\otimes} \Lambda_{I_s}, \bW_{k+2s})$. Finally on $q$-expansions we have
$$\nabla_k^s\bigl(g\bigr)(q):=\sum_h \sum_{j=0} \left(
\begin{array}{cc} u_s \\ j
\end{array} \right)\prod_{i=0}^{j-1}(u_k+u_s-h-1-i)
\partial^{s-j}\bigl(g_h(q)\bigr)V_{k+2s,j+h}.$$
\end{proposition}

\begin{proof}
The first convergence follows immediately from the first claim of Corollary \ref{cor:mainiterations}. Therefore ${\rm log}\bigl(\nabla_k^{p-1}\bigr)(g) $ converges
$p$-adically in ${\rm Hdg}^{-c_n(p-1)^2} {\rm H}^0\bigl(\fX_{b,I}, \bW\bigr)$.

We prove the second claim. Thanks to Corollary
\ref{cor:mainiterations} we have a positive integer $\gamma$ such
that for positive integers $h$, $N$ and $j_1,\ldots,j_h$ with $j_1
+\cdots + j_h=N$ we have

$$\mathrm{Hdg}^\gamma \frac{p^h
\bigl(\nabla^{p-1}-\mathrm{Id}\bigr)^{j_1+\ldots + j_h}(g)}{h!j_1\cdots j_h } \subset p^{[N/2p^2]}{\rm H}^0\bigl(\fX_{b,I}, \bW\bigr)\cap {\rm H}^0\bigl(\fX_{b,I},
\bW'_k\bigr).$$In particular the series $B(g,s)_m-B(g,s)_n$ for $m\geq n$ lie in $p^{[n+1/2p^2]}{\rm H}^0\bigl(\fX_{b,I}, \bW\bigr)$. i.e., $B(g,s)_n$ is a Cauchy
sequence for the $p$-adic topology and in particular converges.

To see that $\nabla_k^s(g)$ belongs to ${\rm H}^0(\fX_{b, I}, \bW_{k+2s})$ it is enough to see how a section of the torus $\fT^{\rm ext}$ acts on this section of
$\bW$. By density it is enough to see how an element $t\in \Z_p^\ast$ acts. As $t\ast \nabla_k(\_)=t^2 \nabla_k(t\ast \_)$ and $t\ast g=t^kg$, we obtain: $t\ast
\nabla_k^s(g)= t^{2s}\nabla_k^s(t^k g)$. This proves the claim.

It remains to show the claim on $q$-expansions. Assume first that $s$ is an integral weight. Then $B(g,s)_{n}$ converges $p$-adically to ${\rm
exp}\Bigl(\frac{u_s}{(p-1)}{\rm log}(\nabla^{p-1}_k)\Bigr)(g)$ which is $\nabla_k^s(g)$. Its $q$-expansion coincides with the one claimed in the Proposition thanks
to Lemma \ref{lemma:formulanablaNg}. In the general case, consider the coefficients $\sum_n c_n(q) V_{k+2s,n} $ of the $q$-expansion of $\nabla_k^s(g)$ and the
coefficients of $\sum_n b_n(q) V_{k+2s,n}$ with $$b_n(q):= \sum_{h,j, h+j=n}  \left(
\begin{array}{cc} u_s \\ j
\end{array} \right)\prod_{i=0}^{j-1}(u_k+u_s-h-1-i)
\partial^{s-j}\bigl(g_h(q)\bigr).$$For every $n$ both $c_n(q)$ and $b_n(q)$
are functions with values in $R \widehat{\otimes} \Lambda_{I_s}$, where $R=\Lambda_{I_k}(\!(q)\!)$ is the completed local ring at the cusp. For every $n$ the
coefficients in the $q$-expansion of $c_n$ and $b_n$ lie in $\Lambda_{I_k} \widehat{\otimes} \Lambda_{I_s}$ and coincide for all the integral specializations of
$u_s$, i.e., for infinitely many points. Hence they coincide. The claim follows.
\end{proof}

\section{Applications to the construction of the triple product $p$-adic $L$-function in the finite slope case.}\label{sec:triple}

In the first two sections of this chapter, in which we recall the known
construction of the triple product $p$-adic $L$-function attached to
a triple of Hida families, we follow closely the exposition in
Section \S 4 of \cite{darmon_rotger}.

\noindent Let $f$ be a newform of level $N_f$, character $\chi_f$ and let $\Q_f$ denote the number field generated by all Hecke eigenvalues of $f$. We write $f\in
S_k(N_f,\chi_f, \Q_f)$. We denote by $\pi_f$ the automorphic representation of $\GL_2(\bA_\Q)$ generated by $f$. If $N$ is a multiple of $N_f$ and $\Q_f\subset K$
we let $S_k(N, K)[\pi_f]$ denote the $f$-isotypic subspace of $S_k(N, K)$ attached to the automorphic representation $\pi_f$. For every divisor $a$ of $N/N_f$
consider the elements $[a]^\ast(f)$ of $S_k(N, K)[\pi_f]$ defined by pull--back via the morphism $[a]$ from the modular curve of level $\Gamma_1(N)$ to the modular
curve of level $\Gamma_1(N_f)$ given as follows. Take an elliptic curve $E$ with cyclic subgroup $H_N$. Let $H_{a N_f }$, resp.~$H_a$ be the kernel of
multiplication by $a N_f$, resp.~$a$ on $H_N$. Then $[a]\bigl(E,H_N\bigr)=\bigl(E',H_{N_f}'\bigr)$ with $E':=E/H_a$ and $H_{N_f}'=H_{a N_f}/H_a$. Note that
$[a]^\ast(f)=f(q^a)$. Then, as recalled in loc.~cit.:

\begin{lemma} \label{lemma:basis} The space $S_k(N, K)[\pi_f]$ is a finite dimensional $K$-vector space of dimension
$\displaystyle \sigma_0\bigl(\frac{N}{N_f}\bigr)$, where $\sigma_0(n)=\#\{d\ | \ d|n\}$, and a basis of $S_k(N, K)[\pi_f]$ is given by $\displaystyle
\bigl\{[a]^\ast(f)\bigr\}_{a\vert\frac{N}{N_f}}$.\end{lemma}

Let $f$, $g$, $h$ be a triple of normalized primitive cuspidal classical eigenforms of weights $k$, $\ell$, $m$, characters $\chi_f$, $\chi_g$, $\chi_h$ and tame
levels $N_f$, $N_g$, $N_h$ respectively. We write $f\in S_k(N_f,\chi_f)$, $g\in S_\ell(N_g,\chi_g)$, $h\in S_m(N_h, \chi_h)$. We set
$N:=\mathrm{\ell.c.m}(N_f,N_g,N_h)$, $\Q_{f,g,h}:=\Q_f\cdot\Q_g\cdot\Q_h$ the number field generated over $\Q$ by the Hecke eigenvalues of $f$, $g$, $h$. We assume
that $\chi_f\cdot\chi_g\cdot\chi_h=1$ and the triple of weights $(k,\ell,m)$ is unbalanced, i.e., there is $t\in \Z_{\ge 0}$ such that $k=\ell+m+2t$. We have the
following result of M.~Harris and S.~Kudla (\cite{harris_kudla}), previously conjectured by H.~Jacquet and recently refined by A.~Ichino (\cite{ichino}) and
T.C.~Watson (\cite{watson}):

\begin{theorem}[Theorem 4.2. \cite{darmon_rotger}]\label{thm:fghcirc}
Let $f$, $g$, $h$ be a triple as at the beginning of this section. Then there exist:

$\bullet$ holomorphic modular forms $$f^o\in S_k(N,
\Q_{f,g,h})[\pi_f], ,g^o\in S_\ell(N, \Q_{f,g,h})[\pi_g], h^o\in
S_m(N,\Q_{f,g,h})[\pi_h]$$

$\bullet$ for each $q\vert N\infty$ a constant $C_q\in \Q_{f,g,h}$, which only depends on the local components at $q$ of $f^o,g^o,h^o$ such that
$$
\frac{\prod_{q|N\infty}C_q}{\pi^{2k}} L\Bigl(f,g,h,
\frac{k+\ell+m-2}{2}\Bigr)=|I(f^o,g^o,h^o)|^2.
$$
Moreover, there is a choice of $f^o$, $g^o$, $h^o$ such that all $C_q\neq 0$.

\end{theorem}

In the above theorem $L(f,g,h,s)$ is the complex Garrett-Rankin triple product $L$-function attached to $f$, $g$, $h$ and
$$
I(f^o,g^o,h^o):=\langle (f^o)^\ast, \delta^t(g^o)h^o\rangle,
$$
where $\langle \ ,\ \rangle$ is the Peterson inner product on weight $k$-modular forms, $\delta$ is the Shimura-Maass differential operator and
$(f^o)^\ast=f^o\otimes \chi_f^{-1}$ is an eigenform having prime-to-$N$ eigenvalues equal to those of $f^o$, twisted by the character $\chi_f^{-1}$.

\subsection{The triple product $p$-adic $L$-function in the ordinary case.}\label{sec:tripleordinary}

Let $f$, $g$, $h$ be as at the beginning of Section \ref{sec:triple} with the additional assumption that $f$, $g$, $h$ are ordinary at $p$. Let $f^o$, $g^o$, $h^o$
be as in Theorem \ref{thm:fghcirc} such that all constants $C_q$ for $q\vert N\infty$ are non-zero. Let $\bof$, $\bg$, $\bh$ be Hida families of modular forms on
$\Gamma_1(N)$ (seen as $q$-expansions with coefficients in the finite flat extensions of $\Lambda$ denoted $\Lambda_f$, $\Lambda_g$, $\Lambda_h$ respectively)
deforming the ordinary $p$-stabilizations of $f$, $g$, $h$ in the weights $k$, $\ell$, $m$ respectively. As explained in \cite[\S 2.6]{darmon_rotger} the families
$\bof$, $\bg$, $\bh$ determine  Hida families $\bof^o$, $\bg^o$ and $\bh^o$ deforming the ordinary $p$-stabilization of $f^o$, $g^o$ and $h^o$ respectively. Define
also $({\bg^o})^{[p]}$, the $p$-depletion of $(\bg)^o$, on $q$-expansions by: if ${\bg^o}(q)=\sum_{n=1}^\infty a_nq^n$, then $({\bg^o})^{[p]}(q):=\sum_{n=1,
(n,p)=1}^\infty a_nq^n$. We then have

\begin{definition}[Definition 4.4 \cite{darmon_rotger}]\label{def:padiclord}
The Garrett-Rankin triple product $p$-adic $L$-function attached
to the triple $(\bof^o,\bg^o, \bh^o)$ of Hida families is the
element

$$
\mathcal{L}_p^f\bigl(\bof^o,\bg^o,\bh^o\bigr):=\frac{\langle (\bof^o)^\ast, e^{\rm ord}\bigl(d^\bullet(\bg^o)^{[p]}\times \bh^0\bigr)  \rangle}{\langle \bof^\ast,
\bof^\ast \rangle}\in \Lambda_f'\otimes_\Lambda(\Lambda_g\otimes\Lambda_h\otimes \Lambda).
$$
\end{definition}

The $p$-adic $L$-function $\mathcal{L}_p^f\bigl(\bof^o,\bg^o,\bh^o\bigr)$ in Definition \ref{def:padiclord} is a function of three weight variables. In particular
if $x$, $y$, $z\in W$ are classical weights which are unbalanced and if we denote by $t\ge 0$ the integer such that $x=y+z+2t$ then we have (see section 4 of
\cite{darmon_rotger})
$$
\mathcal{L}_p^f\bigl({{\bof^o,\bg^o,\bh^o}}\bigr)(x,y,z)=\frac{\langle \bigl(\bof_x^o\bigr)^\ast, e^{\rm ord}\bigl(d^t(\bg_y^o)^{[p]}\times \bh_z^0\bigr)
\rangle}{\langle \bof_x^\ast, \bof_x^\ast\rangle}.
$$
Thanks to \cite[Thm.~4.7]{darmon_rotger} the above value of the $p$-adic $L$-function at $x=k$, $y=\ell$, $z=m$ is related to the classical $L$-function via
$$
\mathcal{L}_p^f\bigl({{\bof^o,\bg^o,\bh^o}}\bigr){{(k,\ell,m)}}=\times{\Bigl(L^{\rm alg}\bigl(f,g,h, \frac{k+\ell+m-2}{2}\bigr)  \Bigr)^{\frac{1}{2}}}$$for some
non-zero constant $\times$. It follows from Theorem \ref{thm:fghcirc} that the $p$-adic $L$-function $\mathcal{L}_p^f\bigl(\bof^o,\bg^o,\bh^o\bigr)$ is non-zero if
the value of the special value of the classical $L$-function is non zero; see \cite[Rmk 4.8]{darmon_rotger}.

\subsection{The triple product $p$-adic $L$-function in the finite slope case.}\label{sec:tripleplfs}

Let $f\in S_k(N_f,\chi_f)$, $g\in S_\ell(N_g, \chi_g)$, $h\in S_m(N_h,\chi_h)$ be a triple of normalized primitive cuspidal eigenforms  such that $f$ has finite
slope $a$ and $\chi_f\cdot\chi_g\cdot\chi_h=1$ and assume $(k,\ell, m)$ is unbalanced and denote $t$ a non-negative integer such that $k=\ell+m+2t$. Let
$N:=\mathrm{\ell.c.m}(N_f,N_g,N_h)$ and let $f^o$, $g^o$, $h^o$ be as in Theorem \ref{thm:fghcirc} such that all constants $C_q$ are non-zero. In particular $f^o$
has finite slope $a$. We denote by $K$ a finite extension of $\Q_p$ which contains all the values of $\chi_f$, $\chi_g$, $\chi_h$.

Let $\omega_f$, $\omega_g$, $\omega_h$ denote overconvergent families of modular forms deforming $f$, $g$, $h$ and let $\omega_f^o$, $\omega_g^o$ and $\omega_h^o$
be the overconvergent families deforming $f^{o}$, $g^{o}$, $h^{o}$ and associated to $\omega_f$, $\omega_g$ and $\omega_h$ via the procedure described in
\cite[\S2]{darmon_rotger}; for example if we express $f^o$  as a $K$-linear combination $\sum_a \lambda_a \cdot [a]^\ast(f)$ of the basis elements $[a]^\ast(f)$'s,
for $a$ varying among the divisors of $N/N_f$, provided in Lemma \ref{lemma:basis}, then $ \omega_f^o:=\sum_a \lambda_a\cdot [a]^\ast\bigl(\omega_f\bigr)$.

We then have a non-negative integer $r$, closed intervals $I_f$, $I_g$ and $I_h$ such that the weights of these families, denoted respectively $k_f\colon
\Z_p^\ast\to \Lambda_{I_f,K}^\ast$, $k_g\colon \Z_p^\ast\to \Lambda_{I_g,K}^\ast$, $k_h\colon \Z_p^\ast\to \Lambda_{I_h,K}^\ast$ are all adapted to a certain
integer $n\ge 0$. This data gives a tower of formal schemes $\fIG_{n,r,I}\lra \fX_{r,I}\lra \fX$, where $\fX$ is the formal completion along its special fiber of
the modular curve $X_1(N)_{\Z_p}$ and $\fX_{r,I}=\fX_{r,I_f} \times_{\fX} \fX_{r,I_g} \times_{\fX} \fX_{r,I_h}$ and likewise for $\fIG_{n,r,I}$. We denote by
$\fw^{k_f}, \fw^{k_g}, \fw^{k_h}$ the respective modular sheaves (over $\fX_{r,I}$ or on the analytic adic fiber $\cX_{r,I}$), then $\omega_f$, $\omega_f^o\in {\rm
H}^0(\cX_{r,I_f}, \fw^{k_f})$, similarly $\omega_g$, $\omega_g^o\in {\rm H}^0(\cX_{r,I_g}, \fw^{k_g})$ and $\omega_h$, $\omega_h^o\in {\rm H}^0(\cX_{r,I_h},
\fw^{k_h})$. We make the following assumption on the weights of $\omega_f^o$, $\omega_g^o$, $\omega_h^o$: \

\begin{assumption}
\label{ass1}

 1) Suppose that the weights $k_f$, $k_g$, $k_h$ are such that $k_f-k_g-k_h$ is even, i.e. there is a weight $u\colon\Z_p^\ast\lra (\Lambda_{I,K})^\ast$ such that
$2u=k_f-k_g-k_h.$

2) the weights $k_g$, $u$ (in this order) satisfy the Assumption \ref{ass10}, i.e. $k_g=\ell\cdot \chi_g\cdot k'$ and $u=t\cdot\epsilon \cdot s$ where $\epsilon$
is a finite order, even character of $\Z_p^\ast$ and $k'$, $s$ are weights such that $k'(\eta)={\rm exp}(u_{k'}\log(\eta))$, $s(\eta)={\rm exp}(u_s\log(\eta))$, for
all $\eta\in \Z_p^\ast$ with $u_{k'}\in p\Lambda_I$, $u_s\in q\Lambda_I$.
\end{assumption}
\

We see $\omega_f$, $\omega_f^o$, $\omega_g$, $\omega_g^o$, $\omega_h$, $\omega_h^o$ as global sections of ${\rm Fil}_0(\bW_{k_f}^{\rm an})$, ${\rm
Fil}_0(\bW_{k_g}^{\rm an})$ and ${\rm Fil}_0(\bW_{k_h}^{\rm an})$ respectively. Let $\omega_g^{o,[p]}$ be the $p$-depletion of $\omega^o_g$ as in Definition
\ref{def:pdepletion}. Then Assumption \ref{ass10} implies via Theorem \ref{theorem:mostgeneral} that $(\nabla_{k_g})^u(\omega_g^{o,[p]})$ makes sense and
$$(\nabla_{k_g})^u(\omega_g^{o,[p]})\in {\rm H}^0(\cX_{r',I},
\bW^{\rm an}_{k_g+2u}),$$ for some positive integer $r'\ge r$.
Therefore we have a section
$$(\nabla_{k_g})^u\bigl(\omega_g^{o,[p]} \bigr)\times \omega^o_h\in {\rm
H}^0(\cX_{r',I_u}, \bW^{\rm an}_{k_f}).$$Consider its class  in ${\rm H}^1_{\rm dR}\bigl(\cX_{r',I_u}, \bW_{k_f-2} \bigr)$, which after base change to $\fK_f$,
where  $\fK_f$ is obtained from $\Lambda_{I_f,K}$ by inverting the elements $\{u_s-n\vert n\in \N\}$, we obtain a section
 in ${\rm H}^0(\cX_{r',I_u}, \fw^{k_f})\otimes_{\Lambda_{I_f,K}}\fK_f$. Using  Definition
\ref{def:holomorphicproj} and the spectral theory of the
$U$-operator on ${\rm H}^0\bigl(\cX_{r',I_u}, \fw^{k_f} \bigr)$
developed in  \cite[Appendice B]{halo_spectral}  we have its
overconvergent projection onto the slope $\le a$ subspace:

$$H^{\dagger, \le
a}\left((\nabla_{k_g})^u\bigl(\omega_g^{o,[p]} \bigr)\times \omega^o_h\right)\in {\rm H}^0\bigl(\cX_{r',I_u}, \fw^{k_f} \bigr)^{\le a}\otimes_{\Lambda_{I_f}}
{\fK_f}.$$

{\it The family $\omega_f^{o,\ast}$:} In order to define triple product $L$-functions we need to pass from the family $\omega^o_f$ to a different family  $\omega_f^{o,\ast}\in {\rm
H}^0(\cX_{r,I_f}, \fw^{k_f})$, with the property that   for any classical specialization of $\omega^o_f$ which is an eigenform of conductor prime to $p$  the specialization of $\omega_f^{o,\ast}$ is also an eigenform of conductor prime to $p$ with prime-to-$N$ Hecke eigenvalues twisted by $\chi_f^{-1}$. For this reason one also writes  $\omega^o_f\otimes \chi_f^{-1}$ for $\omega_f^{o,\ast}$.

We follow \cite[Lemma 5.2]{BDP} and \cite[Lemma 5.1]{BSV}.  Possibly after base change from $\Z_p$  to the ring of integers of a finite extension of $\Z_p$ we may assume that $\Lambda$ contains a primitive $N$-th root of unity $\zeta$. This allows to define an Atkin-Lehner involution $w_N$ on $X_1(N)$: given an elliptic curve $E$ with a cyclic subgroup of $\psi_N\colon \Z/N\Z\subset E[N]$ we let $w_N(E,\psi_N)$ be the elliptic curve $E'$, quotient of $E$ by the image $H_N$ of $\psi_N$, with subgroup $H_N':= E[N]/H_N$ trivialized by identifying $H_N'$ with the Cartier dual $H_N^\vee$, identifying $H_N^\vee$ with $\mu_N$ using the dual of $\Psi_N\colon \Z/\N\Z\cong H_N$ and using the chosen $N$-th root of unity  to provide an isomorphism $\Z/N\Z\cong \mu_N$. Such involution extends to an involution on $\cX_{r,I}$, $\bW_k$ etc. We let $\omega_f^{o,\ast}:=w_N\bigl(\omega_f^{o}\bigr)$. As explained in loc.~cit.~it has

\begin{definition}\label{def:finiteslopel}
The Garrett-Rankin triple product $p$-adic $L$-function attached to the triple $\bigl(\omega_f^o,\omega_g^o,\omega_h^o\bigr)$ of $p$-adic families of modular forms,
of which $\omega_f^o$ has finite slope $\le a$, is
$$
\mathcal{L}^f_p\bigl(\omega^o_f, \omega^o_g, \omega^o_h \bigr):=\frac{\langle \omega_f^{o,\ast}, H^{\dagger, \le a}\left((\nabla_{k_g})^u\bigl(\omega_g^{o,[p]}
\bigr)\times \omega^o_h\right) \rangle}{\langle\omega_f^\ast, \omega_f^\ast \rangle}\in \fK_f\widehat{\otimes} \Lambda_{k_g,K} \widehat{\otimes} \Lambda_{k_h,K} .
$$

\end{definition}

We refer to \cite[\S 4.2.1]{UNO} for the Petersson inner product in this context; see also the discussion below. By the definition of the overconvergent projection in Definition \ref{def:holomorphicproj} the $p$-adic $L$-function $\mathcal{L}_p^f(\omega^o_f,\omega^o_g,\omega^o_h)$ has only finitely many poles, i.e., it is meromorphic.

\smallskip

{\it On the Petersson product for families of overconvergent forms:}\enspace    Consider the space $M=H^0(\cX_{r,I},\fw^k)^{\leq a}$ defined over an affinoid $\cW_I:=\Spm A$ of the weight sapce with total ring of fractions $\mathbb{K}$. Let $\mathbb{T}$ be the subalgebra of  $\End M$ generated by the Hecke operators. It defines an open affinoid of the eigencurve and $\Spm \mathbb{T}\to \cW_I$ is finite and generically \'etale. 
Thus we have a trace map $\mathbb{T} \to \mathbb{T}^\vee:={\rm Hom}_A(\mathbb{T},A)$ which defines an isomorphism $\iota\colon \mathbb{T}\otimes_A \mathbb{K}\cong \mathbb{T}^\vee\otimes_A \mathbb{K}$. We also have a pairing $M \times \mathbb{T} \to A $  sending a pair $(f,T)$, consisting of a form $f$ and a Hecke operator $T$, to the Fourier coefficient $a_1(f\vert T)$ in the $q$-expansion of $f\vert T$. This defines an $A$-linear map $j\colon M\to \mathbb{T}^\vee$.
The Petersson product is defined as the composite $$\langle,\rangle\colon M \times M \stackrel{j\times j}{\longrightarrow} \mathbb{T}^\vee \times \mathbb{T}^\vee \to \mathbb{T}^\vee\otimes_A \mathbb{K} \times \mathbb{T}^\vee\otimes_A \mathbb{K} \stackrel{\iota^{-1}\times 1}{\longrightarrow}  \mathbb{T}\otimes_A \mathbb{K} \times \mathbb{T}^\vee\otimes_A \mathbb{K} \to \mathbb{K}$$ (the last map is defined by the natural pairing $\mathbb{T}\times \mathbb{T}^\vee\to A$).

\subsection{Interpolation properties}

Let now $x\in \cW_{I_f}$, $y\in \cW_{I_g}$, $z\in \cW_{I_g}$ be a triple of unbalanced classical weights, i.e., such that $x$,  $y$ and $z$ are obtained by specializing
$k_f$,  $k_g$ and $k_z$ at integral weights in $\Z_{\ge 2}$ and there is a
classical weight $t'$ with $x-y-z=2t'$. Let us denote by $f_x$, $f_x^o$, $g_y$, $g_y^o$, $h_z$, $h_z^o$ the specializations of $\omega_f$, $\omega_f^o$, $\omega_g$,
$\omega_g^o$, $\omega_h$, $\omega_h^o$ at $x$, $y$, $z$ respectively, seen as sections over $\mathcal{X}_{r',I_u}$ of $\omega^{x}\subset {\rm
Fil}_{x-2}(\bW_{x-2}^{\rm an})={\rm Sym}^{x-2}({\rm H}_E)$, $\omega^{y}\subset{\rm Fil}_{y-2}(\bW^{\rm an}_{y-2})={\rm Sym}^{y-2}({\rm H}_E)$, $\omega^{z}\subset
{\rm Fil}_{z-2}(\bW^{\rm an}_{z-2})={\rm Sym}^{z-2}({\rm H}_E)$ respectively.  Let us denote by $\bigl(\nabla_{k_g}^u(\omega_g^{o,[p]}) \bigr)_{y,t'}$ the
specialization of $\nabla_{k_g}^u(\omega_g^{o,[p]})\in {\rm H}^0(\mathcal{X}_{r', I_g}, \bW^{\rm an}_{k_g+2u})$ at the classical weight $y+2t'$. We have:

\begin{lemma}\label{lemma:rightinterpolation}
We have
$\bigl(\nabla_{k_g}^u(\omega_g^{o,[p]})\bigr)_{y,t'}=\nabla^{t'}(g_y^{o,[p]})$,
the equality taking place in ${\rm H}^0\bigl(\cX_{r',I_u},
\bW_{y+2t'}^{\rm an}\bigr)$. In particular, $$
\mathcal{L}^f_p\bigl(\omega^o_f,\omega^o_g,\omega^o_h\bigr)(x,y,z):=\bigl(\mathcal{L}_p^f(\omega^o_f,\omega^o_g,\omega^o_h)
\bigr)_{x,y,z}=\frac{\langle f_x^{o,\ast}, H^{\dagger,
a}\Bigl(\nabla^{t'}\bigl(g_y^{o,[p]}\bigr)\times h_z^o    \Bigr)
\rangle}{\langle f_x^\ast, f_x^\ast \rangle}.
$$

\end{lemma}
\begin{proof} The first claim follows  from
Corollary \ref{cor:specializas}. The second claim follows as the specialization map commutes with the overconvergent projection and the cup product by Corollary
\ref{cor:special}.
\end{proof}

As now $t'\ge 0$ is a classical weight we can relate the right
hand side of the formula of Lemma \ref{lemma:rightinterpolation}
to more classical objects. This is the content of the present section. 

In order to do that we fix embeddings of $\overline{\Q}$ in $\C$ and $\C_p$ respectively.  We also {\it
assume} that $f_x$, $g_y$ and $h_z$ are eigenforms of level $\Gamma_1(N)$ and nebentypus $\chi_x$, $\chi_y$ and $\chi_z$ respectively and with eigenvalues $a_x$, $a_y$ and $a_z$
respectively for the operator $T_p$; that is the Hecke polynomial for $T_p$ and the eigenform $f_x$, for example, is  $X^2-a_xX + \chi_x(p) p^{x-1}$ and likewise for $g_y$ and $h_z$.

Let $\alpha_x$, $\beta_x$, $\alpha_x^\ast$, $\beta_x^\ast$, $\alpha_y$, $\beta_y$
and $\alpha_z$, $\beta_z$ be the corresponding roots of the Hecke
polynomials of $T_p$ for the forms  $f_x$, resp.~$f_x^\ast$, resp.~$g_y$,
resp.~$h_z$.  Recall that~$f_x^\ast=f_x\otimes \chi_x^{-1}$; it  has nebentypus $\chi_x^{-1}$ and its eigenvalues for $T_p$ are the complex conjugates of $\alpha_x$ and $\beta_x$. In particular $a_x^\ast$ is the complex conjugte of $a_x$. 
We {\it assume} that $\alpha_x\neq \beta_x$, $\alpha_y\neq
\beta_y$ and $\alpha_z\neq \beta_z$. In particular also  $\alpha_x^\ast\neq \beta_x^\ast$. Then we
have the following interpolation result. With the notation of
Theorem \ref{thm:fghcirc} write $$L^{\rm alg}\Bigl(f_x,g_y,h_z,
\frac{x+y+z-2}{2}\Bigr):=\frac{\Bigl(\frac{\prod_{q|N\infty}C_q}{\pi^{2k_x}}
L\Bigl(f_x,g_y,h_z, \frac{x+y+z-2}{2}\Bigr)\Bigr)^{\frac{1}{2}}}
{{\langle f_x^\ast, f_x^\ast \rangle}}.$$Following
\cite[Thm.~1.3]{darmon_rotger} define $$\mathcal{E}(g_y,h_z,T):=
\bigl(1-p^{t'}\alpha_y \alpha_z T^{-1}\bigr)
\bigl(1-p^{t'}\alpha_y \beta_z T^{-1}\bigr) \bigl(1-p^{t'}\beta_y
\alpha_z T^{-1}\bigr) \bigl(1-p^{t'}\beta_y \beta_z
T^{-1}\bigr),$$
$$\mathcal{E}_1(g_y,h_z,T):=1-p^{2t'}\alpha_y\beta_y\alpha_z\beta_z
T^{-2}, \qquad \mathcal{E}_0(S,T):=1-\frac{T}{S}$$and
$$\mathcal{E}_2(T)=1-\frac{\chi_x^{-2}(p) a_{x}^\ast T}{p^{x-1}
(p+1)}.$$

\begin{theorem}\label{thm:Interpolate} We have $\displaystyle
\mathcal{L}^f_p\bigl(\omega^o_f,\omega^o_g,\omega^o_h\bigr)(x,y,z)=$
$$=\Bigl(\frac{\mathcal{E}\bigl(g_y,h_z,\alpha_x^\ast\bigr)\mathcal{E}_2\bigl(\beta_x^\ast\bigr) }
{\mathcal{E}_0\bigl(\alpha_x^\ast,\beta_x^\ast\bigr)\mathcal{E}_1\bigl(g_y,h_z,\alpha_x^\ast\bigr)}
+
\frac{\mathcal{E}\bigl(g_y,h_z,\beta_x^\ast\bigr)\mathcal{E}_2\bigl(\alpha_x^\ast\bigr)
}{\mathcal{E}_0\bigl(\beta_x^\ast,\alpha_x^\ast\bigr)\mathcal{E}_1\bigl(g_y,h_z,\beta_x^\ast\bigr)}\Bigr)
L^{\rm alg}\Bigl(f_x,g_y,h_z, \frac{x+y+z-2}{2}\Bigr).$$
\end{theorem}

The Theorem will be proven via a series of Lemmas and
Propositions.   We start with:

\begin{lemma}\label{lemma:2.17} We have $U\bigl(\nabla^u_y\bigl(g_y^{o,[p]}
\bigr)\times V\bigl(h_z^o\bigr) \bigr)=0$ and $U\bigl(V\bigl((\nabla_{y})^u\bigl(g_y^{o} \bigr)\bigr)\times h_z^{o,[p]} \bigr)=0$.

\end{lemma}
\begin{proof} This is the analogue of \cite[lemma 2.17]{darmon_rotger}. We prove the first formula, the second one being analogous to
the first. Using Theorem \ref{theorem:mostgeneral} we have the following formula on $q$-expansion
$$\left(\nabla_{y}^{t'} \bigl(g_y^{o,[p]} \bigr)\times V\bigl(h_z^o\bigr)\right)(q)= \sum_\ell \sum_{j=0} \left(
\begin{array}{cc} {t'} \\ j
\end{array} \right)\prod_{i=0}^{j-1}(u+y-\ell-1-i)
\partial^{t'-j}\bigl(g_y{o,[p]}(q)\bigr) \cdot V\bigl(h_z^o\bigr)(q)   V_{y+2t',j+\ell}.$$As $U$ acts on
$\partial^{t'-j}\bigl(g_y^{o,[p]}(q)\bigr)  V\bigl(h_z^o\bigr)(q)
V_{y+2t',j+\ell}$ as
$p^{j+\ell}U\left(\partial^{t'-j}\bigl(g_y^{o,[p]}(q)\bigr)
V\bigl(h_z^o\bigr)(q)\right) V_{y+2t',j+\ell}$,  where
$U\bigl(\sum_n a_n q^n\bigr)= \sum_n a_{pn} q^n$, it suffices to
prove that the Fourier coefficients $a_n$ of the product
$\partial^{t'-j}\bigl(g_y^{o,[p]}(q)\bigr) V\bigl(h_z\bigr)(q)$
are zero whenever $p$ divides $n$. By construction the Fourier
coefficients $\partial^{t'-j}\bigl(g_y^{o,[p]}(q)\bigr)=\sum_n b_n
q^n$ are zero if $p\vert n$ and $V\bigl(h_z^o\bigr)(q)=\sum_n c_n
q^{pn}$. The claim is then clear.

\end{proof}

Given the roots $\alpha_y$ and $\beta_y$ of the Hecke polynomial
of $T_p$ associated to the form $g_y$, we get two associated
eigenforms for $U$, of level $\Gamma_1(Np)$, with eigenvalues
$\alpha_y$ and $\beta_y$ respectively, namely
$g_{\alpha_y}:=g_y-\beta_y V\bigl(g_y\bigr)$ and
$g_{\beta_y}:=g_y-\alpha_y V\bigl(g_y\bigr)$. These are called the
$p$-{\it stabilizations} of $g_y$. We start with the following analogue of \cite[Lemma 4.10]{darmon_rotger}:

\begin{lemma}\label{lemma:4.10}
Fix $p$-stabilizations  $g_{\alpha_y}^o$ and $h_{\alpha_z}^o$ of
$g_y^o$ and $h_z^o$ with eigenvalues $\alpha_{y}$ and $\alpha_{z}$
respectively. Then,
$$H^{\dagger,\leq a}\bigl(\nabla^{t'}\bigl(g_{y}^{o,[p]}\bigr)\times h_{z}^o\bigr)=\bigl(1-p^{t'} \alpha_y \alpha_z U^{-1}\bigr)
H^{\dagger,\leq
a}\bigl(\nabla^{t'}\bigl(g_{\alpha_y}^{o}\bigr)\times
h_{\alpha_z}^o\bigr).$$Notice that $U$ is invertible on the
slope $\leq a$ part so that the formula makes sense.
\end{lemma}
\begin{proof} This is the analogue of \cite[lemma 4.10(iv)]{darmon_rotger}. Recall from \S \ref{sec:depletion} that
$g_{\alpha_y}^{o,[p]}:=g_{\alpha_y}^{o}-
V\bigl(U(g_{\alpha_y}^{o})\bigr)$. In particular, as
$U(g_{\alpha_y}^o)=\alpha_{y} g_{\alpha_y}^o$, then
$g_{\alpha_y}^{o,[p]}=g_{\alpha_y}^o-\alpha_{y}
V\bigl(g_{\alpha_y}^o \bigr)$. We also have
$\nabla^{t'}\bigl(V\bigl(g_{\alpha_y}^o \bigr)\bigr)=p^t V\bigl(
\nabla^{t'}(g_{\alpha_y}^o) \bigr)   $ as $ \nabla \circ V= p  V
\circ \nabla$. Hence,
$$H^{\dagger,\leq a}\bigl(\nabla^{t'}\bigl(g_{\alpha_y}^{o,[p]}\bigr)\times h_{\alpha_z}^o\bigr)=H^{\dagger,\leq
a}\bigl(\nabla^{t'}\bigl(g_{\alpha_y}^{o}\bigr)\times
h_{\alpha_z}^o\bigr) -p^{t'} \alpha_y  H^{\dagger,\leq
a}\bigl(V\bigl(\nabla^{t'}\bigl(g_{\alpha_y}^{o}\bigr)\bigr)\times
h_{\alpha_z}^o\bigr).$$One computes  $$H^{\dagger,\leq
a}\bigl(V\bigl(\nabla^{t'}(g_{\alpha_y}^o) \bigr) \times
V(h_{\alpha_z}^o)\bigr)= H^{\dagger,\leq
a}\bigl(V\bigl(\nabla^{t'}(g_{\alpha_y}^o)  \times
h_{\alpha_z}^o\bigr)\bigr)=U^{-1} H^{\dagger,\leq
a}\bigl(\nabla^{t'}(g_{\alpha_y}^o) \times h_{\alpha_z}^o\bigr)
;$$the last equality follows using that $U \circ V={\rm Id}$ and
the fact that $H^{\dagger,\leq a}$ can be expressed as an entire
power series $\sum_{n\geq 1} s_n U^n$ so that $\sum_{n\geq 1}
s_n U^n \circ V = \sum_{n\geq 1} s_n U^{n-1}$. It then follows
from the second formula of Lemma \ref{lemma:2.17} that
$$H^{\dagger,\leq a}\bigl(V\bigl(\nabla^{t'}\bigl(g_{\alpha_y}^o \bigr)\bigr) \times h_{\alpha_z}^o\bigr)= \alpha_z H^{\dagger,\leq
a}\bigl(V\bigl(\nabla^{t'}\bigl(g_{\alpha_y}^o \bigr)\bigr) \times
V\bigl(h_{\alpha_z}^o\bigr)\bigr)=\alpha_z U^{-1}
H^{\dagger,\leq a}\bigl(\nabla^{t'}(g_{\alpha_y}^o) \times
h_{\alpha_z}^o\bigr).$$Assembling these formulas we get that
$$H^{\dagger,\leq
a}\bigl(\nabla^{t'}\bigl(g_{\alpha_y}^{o,[p]}\bigr)\times
h_{\alpha_z}^o\bigr)=\bigl(1-p^{t'} \alpha_y \alpha_z
U^{-1}\bigr) H^{\dagger,\leq
a}\bigl(\nabla^{t'}\bigl(g_{\alpha_y}^{o}\bigr)\times
h_{\alpha_z}^o\bigr).$$Since $h_{\alpha_z}=h_z-\beta_z V(h_z)$, it
follows using Lemma \ref{lemma:2.17} that
$$H^{\dagger,\leq a}\bigl(\nabla^{t'}\bigl(g_{y}^{o,[p]}\bigr)\times h_{z}^o\bigr)=
H^{\dagger,\leq a}\bigl(\nabla^{t'}\bigl(g_{y}^{o,[p]}\bigr)\times
h_{\alpha_z}^o\bigr)$$which is also $H^{\dagger,\leq
a}\bigl(\nabla^{t'}\bigl(g_{y}^{o,[p]}\bigr)\times
h_{\alpha_z}^{o,[p]}\bigr)$, using again the Lemma
\ref{lemma:2.17},  as $
h_{\alpha_z}^{o,[p]}=h_{\alpha_z}^{o}-\alpha_z
V\bigl(h_{\alpha_z}^{o}\bigr)$. Since
$\nabla^{t'}\bigl(g_{y}^{o,[p]}\bigr)=\nabla^{t'}\bigl(g_{y}^{o}\bigr)+p^{t'}
V\bigl(\nabla^{t'}\bigl(U(g_{y}^{o})\bigr)\bigr)$, by loc.~cit.~we
have
$$H^{\dagger,\leq
a}\bigl(\nabla^{t'}\bigl(g_{y}^{o,[p]}\bigr)\times
h_{\alpha_z}^{o,[p]}\bigr)=H^{\dagger,\leq
a}\bigl(\nabla^{t'}\bigl(g_{y}^{o}\bigr)\times
h_{\alpha_z}^{o,[p]}\bigr)=H^{\dagger,\leq
a}\bigl(\nabla^{t'}\bigl(g_{\alpha_y}^{o}\bigr)\times
h_{\alpha_z}^{o,[p]}\bigr);$$the last equality follows from
$g_{\alpha_y}^{o}=g_{y}^{o}-\beta_y V\bigl(g_{y}^{o}\bigr) $ and
the fact that $\nabla^{t'}\bigl( V(g_{y}^{o})\bigr)=p^{t'}
V\bigl(\nabla^{t'}(g_{y}^{o}) \bigr)$ so that $ H^{\dagger,\leq
a}\bigl(\nabla^{t'}\bigl(V(g_{\alpha_y}^{o})\bigr)\times
h_{\alpha_z}^{o,[p]}\bigr)=0$. Arguing in the same way backwards
we have
$$H^{\dagger,\leq a}\bigl(\nabla^{t'}\bigl(g_{\alpha_y}^{o}\bigr)\times h_{\alpha_z}^{o,[p]}\bigr)=H^{\dagger,\leq
a}\bigl(\nabla^{t'}\bigl(g_{\alpha_y}^{o,[p]}\bigr)\times
h_{\alpha_z}^{o}\bigr).$$The claim follows.

\end{proof}

We also have the following analogue of \cite[Prop.~4.11]{darmon_rotger}:

\begin{lemma}\label{prop:4.11} We have
$$ H^{\dagger,\leq a}\bigl(\nabla^{t'}\bigl(g_{y}^{o,[p]}\bigr)\times h_{z}^o\bigr)=\frac{\mathcal{E}(g_y,h_z,U)}{\mathcal{E}_1(g_y,h_z,U)}
H^{\dagger,\leq a}\bigl(\nabla^{t'}\bigl(g_{y}^{o}\bigr)\times
h_{z}^o\bigr)  .$$

\end{lemma}

\begin{proof} It follows from Lemma \ref{lemma:4.10} that $$H^{\dagger,\leq
a}\bigl(\nabla^{t'}\bigl(g_{y}^{o,[p]}\bigr)\times h_{z}^o\bigr) =
\bigl(1-p^{t'}a_y b_z U^{-1}\bigr) H^{\dagger,\leq
a}\bigl(\nabla^{t'}\bigl(g_{a_y}^{o}\bigr)\times
h_{b_z}^o\bigr)$$for $a$, $b=\alpha,\beta$.

If $h_{\alpha_z}^o$ and $h_{\beta_z}^o$ are the two $p$-stabilizations of $h_z^o$, then $h_z^o=(\alpha_z-\beta_z)^{-1} \bigl(\alpha_z h_{\alpha_z}^o-\beta_z
h_{\beta_z}^o\bigr)$ and similarly for $g_y$. Hence
$$H^{\dagger,\leq a}\bigl(\nabla^{t'}\bigl(g_{y}^{o,[p]}\bigr)\times h_{z}^o\bigr)=\bigl(1-p^{t'}\alpha_y
\alpha_z U^{-1}\bigr) \bigl(1-p^{t'}\alpha_y \beta_z
U^{-1}\bigr) H^{\dagger,\leq
a}\bigl(\nabla^{t'}\bigl(g_{\alpha_y}^{o}\bigr)\times
h_{z}^o\bigr)
$$and
$$H^{\dagger,\leq
a}\bigl(\nabla^{t'}\bigl(g_{y}^{o,[p]}\bigr)\times
h_{z}^o\bigr)=\bigl(1-p^{t'}\beta_y \alpha_z U^{-1}\bigr)
\bigl(1-p^{t'}\beta_y \beta_z U^{-1}\bigr) H^{\dagger,\leq
a}\bigl(\nabla^{t'}\bigl(g_{\beta_y}^{o}\bigr)\times h_{z}^o\bigr)
.$$Thus, using that $g_y^o=(\alpha_y-\beta_y)^{-1} \bigl(\alpha_y
g_{\alpha_y}^o-\beta_y g_{\beta_y}^o\bigr)$, we obtain
$$H^{\dagger,\leq a}\bigl(\nabla^{t'}\bigl(g_{y}^{o}\bigr)\times h_{z}^o\bigr)= (\alpha_y-\beta_y)^{-1} \alpha_y
H^{\dagger,\leq
a}\bigl(\nabla^{t'}\bigl(g_{\alpha_y}^{o}\bigr)\times
h_{z}^o\bigr)- (\alpha_y-\beta_y)^{-1} \beta_y H^{\dagger,\leq
a}\bigl(\nabla^{t'}\bigl(g_{\beta_y}^{o}\bigr)\times
h_{z}^o\bigr).$$A simple computation provides the claimed formula.

\end{proof}

\begin{proposition}\label{prop:fundprop}
We
have $\langle f_x^{o,\ast}, H^{\dagger,\leq
a}\bigl(\nabla^{t'}\bigl(g_{y}^{o,[p]}\bigr)\times h_{z}^o\bigr)
\rangle=$ $$=
\Bigl(\frac{\mathcal{E}\bigl(g_y,h_z,\alpha_x^\ast\bigr)\mathcal{E}_2\bigl(\beta_x^\ast\bigr) }
{\mathcal{E}_0\bigl(\alpha_x^\ast,\beta_x^\ast\bigr)\mathcal{E}_1\bigl(g_y,h_z,\alpha_x^\ast\bigr)}
+
\frac{\mathcal{E}\bigl(g_y,h_z,\beta_x^\ast\bigr)\mathcal{E}_2\bigl(\alpha_x^\ast\bigr)
}{\mathcal{E}_0\bigl(\beta_x^\ast,\alpha_x^\ast\bigr)\mathcal{E}_1\bigl(g_y,h_z,\beta_x^\ast\bigr)}\Bigr)
\langle f_x^{o,\ast}, H^{\dagger,\leq
a}\bigl(\nabla^{t'}\bigl(g_{y}^{o}\bigr)\times h_{z}^o\bigr)
\rangle.$$

\end{proposition}
\begin{proof}
Consider now the projection  $e_{f_x^{o,\ast}}$ onto the Hecke
eigenspace corresponding to $f_x^{o,\ast}$. Write
$$\gamma:=e_{f_x^{o,\ast}}H^{\dagger,\leq
a}\bigl(\nabla^{t'}\bigl(g_{y}^{o}\bigr)\times h_{z}^o\bigr)$$and
write $\gamma_{\alpha_x^\ast}=\gamma-\beta_x^\ast V(\gamma)$ and
$\gamma_{\beta_x^\ast}=\gamma-\alpha_x^\ast V(\gamma)$ for the two
$p$-stabilizations. Then $$\langle f_x^{o,\ast}, H^{\dagger,\leq
a}\bigl(\nabla^{t'}\bigl(g_{y}^{o}\bigr)\times
h_{z}^o\bigr)\rangle =\langle f_x^{o,\ast},\gamma\rangle$$and
$$\frac{\mathcal{E}\bigl(g_y,h_z,U\bigr)}{\mathcal{E}_1\bigl(g_y,h_z,U\bigr)}\bigl(\gamma_{\alpha_x^\ast}\bigr)=
\frac{\mathcal{E}\bigl(g_y,h_z,\alpha_x^\ast\bigr)}{\mathcal{E}_1\bigl(g_y,h_z,\alpha_x^\ast\bigr)}
\bigl(\gamma_{\alpha_x^\ast}\bigr)$$and similarly for
$\gamma_{\beta_x^\ast}$. Recalling that
$\gamma=(\alpha_x^\ast-\beta_x^\ast)^{-1}\bigl(\alpha_x^\ast
\gamma_{\alpha_x^\ast}-\beta_x^\ast
\gamma_{\beta_x^\ast}\bigr)=\bigl(1-\beta_x^\ast/\alpha_x^\ast\bigr)^{-1}\gamma_{\alpha_x^\ast}+
\bigl(1-\alpha_x^\ast/\beta_x^\ast\bigr)^{-1} \gamma_{\beta_x^\ast}$ the
conclusion follows from Lemma \ref{prop:4.11} and Lemma
\ref{lemma:PeterssonV} noticing that the Hecke eignspace associated to $f_x^{o,\ast}$ has nebentypus $\chi_x^{-1}$.

\end{proof}

\begin{lemma}\label{lemma:PeterssonV} Let $\delta$ and $\gamma\in
S_k\bigl(\Gamma_1(N)\bigr)$. Assume that $\gamma$ is an eigenform
with eigenvalue $a_p$ for the operator $T_p$ and with nebentypus
$\chi$. Then $$\langle \delta, V(\gamma)\rangle= \frac{\chi(p)^2
a_p}{p^{k-1} (p+1)} \langle \delta,\gamma\rangle.
$$

\end{lemma}
\begin{proof} Let $\alpha:=\left(\begin{matrix} 1 & 0 \cr 0 & p\end{matrix}\right)$.
Following \cite[\S 5.2]{Diamond} we write $$\Gamma_1(N) \alpha
\Gamma_1(N)=\amalg_{j=0}^{p-1} \Gamma_1(N) \beta_j \amalg
\Gamma_1(N) \left(\begin{matrix}m & n \cr N & p\end{matrix}\right)
\beta_\infty,$$where $\beta_j=\left(\begin{matrix} 1 & j \cr 0 &
p\end{matrix}\right)$ for $0\leq j \leq p-1$,
$\beta_\infty=\left(\begin{matrix} p & 0 \cr 0 &
1\end{matrix}\right)$ and $mp-n N=1$. Moreover $T_p
(\gamma)=\sum_{j=0}^{p-1} \gamma\vert_k \beta_j + \gamma\vert_k
\left(\begin{matrix}m & n \cr N & p\end{matrix}\right)
\beta_\infty$. Hence, $a_p \langle \delta,\gamma\rangle =\langle
\delta, T_p (\gamma)\rangle =\sum_{j=0}^{p-1} \langle \delta,
\gamma\vert_k \beta_j\rangle + \langle \delta, \gamma\vert_k
\left(\begin{matrix}m & n \cr N & p\end{matrix}\right)
\beta_\infty \rangle $. Write $\beta_j= a_j \alpha b_j$. Then
$$\langle \delta, \gamma\vert_k \beta_j\rangle= \langle \delta,
\gamma\vert_k a_j \alpha b_j\rangle=\langle \delta\vert_k
b_j^{-1}, \gamma\vert_k a_j \alpha\rangle= \langle \delta,
\gamma\vert_k \alpha\rangle$$as the Petersson product is invariant
for the action of elements of $\Gamma_1(N)$. Similarly, writing
$\left(\begin{matrix}m & n \cr N & p\end{matrix}\right)
\beta_\infty=a \alpha b$ we have $$\langle \delta, \gamma\vert_k
\alpha\rangle=\langle \delta, \gamma\vert_k a^{-1}
\left(\begin{matrix}m & n \cr N & p\end{matrix}\right) b^{-1}
\beta_\infty\rangle = \langle \delta, \gamma\vert_k
\left(\begin{matrix}m & n \cr N & p\end{matrix}\right)
\beta_\infty\rangle.$$We conclude that
$$a_p \langle \delta,\gamma\rangle=(p+1) \overline{\chi}(p) \langle \delta, \gamma\vert_k
\beta_\infty\rangle.
$$

Recall that $T_p(\gamma)=U(\gamma)+\chi(p) p^{k-1} V(\gamma)$ with $V(\gamma)=\overline{\chi}(p) p^{1-k} \gamma\vert_k
\beta_\infty$ and $U(\gamma)=\sum_{j=0}^{p-1} \gamma\vert_k
\beta_j$. Hence $$\langle \delta, V(\gamma)\rangle=
\frac{\chi(p)^2 a_p}{p^{k-1} (p+1)} \langle \delta,\gamma\rangle.
$$

\end{proof}

\begin{proof} (of Theorem  \ref{thm:Interpolate})
It follows from Theorem \ref{thm:fghcirc} that $L^{\rm
alg}\Bigl(f_x,g_y,h_z, \frac{x+y+z-2}{2}\Bigr) =
\frac{I(f^o_x,g^o_x,h^o_y)}{\langle f_x^\ast, f_x^\ast \rangle}$.
On the other hand $$\langle f_x^{o,\ast},e_{f_x^{o,\ast}}
H^{\dagger,\leq a}\bigl(\nabla^{t'}\bigl(g_{y}^{o}\bigr)\times
h_{z}^o\bigr) \rangle= {\langle f_x^{o,\ast}, e_{f_x^{o,\ast}}
\mathcal{H}^{\rm hol}\Bigl(\delta^{t'}\bigl(g_y^o\bigr)\times
h_z^o \Bigr) \rangle}=I(f^o_x,g^o_x,h^o_y),$$where
$\mathcal{H}^{\rm hol}$ the classical holomorphic projection on
nearly holomorphic modular forms and  $\delta$ be the
Shimura--Maass operator; see \cite[\S2]{UNO} or \cite[\S 2.3 \&
2.4]{darmon_rotger}. The claim follows now from Proposition
\ref{prop:fundprop}.
\end{proof}

In particular for $x=k$, $y=\ell$,  $z=m$  we have by construction
$f_x=f$, $g_y=g$, $h_z=h$ and $f_x^{o}=f^{o}$, $g_y^{o}=g^{o}$,
$h_z^{o}=h^{o}$. Then
\begin{corollary} We have
$$\mathcal{L}_p^f(\omega^o_f,\omega^o_g,\omega^o_h)(x,y,z)=  \times{\Bigl(L^{\rm alg}\bigl(f,g,h, \frac{k+\ell+m-2}{2}\bigr)  \Bigr)^{\frac{1}{2}}} $$
for some non-zero constant $\times$ so that $\mathcal{L}_p^f(\omega^o_f,\omega^o_g,\omega^o_h) \neq 0$ if the value of the classical $L$-function is non-zero.
\end{corollary}

\begin{remark}\label{rmk:DRanalogue} For Hida families the Euler factors
appearing in the formula in \ref{thm:Interpolate} differ from
those in \cite{darmon_rotger}. This is due to the fact that the
pairing $\langle f_x^{o,\ast}, H^{\dagger,\leq
a}\bigl(\nabla^{t'}\bigl(g_{y}^{o,[p]}\bigr)\times h_{z}^o\bigr)
\rangle$ computed in Proposition \ref{prop:fundprop} is
substituted in loc.~cit.~by the ordinary stabilizations, namely
one computes $$\langle e_{\rm ord} \bigl(f_x^{o,\ast}\bigr) , e_{\rm
ord}\bigl(\nabla^{t'}\bigl(g_{y}^{o,[p]}\bigr)\times h_{z}^o\bigr)
\rangle ,$$where $e_{\rm ord}=H^{\dagger,0}$ is the ordinary
projection appearing in Hida theory. Nevertheless, under the
Assumptions (\ref{ass10}), one can use the techniques of the
present paper to provide an alternative proof of \cite[Thm.
4.7]{darmon_rotger}.

\end{remark}

\section{Appendix I.}

In this appendix we set-up the general theory of formal vector bundles with marked sections for families of $p$-divisible groups ``which are not far from being ordinay"
in order to facilitate the construction of sheaves of type $\bW^0_k$ on Shimura varieties of type PEL other then modular curves. However we do not construct these
sheaves and we do not construct the triple product $p$-adic $L$-functions in the finite slope case here for other Shimura varieties, only set-up the geometric
machine which should produce the modular sheaves.

\subsection{Vector bundles with marked sections associated to $p$-divisible groups.}\label{sec:pdivconstructions}

We start by fixing a flat $\Z_p$-algebra $A_0$ such that $A_0$ is
$p$-adically complete and separated integral domain. Let $R$ be a normal domain, which is a $p$-adically  complete and separated $A_0$-algebra, without $A_0$-torsion. Let $G$ be a $p$-divisible
group  over $R$ of height $h$ and dimension $d<h$. Let $\det
V_G$ be the determinant ideal of the Vershiebung morphism
$V_G\colon \overline{G}\to \overline{G}^{(p)}$, where
$\overline{G}:=G\times_R (R/pR)$. Its inverse image via the projection $R\to R/pR$ defines an ideal of
$R$ that we denote by $\mathrm{Hdg}(G)$. Let $n$ be a positive
integer and assume that $p \in \mathrm{Hdg}(G)^{p^{ n +1}}$. It
then follows from \cite[Lemma A.1]{halo_spectral}  that $\mathrm{Hdg}(G)$ is an invertible ideal. Furthermore,  $G$ admits a
canonical subgroup $H_n\subset G[p^n]$ of rank $p^{nd}$ thanks to \cite[Cor.~A.2]{halo_spectral}. We assume
that $H_n^\vee(R)=\bigl(\Z/p^n)^d$ and that $G[p](R)/H_1(R)\cong (\Z/p\Z)^{h-d}$. Thanks to  \cite[Prop.~A.3]{halo_spectral} this implies that there exists an invertible ideal
$\mathrm{Hdg}(G)^{\frac{1}{p-1}}\subset R$ whose $(p-1)$-th
power is $\mathrm{Hdg}(G)$.  We also know from \cite[Cor.~A.2]{halo_spectral} that
$\displaystyle {\rm Ker}\bigl(\omega_G\lra
\omega_{H_n}\bigr)\subset p^n
\mathrm{Hdg}(G)^{-\frac{p^n-1}{p-1}}\omega_G$  so that we have a
natural diagram
\begin{equation}\label{eq:dlogs}
\begin{array}{cccccccc}
&&\omega_G\\
&&\downarrow\\
H_n^\vee&\stackrel{{\rm dlog}}{\lra}&\omega_{H_n}\\
&&\downarrow\\
&&\frac{\omega_G}{\bigl(p^n\mathrm{Hdg}(G)^{-\frac{(p^n-1)}{p-1}}\bigr)\omega_G}
\end{array}
\end{equation}

Let $\cI\subset R$ be the invertible ideal
$p^n\mathrm{Hdg}(G)^{-\frac{p^n}{p-1}}$ of $R$.
Let $\Omega_G\subset \omega_G$ be the $R$-submodule generated
by (any) lifts of the images of a $\Z/p^n\Z$-basis of $H_n^\vee(R)$ in
$\omega_G/\bigl(p^n\mathrm{Hdg}(G)^{-\frac{(p^n-1)}{p-1}}\bigr)\omega_G$
via ${d\log}$. It follows from \cite[\S3]{SiegelAIP} and \cite[\S
A]{halo_spectral}  that the sheaf $\Omega_G$ has the following
properties:

\smallskip

a) the cokernel of $\Omega_G\subset \omega_G$ is annihilated by
$\mathrm{Hdg}(G)^{\frac{1}{p-1}}$;
\smallskip

b)  $\Omega_G$ is a free $R$-module of rank $d$ and the map
${\rm dlog}$ defines an isomorphism
$$H_n^\vee(R)\otimes_\Z \bigl(R/\cI\bigr)\cong
\Omega_G\otimes_R
\bigl(R/\cI\bigr).$$

Let $\mathbb{E}(G^\vee)\to G^\vee$ the universal vector extension of the dual $p$-divisible group $G^\vee$ and let
$\mathrm{H}^1_{\rm dR}(G)$ be the sheaf of invariant
differentials of $\mathbb{E}(G^\vee)$. It is a
locally free $R$-module of rank $h$ endowed with an integrable
connection $\nabla\colon \mathrm{H}^1_{\rm dR}(G) \to
\mathrm{H}^1_{\rm dR}(G)\widehat{\otimes}_{R}
\Omega^1_{R/A_0}$, called the Gauss-Manin connection. It also fits
into the exact sequence

$$ 0 \to \omega_{G} \to  \mathrm{H}^1_{\rm dR}(G) \to \omega_{G^\vee}^\vee \to 0.$$This defines the so called Hodge 
filtration on $\mathrm{H}^1_{\rm dR}(G)$. Consider the exact sequence

$$\begin{matrix} 0 \lra & \mathrm{Hdg}(G)^{\frac{p}{p-1}}\cdot \omega_{G} & \lra &  \mathrm{Hdg}(G)^{\frac{p}{p-1}}\cdot \mathrm{H}^1_{\rm dR}(G) &
\lra &  \mathrm{Hdg}(G)^{\frac{p}{p-1}} \cdot\omega_{G^\vee}^\vee &
\lra  0 \cr & \downarrow & & \downarrow & & \downarrow \cr 0 \lra
&  \omega_{G} & \lra & \mathrm{H}^1_{\rm dR}(G) & \lra &
\omega_{G^\vee}^\vee & \lra  0\end{matrix} $$obtained by multiplying
by the invertible ideal $\mathrm{Hdg}(G)^{\frac{p}{p-1}}$.

\begin{definition}\label{def:Hnatural} Using the inclusion $\mathrm{Hdg}(G)^{\frac{p}{p-1}}\cdot \omega_G \subset \Omega_G \subset \omega_G$ define
$\mathrm{H}^\sharp_G$ to be the pushout of
$\mathrm{Hdg}(G)^{\frac{1}{p-1}} \cdot \mathrm{H}_{\rm dR}^1(G)$
via the inclusion $\mathrm{Hdg}(G)^{\frac{1}{p-1}}\omega_G \subset
\Omega_G$. 

The $R$-module $\mathrm{H}^\sharp_G$ has the simple description $\mathrm{H}^\sharp_G:=\mathrm{Hdg}(G)^{\frac{p}{p-1}} \mathrm{H}^1_{\rm dR}(G)+ \Omega_G$
as $R$-submodule of $\mathrm{H}_{\rm dR}^1(G)$.
\end{definition}

\begin{proposition} The $R$-module $\mathrm{H}^\sharp_G$ has the following properties:

\begin{itemize}

\item[i.] we have an exact sequence $0 \to \Omega_G \to
\mathrm{H}^\sharp_G \to  \mathrm{Hdg}(G)^{\frac{p}{p-1}} \cdot
\omega_{G^\vee}^\vee \to  0$. In particular, it is a locally free
$R$-module of rank $d$ and it contains $\Omega_G\subset
\mathrm{H}^\sharp_G$ as a locally direct summand;

\item[ii.] it fits into the following diagram with exact rows:
$$\begin{matrix} 0 \lra & \Omega_G & \lra &  \mathrm{H}^\sharp_G &
\lra &  \mathrm{Hdg}(G)^{\frac{p}{p-1}}\cdot \omega_{G^\vee}^\vee &
\lra  0 \cr & \downarrow & & \downarrow & & \downarrow \cr 0 \lra
&  \omega_{G} & \lra & \mathrm{H}^1_{\rm dR}(G) & \lra &
\omega_{G^\vee}^\vee & \lra  0.\end{matrix} $$

\item[iii.] the choice of a $\Z/p^n\Z$-basis of $H_n^\vee(S)$ defines a
basis $s_1,\ldots,s_d$ of the $R/\cI$-module of $\Omega_G/\cI
\Omega_G$ via the map ${\rm dlog}$.
\end{itemize}

In particular, we are in the hypotheses of \S\ref{sec:vbfil} with
$\cE:=\mathrm{H}^\sharp_G$, $\cF=\Omega_G$ and the sections
$s_1,\ldots,s_d$ of $\Omega_G/\cI \Omega_G$
where $\cI=p^n \mathrm{Hdg}(G)^{-\frac{p^n}{p-1}}$.
\end{proposition}

\begin{proposition}\label{prop:nablasharp}
Assume that $G^\vee[p^n](R)\cong (\Z/p^n\Z)^h$. Then the Gauss-Manin connection $\nabla$ on
$\mathrm{H}^1_{\rm dR}(G)$ defines a connection
$$\nabla_{G,\sharp}\colon \mathrm{H}^\sharp_G\lra \mathrm{H}^\sharp_G\widehat{\otimes}_{R} \Omega^1_{R/A_0} $$
such that $\nabla_{G,\sharp}\vert_{\Omega_G}\equiv 0$ modulo $\cI$. In
particular, the hypotheses of \S\ref{sec:fvvconenction}, namely that $s_1,\ldots,s_d$ are horizontal for $\nabla_{G,\sharp}$
modulo $\cI$, hold true.
\end{proposition}
\begin{proof} The isomorphism $\rho_S\colon  G^\vee[p^n](R)\cong (\Z/p^n\Z)^h$ induces
a morphism of finite and flat group schemes $\rho\colon
(\Z/p^n\Z)^h \to G^\vee[p^n]$ over $R$. Let $R_n:=R/p^n R$. Since ${\rm dlog}$ is functorial and
$\omega_{G[p^n],R_n}=\omega_G/p^n\omega_G$ as $G$ is a
$p$-divisible group, we have a commutative diagram:
$$\begin{matrix} (\Z/p^n\Z)^h & \stackrel{{\rm dlog}_{\mu_{p^n}^h}}{\lra} & \omega_{\mu_{p^n,R_n}}^h \cr
\downarrow \rho(S) & & \downarrow  d\rho^\vee \cr  G^\vee[p^n](R)
&\stackrel{{\rm dlog}_{G^\vee[p^n]}}{\lra}    &
\omega_{G[p^n],R_n}=\omega_G/p^n\omega_G.\cr
\end{matrix}$$

The connection on $\mathrm{H}^1_{\rm dR}(G)$ modulo $p^n$ is the
connection  $\overline{\nabla}_{G[p^n]}$ on the invariant
differentials  of the universal extension of $G[p^n]^\vee=G^\vee[p^n]$ relative to $R_n$,
that we denote by $\mathrm{H}^1_{\rm dR}\bigl(G[p^n]/R_n\bigr)$.
Since $\mu_{p^n}^h$ is isotrivial over $R_n$, it follows that the
Gauss-Manin connection $\overline{\nabla}_{\mu_{p^n}^h}$ on
$\mathrm{H}^1_{\rm dR}(\mu_{p^n}^h/R_n\bigr)$ is trivial so that
$\overline{\nabla}_{\mu_{p^n}^h} \circ {\rm
dlog}_{\mu_{p^n}^h}=0$. By the functoriality of the Gauss-Manin
connection and the commutativity of the diagram above it follows
that $\overline{\nabla}_{G[p^n]} \circ d\log_{G^\vee[p^n]}=0$. Due to
(\ref{eq:dlogs}) the map ${\rm dlog}_{G^\vee[p^n]}$ composed with the
projection to $\omega_G/\mathrm{Hdg}(G)^{\frac{1}{p-1}} \cI
\omega_G$ factors via $G^\vee[p^n](R)\to H_n^\vee(R)$ and ${\rm
dlog}_{H_n^\vee}$. In particular, we can choose lifts $\tilde{s}_1,\ldots,\tilde{s}_d\in \Omega_G$ of $s_1,\ldots,s_d$ in the image
of ${\rm dlog}_{G^\vee[p^n]}$ modulo $p^n$ and we deduce that
$\nabla(\tilde{s}_i)\equiv 0$ modulo $p^n \mathrm{H}^1_{\rm dR}(G)$ for
$i=1,\ldots,d$. Thus the restriction of $\nabla$ to $\Omega_G$
factors through $p^n \mathrm{H}^1_{\rm
dR}(G) \subset \mathrm{Hdg}(G)^{\frac{p}{p-1}} \mathrm{H}^1_{\rm
dR}(G)\subset \mathrm{H}^\sharp_G$  (recall that $p\in \mathrm{Hdg}(G)^{p^{n+1}}$) and the images
$\nabla(s_1),\ldots,\nabla(s_d)$ are $ 0$ modulo $\cI
\mathrm{H}^\sharp_G \otimes \Omega^1_{S/A_0}$ (recall $\cI=p^n \mathrm{Hdg}(G)^{-\frac{p^n}{p-1}}$). This defines
$\nabla_{G,\sharp}$ on $\Omega_G$.

As $\mathrm{H}^\sharp_G=\Omega_G + \mathrm{Hdg}(G)^{\frac{p}{p-1}}
\mathrm{H}^1_{\rm dR}(G) $, as $R$-submodules of
$\mathrm{H}^1_{\rm dR}(G)$, to conclude we are left to show that
$\nabla$ sends $\mathrm{Hdg}(G)^{\frac{p}{p-1}} \mathrm{H}^1_{\rm
dR}(G)$ into $\mathrm{Hdg}(G)^{\frac{p}{p-1}} \mathrm{H}^1_{\rm
dR}(G)$. Using Leibniz's rule this follows as
$\mathrm{Hdg}(G)^{p/(p-1)}$ is a $p$-th power  so that
$d \mathrm{Hdg}(G)^{p/(p-1)}\equiv 0$ modulo $p R$ and $p
\in \mathrm{Hdg}(G)^{p^{ n +1}}$ by assumption.

\end{proof}

\subsection{Functoriality in the elliptic case}\label{sec:functoriality}

We keep the assumptions of the previous section on the rings $A_0$ and $R$.  Let $G$ and $G'$ be $p$-divisible groups over $R$ associated to elliptic curves over $R$. We assume
that $p \in \mathrm{Hdg}(G')^{p^{ n +1}}$ and that
$\mathrm{Hdg}(G')\subset \mathrm{Hdg}(G)$. Then  both $G$ and $G'$ admit canonical subgroups $H_n\subset G[p^n]$ and $H_n'\subset G'[p^n]$, of rank $p^{n}$. We assume that 
 $H_n^\vee(R)\cong\bigl(\Z/p^n )\cong H_n^{',\vee}(R)$. We
set $\cI=p^n \mathrm{Hdg}(G')^{-\frac{p^n}{p-1}}$ (which contains
$p^n \mathrm{Hdg}(G)^{-\frac{p^n}{p-1}}$). 

\

We let $\lambda\colon G'\to G$ be an isogeny such that $H_n'$ maps to $H_n$ and the induced map $H_n'\to H_n$ 
is an isomorphism after inverting $p$. 
Then the dual isogeny $\lambda^\vee\colon G^\vee \to G^{',\vee}$ defines a map of universal vector extensions $\lambda^\vee\colon \mathbb{E}(G^\vee)\to \mathbb{E}(G^{,\vee'})$ and, taking the induced map on Lie algebras $\lambda^\vee_\ast$, a commutative diagram:
$$
\begin{array}{lllllllll}
0&\rightarrow&\omega_G &\rightarrow&\mathrm{H}^1_{\rm dR}(G) &\rightarrow& \omega_{G^\vee}^\vee   &\rightarrow& 0\\
&&\downarrow \lambda^\ast&&\downarrow \lambda^\vee_\ast & & \downarrow \bigl((\lambda^{\vee})^\ast\bigr)^\vee\\
0&\rightarrow&\omega_{G'}&\rightarrow&\mathrm{H}^1_{\rm dR}(G^{'}) &\rightarrow& \omega_{G^{',\vee}}^\vee   &\rightarrow& 0.
\end{array}
$$
Here $\lambda^\ast\colon \omega_G\to \omega_{G'}$, resp. $(\lambda^\vee)^{\ast}\colon \omega_{G^{',\vee}} \to \omega_{G^\vee}$ is the pull-back on invariant differentials defined by $\lambda$, resp. $\lambda^\vee$ and $\bigl((\lambda^\vee)^\ast\bigr)^\vee$ is the $R$-dual of $(\lambda^\vee)^{\ast}$. 
As $\lambda$ induces a map $H_n'\to H_n$, which is an isomorphism after inverting $p$, then $\lambda$ induces a map $H_n^\vee \to H_n^{',\vee}$ which is an isomorphism after inverting $p$ and we get an  isomorphism $H_n^\vee(R) \cong H_n^{',\vee}(R)$. Then,  the functoriality of 
diagram (\ref{eq:dlogs}) provides  the commutative diagram
$$
\begin{array}{lllllllll}
0&\rightarrow&\Omega_{G}&\subset&\omega_{G}\\
&&\downarrow\cong&&\downarrow \lambda^\ast\\
0&\rightarrow&\Omega_{G'}&\subset&\omega_{G'}
\end{array}
$$

We denote by $\lambda^\ast\colon \Omega_{G} \to \Omega_{G'}$ the induced
isomorphism. The choice of a $\Z/p^n\Z$-basis of $H_n^\vee(R)$ defines
a basis $s$ of the $R/\cI$-module 
$\Omega_G/\cI \Omega_G$ and, via the isomorphism $H_n^\vee(R)\lra
H_n^{',\vee}(R)$ induced by $\lambda^\vee$, also a basis $s'$
of the $\cO_S/\cI$-module of $\Omega_{G'}/\cI \Omega_{G'}$.

\begin{lemma}
\label{lemma:fsharp} Assume that $\lambda$ has degree $p^n$. Then the  map 
$\lambda^\vee_\ast$ induces a
morphism $\lambda^\sharp\colon \mathrm{H}^\sharp_G\rightarrow
\mathrm{H}^\sharp_{G'}$. Moreover $\lambda^\sharp$ fits in following commutative diagram
$$\begin{matrix} 0 \lra & \Omega_G & \lra &  \mathrm{H}^\sharp_G & \lra &
\mathrm{Hdg}(G)^{\frac{p}{p-1}}\cdot \omega_{G^\vee}^\vee & \lra  0
\cr & \downarrow \lambda^\ast & & \downarrow \lambda^\sharp & & \downarrow  \bigl((\lambda^\vee)^\ast\bigr)^\vee \cr 0
\lra & \Omega_{G'} & \lra & \mathrm{H}^\sharp_{G'} & \lra &
\mathrm{Hdg}(G')^{\frac{p}{p-1}}\cdot \omega_{G^{',\vee}}^\vee & \lra
0,\end{matrix} $$with the properties that $f^\ast(s)=s'$ (modulo $\cI$) and the image of $\mathrm{Hdg}(G)^{\frac{p}{p-1}}\cdot \omega_{G^\vee}^\vee$ via $\bigl((\lambda^\vee)^\ast\bigr)^\vee$
is equal to $\tau_{\lambda} \cdot \mathrm{Hdg}(G')^{\frac{p}{p-1}}\cdot \omega_{G^{',\vee}}^\vee$ with $\tau_{\lambda}=1$ if $n=0$ and $\tau_\lambda=p^n / \mathrm{Hdg}(G')^{\frac{(p+1)(p^n-1)}{p^{n}(p-1)}}$.
\end{lemma}

\begin{proof} If $\lambda$ is an isomorphism there is nothing to prove. For general $n$, we remark the $\mathrm{H}^\sharp_G$ and $
\mathrm{H}^\sharp_{G'}$ are locally free $R$-modules of rank $2$ and $R$ is normal; hence it suffices to prove that $\lambda^\vee_\ast\bigl(\mathrm{H}^\sharp_G\bigr)\subset \mathrm{H}^\sharp_{G'}$ holds after localization at codimension $1$ prime idelas of $R$. This is clear for prime ideals not containing $p$. Thus, after replacing $R$ with the localization at a prime ideal cotaining $p$, we may assume that $R$ is a dvr. In this case, we may  write $\lambda$ as the composite of $n$ isogenies of degree $p$ and we reduce to the case that $n=1$, i.e., that $\lambda$ has degree $p$.  Then $\lambda$ is the quotient under a subgroup scheme $N$ such that $N\cap H_1'=\{0\}$. From now on we view the dual isogeny $\lambda^\vee$ as a morphism $\lambda^\vee \colon G\to G^{'}$, identifying $G\cong G^\vee$ and $G'\cong G^{',\vee}$ via the principal polarizations on  $G$ and $G'$ and $\omega_G\cong \omega_{G^\vee}$ and $\omega_{G'}\cong \omega_{G^{',\vee}}$. Then $\lambda^\vee$ is the quotient by the canonical subgroup $H_1$. This forces $\mathrm{Hdg}(G')=\mathrm{Hdg}(G)^p$ and $\lambda^\vee$ coincides with Frobenius modulo $p/{\rm Hdg}(G)$ (see \cite[Cor. A.2]{halo_spectral}). 

As $\mathrm{Hdg}(G')=\mathrm{Hdg}(G)^p$, the image $\lambda^{\vee}_\ast\bigl(\mathrm{H}^\sharp_G \bigr)$ is contained in $\widetilde{\mathrm{H}}^\sharp_{G'}:= \mathrm{Hdg}(G')^{\frac{1}{p-1}} \mathrm{H}^1_{\rm dR}(G^{'}) + \Omega_{G'}$. We clearly have ${\mathrm{H}}^\sharp_{G'} \subset \widetilde{\mathrm{H}}^\sharp_{G'} $ and since $\mathrm{Hdg}(G')^{\frac{1}{p-1}} \omega_{G'}\subset \Omega_{G'}$, we have an exact sequence 

$$0 \to \Omega_{G'} \to
\widetilde{\mathrm{H}}^\sharp_{G'} \to  \mathrm{Hdg}(G')^{\frac{1}{p-1}} \cdot
\omega_{G'}^\vee \to  0.$$

In particular ${\mathrm{H}}^\sharp_{G'}$ is identified with the pull-back of $\widetilde{\mathrm{H}}^\sharp_{G'}$ via the inclusion $\mathrm{Hdg}(G')^{\frac{p}{p-1}} \cdot
\omega_{G'}^\vee \subset \mathrm{Hdg}(G')^{\frac{1}{p-1}} \cdot
\omega_{G'}^\vee$. Then $\lambda^\vee_\ast\bigl(\mathrm{H}^\sharp_G \bigr)$ is contained in ${\mathrm{H}}^\sharp_{G'}$ if and only if the image of $\mathrm{Hdg}(G)^{\frac{p}{p-1}} \cdot
\omega_{G}^\vee $ via $\bigl((\lambda^\vee)^\ast)^\vee$ is contained in $\mathrm{Hdg}(G')^{\frac{p}{p-1}} \cdot
\omega_{G'}^\vee$. This amounts to prove that the image of  $
\omega_{G}^\vee$ via $\bigl((\lambda^\vee)^\ast)^\vee$ is contained in $\mathrm{Hdg}(G')^{\frac{p}{p-1}} \mathrm{Hdg}(G)^{\frac{p}{p-1}} \cdot
\omega_{G'}^\vee=\mathrm{Hdg}(G')\omega_{G'}^\vee$. We remarked above that $\lambda^\vee$ is Frobenius modulo $p/{\rm Hdg}(G)$ so that the map $\lambda$ is Vershiebung
modulo $p/{\rm Hdg}(G)$ and hence $\lambda^\ast(\omega_G)={\rm Hdg}(G) \cdot \omega_{G'}$ modulo $p/{\rm Hdg}(G) \omega_{G'}$. Since $p\in {\rm Hdg}(G)^{p^{n+1}}$ this implies that 
$\lambda^\ast(\omega_G)={\rm Hdg}(G) \omega_{G'}$. Using that $\lambda^\vee \circ \lambda$ is multiplication by $p$, we deduce that the map $(\lambda^\vee)^\ast \circ \lambda^\ast$ 
on differentials is mutiplication by $p$ so that $(\lambda^\vee)^\ast(\omega_{G'})=p/{\rm Hdg}(G)  \cdot \omega_G$.
Taking $R$-duals we conclude that the image of  $
\omega_{G}^\vee$ via $\bigl((\lambda^\vee)^\ast)^\vee$ is  $p/{\rm Hdg}(G) \cdot
\omega_{G'}^\vee$ and the image of  $\mathrm{Hdg}(G)^{\frac{p}{p-1}}\cdot
\omega_{G}^\vee$ via $\bigl((\lambda^\vee)^\ast)^\vee$ is  $\tau_{\lambda} \cdot \mathrm{Hdg}(G')^{\frac{p}{p-1}}\cdot
\omega_{G'}^\vee$ with $\tau_{\lambda}=p{\rm Hdg}(G)^{-1} \mathrm{Hdg}(G)^{\frac{p}{p-1}}\mathrm{Hdg}(G')^{\frac{-p}{p-1}}$.

Since $\mathrm{Hdg}(G)=\mathrm{Hdg}(G')^{\frac{1}{p}}$ then $\tau_{\lambda}=p / \mathrm{Hdg}(G')^{\frac{p+1}{p}}=p/\mathrm{Hdg}(G)^{{p+1}}$and, as $p \in \mathrm{Hdg}(G)^{p^{ n +1}} \subset {\rm Hdg}(G)^{p+1}$, we deduce that 
$\tau_{\lambda}\in R$ so that the first and last claims follow.

The statement concerning  $f^\ast(s)$ follows from the fact that $\lambda^\ast$  is compatible
with ${\rm dlog}\colon H_n^\vee(R) \to \Omega_G/\cI \Omega_G$ and
${\rm dlog}\colon H_n^{',\vee}(R) \to \Omega_{G'}/\cI \Omega_{G'}$
and the isomorphism $H_n^\vee(R) \to H_n^{',\vee}(R)$ provided by $\lambda^\vee$.

\end{proof}

Let $f_0'\colon \bV_0\bigl(\mathrm{H}^\sharp_{G'}\bigr) \to S$ and
$f_0\colon \bV_0\bigl(\mathrm{H}^\sharp_{G}\bigr) \to S$ be the
formal schemes of Definition \ref{def:V0}. It follows from the
functoriality of this definition that $f^\sharp$ defines a
commutative diagram of formal schemes over $S$:

$$\begin{matrix}\bV_0\bigl(\mathrm{H}^\sharp_{G'}\bigr) & \stackrel{\lambda^\sharp}{\lra} & \bV\bigl(\mathrm{H}^\sharp_{G}\bigr) \cr
\downarrow & & \downarrow \cr \bV_0\bigl(\Omega_{G'}\bigr) &
\stackrel{g}{\lra} & \bV\bigl(\Omega_{G}\bigr)\cr
\end{matrix}.$$In conclusion, we deduce from Corollary
\ref{cor:functfil}:

\begin{proposition}\label{prop:functfilbV0}  Assume that $\lambda$ is an isomorhism or that it has degree $p$. Then 
the morphism $\lambda^\sharp$ induces a morphism $f_{0,\ast} \bigl( \cO_{\bV_0(\mathrm{H}^\sharp_{G})} \bigr)\lra f_{0,\ast}'
\bigl(\cO_{\bV_0(\mathrm{H}^\sharp_{G'})}\bigr)$ preserving the filtrations
$\Fil_\bullet f_{0,\ast} \bigl( \cO_{\bV_0(\mathrm{H}^\sharp_{G})}\bigr)$ and
$\Fil_\bullet f_{0,\ast}' \bigl(\cO_{\bV_0(\mathrm{H}^\sharp_{G'})}\bigr)$.
Via the identifications of the graded pieces in Corollary
\ref{cor:fil} the induced map

$$f_{0,\ast} \bigl( \cO_{\bV_0(\Omega_{G})} \otimes_{\cO_S} \mathrm{Sym}^h\bigl(\mathrm{Hdg}(G)^{\frac{p}{p-1}}\omega_{G^\vee}^\vee\bigr) \lra
f_{0,\ast}'  \bigl(\cO_{\bV_0(\Omega_{G'})}\bigr) \otimes_{\cO_S}
\mathrm{Sym}^h\bigl(\mathrm{Hdg}(G')^{\frac{p}{p-1}}\omega_{G^{',\vee}}^\vee\bigr)$$

is the tensor product of the isomorphism $f_{0,\ast}\bigl(
\cO_{\bV_0(\Omega_{G})}\bigr) \to f_{0,\ast}' \bigl( \cO_{\bV_0(\Omega_{G'})}\bigr)$
provided by $\lambda^\ast$ and the map on $\mathrm{Sym}^h$ provided by the
dual of the map $(\lambda^{\vee})\ast\colon \omega_{G^{',\vee}} \to
\omega_{G^\vee}$.

Furthermore, assume that $G[p^n](R)\cong (\Z/p^n\Z)^2$ and
$G^{'}[p^n](R) \cong (\Z/p^n\Z)^2$. Then $\lambda^\sharp\colon \mathrm{H}^\sharp_G\rightarrow
\mathrm{H}^\sharp_{G'}$ is compatible with the connections
$\nabla_{G,\sharp}$ and $\nabla_{G',\sharp}$ defined in
Proposition \ref{prop:nablasharp}.

\end{proposition}

\section{Appendix II: Application to the  three variable Rankin-Selberg $p$-adic L-functions. A corrigendum to \cite{UNO}, by Eric Urban.}

\bigskip

\subsection{Introduction}
In \cite{UNO}, the author introduced nearly overconvergent modular forms of finite order and their spectral theory. The theory has be refined in \cite{AI} including intgral structure that allows to define families
of nearly overconvergent modular forms of unbounded degree that was  missing in \cite{UNO}.
The  purpose of this appendix is to fill a gap in \cite{UNO} about the construction of the three variable Rankin-Selberg  $p$-adic L-functions which we can now solve thanks to the work of F. Andreatta and A. Iovita \cite{AI}.
The gap lies in the construction made in section \S 4.4.1 a few lines before Proposition 11 where the existence of a finite slope projector denoted $e_{R,\Vf}$ is claimed. Here $\Vf$ is an affinoid of weight space and $R$ is a polynomail in $A(\Vf)[X]$ dividing the Fredholm determinant of $U$ acting on the space of $\Vf$-families of nearly overconvergent modular forms. It was falsely claimed on top of page 434
 that $e_{R,\Vf}$ can be defined as  $S(U)$ for some $S\in X.A(\Vf)[[X]]$ when it would  actually be a limit of polynomial
in the Hecke operator $U$ with coefficient in the fractions ring of $A(\Vf)$ that may have unbounded denominators making the convergence a difficult question. In the following pages, we will explain
how the existence of this projector in the theory of \cite{AI}  can actually be used to define the missing ingredient of the construction in \cite[\S 4.4.1]{UNO}. For the sake of brevity, we will use freely the notations of \cite{UNO} and \cite{AI} without recalling all of them.

I would like to thank Zheng Liu for pointing out the gap to me when she was working in her thesis on a generalization of my work to the Siegel modular case.  I would like also to thank F. Andreatta and A. Iovita for telling me about their work and for including this corrigendum as an appendix of their paper.

\subsection{Families of nearly overconvergent modular forms}
Let $p$ be an odd prime. The purpose of this paragraph is to collect some results of \cite{UNO}and \cite{AI} and harmonize the notations. Recall that for any rigid analytic variety  $X$ over a  non archimedean field, we denote respectively by $A(X)$ and $A^0(X)$  the ring of rigid analytic function on $X$ and its subring of functions bounded by $1$ on $X$. Recall also that we denote weight space by $\Xf$. It is the rigid analytic space over $\Q_p$ such that $\Xf(\Q_p)=Hom_{cont}(\Z_p^\times,\Q_p^\times)$. For any integer $k$, we denote by $[k]\in\Xf(\Q_p)$ the weight given by $x\mapsto x^k, \forall x\in\Z_p^\times$. For any $p$-power root of unity $\zeta$, we denote $\chi_\zeta$ the finite order character of  $\Z_p^\times$ trivial on $\mu_{p-1}$ and such that $\chi_\zeta(1+p)=\zeta$.

Let $\Uf\subset\Xf$ be an affinoid subdomain of weight space and
choose $I=[0,p^c]$ such that $ A^0(\Xf)=\Lambda\subset
\Lambda_I\subset A^0(\Uf)$ with $\Lambda$ and $\Lambda_I$ as
defined in \cite[\S 3.1]{AI}. We also fix integers $r$ and $n$
compatible with $I$ as in loc.~cit. We consider the Frechet space
over the Banach algebra $A(\Uf)$
$$\CN_\Uf^\dag:=\limi{r}(H^0(\Xf_{r,I},\mathbb W_{k_I})\otimes _{\Lambda_I}A(\Uf))$$
Here $\Xf_{r,I}$ is the formal scheme defined in \cite[\S 3.1]{AI} attached to a  strict neighborhood (in rigid geometry) of the ordinary locus of the modular curve.

It is easily seen that the filtration on $\mathbb W_{k_I}$ of \cite[Thm 3.11]{AI} induces the filtration
$$ \CM_\Uf^\dag=\CN^{0,\dag}_\Uf\subset \CN^{1,\dag}_\Uf\subset\dots\subset \CN^{s,\dag}_\Uf\subset \dots\subset \CN_\Uf^\dag$$
where for each integer $s$, $\CN^{s,\dag}_\Uf$ denotes the space of $\Uf$-families of nearly overconvergent modular forms as defined in \cite[\S 3.3]{UNO}. The work
done in \cite[\S 3.1]{AI} that we use here is the rigorous construction using the correct integral structure of what was alluded to in \cite[Remark 10]{UNO}. Moreover, it follows from \cite[\S 3.6]{AI} that there is a completely continuous action of the
$U$ operator on $\CN_\Uf^\dag$ that respect the above filtration and that is compatible with the one defined in \cite{UNO}. Moreover, we easily see for example using \cite[Prop. 7 (ii)]{UNO} that
the Fredholm determinant $P^\infty_\Uf(\kappa,X)$ of $U$ acting on $\CN_\Uf^\dag$ satisfied the relation
$$P^\infty_\Uf(\kappa,X)=\prod_{i=0}^\infty P_{\Uf[-2i]}(\kappa.[-2i],p^iX)$$
where $ P_{\Uf[-2i]}$ stands for the Fredholm determinant of $U$ acting on the space of families of overconvergent modular forms of weights varying in the translated affinoid $\Uf$ by the weight $[-2i]$.

Recall finally that an admissible pair for nearly overconvergent
forms is a data  $(R,\Vf)$ where $R\in A(\Uf)[X]$ is a monic
polynomial such that there is a factorization
$P^\infty_\Uf(\kappa,X)=R^*(X) Q(X)$ where $R^*(X)=R(1/X)X^{deg\;
R}$ and $Q(X)$ are relatively prime in $A(\Uf)\{\{X\}\}$. To such
a pair, one can associate a decomposition
$$\CN_\Uf^\dag=\CN_{R,\Uf}\oplus \CS_{R,\Uf}$$
which is stable under the action of $U$ and such that
$det(1-X.U|\CN_{R,\Uf})=R(X)$. We will call $e_{R,\Uf}$ the
projection of $\CN_\Uf^\dag$ onto $\CN_{R,\Uf}$. This later
subspace consists in families of nearly overconvergent modular
forms of bounded order. This is well-known and follows from the
generalization by Coleman and others of the spectral theory of
completely continuous operators originally due to J.P. Serre.

\subsection{The nearly overconvergent Eisenstein family}
Recall that we have defined in \cite[\S4.3]{UNO} the nearly overconvergent Eisenstein family $q$-expansion $\Theta.E\in A^0(\Xf\times\Xf)[[q]]$ by
$$\Theta.E(\kappa,\kappa'):=\sum_{n=1\atop (n,p)=1}^\infty \langle n\rangle_\kappa a(n,E,\kappa')q^n$$
It satisfied the following interpolation property \cite[Lemma 6]{UNO}.  If $\kappa=[r]$ and $\kappa'=[k]\psi$ with $\psi$  a finite order character and $k$ and $r$ positive integers, the evaluation at $(\kappa,\kappa')$ of $\Theta.E$ is  $\Theta.E(\kappa,\kappa')=\Theta^r.E^{(p)}_k(\psi)(q)$
and is the $p$-adic $q$-expansion of the nearly holomorphic Eisenstein series $\delta_k^rE^{(p)}_k(\psi)$.

A generalization of this statement is the crucial lemma below
which will follow from \cite[Thm 4.6]{AI}. Because of the
hypothesis 4.1 of loc.~cit., we need to introduce the following
notation. We denote by  $\Xf'\subset\Xf$ the affinoid subdomain of
$\Xf$ of the weights $\kappa$ such that
$|\kappa(1+p)-\zeta(1+p)^n|_p\leq 1/p^2$ for some integer $n$ and
some $p$-power root of unity $\zeta$. Notice that $\Xf'(\Q_p)$
contains all the classical weights.
%

\begin{lemma} \label{lemma_eis}There exists
$\Theta E_{\Xf',\Xf'}\in A(\Xf')  \hat\otimes \CN_{\Xf'}^{\dag}$ such that its $q$-expansion is given by the canonical image of $\Theta.E$ into $A(\Xf')\hat\otimes_{\Q_p} A(\Xf')[[q]]$ induced by
the canonical map $\Lambda\otimes_{\Z_p}\Lambda\rightarrow A(\Xf')\hat\otimes_{\Q_p} A(\Xf')$.

\end{lemma}
\proof
For a given integer $n$ and $p$-power root of unity $\zeta$, we denote by $\Xf'_{n,\zeta}\subset\Xf'$ the affinoid subdomain of the weights $\kappa$ such that $|\kappa(1+p)-\zeta(1+p)^n|_p\leq 1/p^2$.  When $\zeta=1$, we just write $\Xf'_{n}$ for $\Xf'_{n,1}$. Since $\Xf'$ is the disjoint union
$$\Xf'=\bigsqcup_{n=0}^{p-1}\bigsqcup_{\zeta}\Xf'_{n,\zeta}$$
it is sufficient to construct $E_{\Xf'_{n,\zeta},\Xf'_{m,\eta}}\in A(\Xf'_{n,\zeta})  \hat\otimes \CN_{\Xf'_{m,\eta}}^{\dag}$ satisfying  the corresponding condition on the $q$-expansion.
Notice also that $\Xf'_{n,\zeta}=[n]\chi_\zeta.\Xf'_0$ and that,
with the notations of \cite{AI},
we have $A^0(\Xf'_0)=\Lambda_{I'}$ with $I'=[0,p^2]$.

It clearly exists $E_{\Xf'_{m,\eta}}^{(p)}\in \CM^\dag_{\Xf'_{m,\eta}}\subset\CN^\dag_{\Xf'_{m,\eta}}$ such that its $q$-expansion in $A(\Xf'_{m,\eta})[[q]]$ is given by $\Theta E([0],\kappa )$.
Indeed it is defined by $E_{\Xf'_{m,\eta}}^{(p)}=E^{ord}_{\Xf'_{m,\eta}}-E^{ord}_{\Xf'_{m,\eta}}|V_p$ where $E^{ord}_{\Xf'_{m,\eta}}\in e_{ord}.\CM^\dag_{\Xf'_{m,\eta}}$ denotes the $\Xf'_{m,\eta}$-family of ordinary Eisenstein series and $V_p$ denotes the Frobenius operator inducing raising $q$ to its $p$-power on the $q$-expansion.

We have the isomorphism $\Lambda\cong \Z_p[(\Z/p\Z)^\times][[T]]$ done by choosing the  topological generator $1+p\in 1+p\Z_p$,. Let $\kappa_{\Xf'_0}$ be the universal weight $\Z_p^\times\rightarrow A(\Xf'_0)^\times$. We can easily see that
$Log(\kappa_{\Xf'_0})=\frac{log(1+T)}{log(1+p)}=u_\kappa$ where $u_\kappa$ is the notation defined in \cite{AI} while $Log(\kappa_{\Xf'_0})$ is  the notation defined in \cite{UNO}. The  assumption 4.1 of \cite{AI}, now reads easily as $I\subset [0, p^2]$ and is therefore satisfied since $A^0(\Xf'_0)=\Lambda_{[0,p^2]}$.

Before pursuing, we note that we will use the notation
$\nabla^\chi$ following the definition 4.11 of \cite{AI} for the
twist of nearly overconvergent forms by a finite order character
$\chi$ of $\Z_p^\times$. We refer the reader to loc.~cit.~for its
properties.

Let $\kappa_s$ the generic weight $\Z_p^\times \rightarrow A(\Xf'_{n,\zeta})$. Since $\Xf'_{n,\zeta}=[n]\chi_\zeta.\Xf'_0$, the weight $\kappa_s.[-n]\chi_\zeta^{-1}$ satisfies the assumption 4.1 of \cite{AI}. Let $m'$ be a natural integer such that $m+2m'$ is divisible by $p$ and let $\eta'$ be a $p$-power root of unity so that $\eta'^2=\eta^{-1}$. Then $([m'].\chi_{\eta'})^2\Xf'_{m,\eta}=\Xf'_0$ and therefore the weight of $\nabla^{\chi_{\eta'}}\nabla^{m'}E^{(p)}_{\Xf'_{m,\eta}}$ satisfies also the assumption 4.1. of \cite{AI}. According to \cite[Thm 4.6]{AI}, one can therefore define  $\nabla^{s-n\chi_\zeta} (\nabla^{\chi_{\eta'}}\nabla^{m'}E^{(p)}_{\Xf'_{m,\eta}})$ where $\nabla^{s-n\chi_\zeta}$ stands for $\nabla^{s'}$ with $s'$ the weight corresponding to $\kappa_{s'}=\kappa_s[-n]\chi_\zeta^{-1}$ taking values in $A^0(\Xf'_0)$.
Since $\Xf'_{n,\zeta}$ depends only on $n$ modulo $p$, we may and do assume that $n>m'$, and we can therefore  set
$$\Theta E_{\Xf'_{n,\zeta},\Xf'_{m,\eta}}:=\nabla^{\chi_{\zeta\eta}}\nabla^{n-m'}(\nabla^{s-n\chi_\zeta} (\nabla^{\chi_{\eta'}}\nabla^{m'}E^{(p)}_{\Xf'_{m,\eta}}))$$
From the effect of $\nabla$ on the $q$-expansion, it is now easy to verify that $\Theta E_{\Xf'_{n,\zeta},\Xf'_{m,\eta}}$ satisfies the condition on the $q$-expansion claimed in the Lemma.  \qed


\subsection{Final construction of $G^E_{Q,\Uf,R,\Vf}$} In this paragraph,  we explain how to replace the bottom of page 433 of \cite{UNO}. We now assume that $\Uf$ and $\Vf$ are affinoid subdomains of $\Xf'$.  Let $(Q,\Uf)$ be an admissible pair for overconvergent forms of tame level $1$ and let $T_{Q,\Uf}$
be the corresponding Hecke algebra over $A(\Uf)$. By definition it is the ring of analytic function
on the affinoid subdomain $\CE_{Q,\Uf}$ sitting over the affinoid subdomain $Z_{Q,\Uf}$ associated to $ (Q,\Uf)$ of the spectral curve of the $U$-operator. Recall that
$$Z_{Q,\Uf}=Max(A(\Uf)[X]/Q^*(X))\subset Z_{U}\subset\A^1_{\mathbf rig}\times\Uf$$
where $Z_{U}$ is the spectral curve attache to $U$  and
$$T_{Q,\Uf}=A(\CE_{Q,\Uf})\hbox{  with  } \CE_{Q,\Uf}=\CE\otimes_{Z_{U}}Z_{Q,\Uf}$$
where $\CE$ stands for the Eigencurve. The universal family of overconvergent  modular eigenforms of type $(Q,\Uf)$ is an element of $\CM_{Q,\Uf}\otimes_{A(\Uf)}T_{Q,\Uf}$ whose $q$-expansion is given by
$$G_{Q,\Uf}(q):=\sum_{n=1}^\infty T(n)q^n\in T_{Q,\Uf}[[q]]$$
Tautologically, for any point $y\in\CE_{Q,\Uf}$ of weight $\kappa_y\in\Uf$, the evaluation $G_{Q,\Uf}(y)$ at $y$ of $G_{Q,\Uf}$ is the overconvergent normalized eigenform $g_y$ of weight $\kappa_y$ associated to $y$.

We set
$$G^{E}_{Q,\Uf}:=G_{Q,\Uf}.\Theta.E_{\Xf'\Xf'}\in T_{Q,\Uf}\otimes A(\Xf')\hat\otimes \CN_{\Xf'}^\dag=A(\CE_{Q,\Uf})\otimes A(\Xf')\hat\otimes \CN_{\Xf'}^\dag$$

Let now $(R,\Vf)$ be an admissible pair for nearly overconvergent forms as in \cite[\S4.1]{UNO}. We consider
$$G^E_{Q,\Uf,R,\Vf}\in A(\Vf\times \CE_{Q,\Uf}\times\Xf')\hat\otimes \CN_{R,\Vf}$$ defined by
$$G^E_{Q,\Uf,R,\Vf}(\kappa,y,\nu):=e_{R,\Vf}. G^E_{Q,\Uf}(y,\nu,\kappa\kappa_y^{-1}\nu^{-2})\in\CN_\kappa^\dag$$
for any $(\kappa,y,\nu)\in\Vf\times \CE_{Q,\Uf}\times\Xf'(\Q_p)$. Notice that since $\Uf$ and $\Vf$ are contained in $\Xf'$, so is $\kappa\kappa_y^{-1}\nu^2$ which allows to evaluate
$G^{E}_{Q,\Uf}$ at $(y,\nu,\kappa\kappa_y^{-1}\nu^{-2})$. Its gives a nearly overconvergent modular form of  weight  $\kappa$ which is the running variable in $\Vf$. We can therefore apply the finite slope projector
$e_{R,\Vf}$ from  $\CN^{\dag}_\Vf$ onto $\CN_{R,\Vf}$ specialized at $\kappa$.

\subsection{Final Remarks}\label{sec:finalremark}
We denote $G^E_{Q,\Uf,R,\Vf}(q)\in A(\Vf\times\CE_{Q,\Uf}\times\Xf')[[q]]$ the $q$-expansion of the family of nearly overconvergent forms we have defined above.
This is the family of $q$-expansion that we wanted to define in \cite[\S 4.4.1]{UNO}. The rest of the statements and results of \cite[\S 4]{UNO} are now valid under
the condition that we replace $\Xf$ by $\Xf'$ and $\CE$ by $\CE'=\CE \times_\Xf\Xf'$ in all of them. To obtain a more general result, we would need to extend the
work of \cite{AI} to relax their assumption 4.1. This seems possible by noticing that the condition $u_\kappa\in p.\Lambda_I$ can be replaced by $u_\kappa$
topologically nilpotent in $\Lambda_I$ and by using a congruence for $\nabla^{(p-1)p^n}-id$ for $n$ sufficiently large. This would allow to replace $\Xf'$ by $\Xf$
in  Lemma \ref{lemma_eis} above which is the only reason we needed to restrict ourself to $\Xf'$.

\end{document}